\documentclass[12pt]{amsart}
\usepackage{amssymb}
\usepackage{pstricks}
\usepackage{hyperref}

\newtheorem{theorem}{Theorem}[section]
\newtheorem{lemma}[theorem]{Lemma}
\newtheorem{proposition}[theorem]{Proposition}
\newtheorem{corollary}[theorem]{Corollary}
\newtheorem{definition}[theorem]{Definition}

\newtheorem{assumption}[theorem]{Assumption}
\newtheorem{property}[theorem]{Property}
\theoremstyle{definition}
\newtheorem{remark}[theorem]{Remark}
\newtheorem{example}[theorem]{Example}
\numberwithin{equation}{section}

\newcommand{\Z}{\mathbb{Z}}
\newcommand{\R}{\mathbb{R}}
\newcommand{\Q}{\mathbb{Q}}
\newcommand{\C}{\mathbb{C}}
\newcommand{\K}{\mathbb{K}}
\newcommand{\CP}{\mathbb{CP}}
\newcommand{\PP}{\mathbb{P}}

\newcommand{\Log}{\mathrm{Log}}
\renewcommand{\O}{\mathcal{O}}
\newcommand{\F}{\mathcal{F}}
\newcommand{\m}{\mathfrak{m}}
\newcommand{\x}{\mathbf{x}}
\renewcommand{\v}{\mathbf{v}}

\title[Blowups and mirror symmetry for hypersurfaces]{Lagrangian fibrations 
on blowups of toric varieties and mirror symmetry for hypersurfaces}
\author{Mohammed Abouzaid}
\address{Department of Mathematics, Columbia University, New York NY 10027, USA}
\email{abouzaid@math.columbia.edu}
\author{Denis Auroux}
\address{Department of Mathematics, UC Berkeley, Berkeley CA 94720-3840, USA}
\email{auroux@math.berkeley.edu}
\author{Ludmil Katzarkov}
\address{
Fakult\"at f\"ur Mathematik, University of Vienna, 1090~Wien, Austria}
\email{ludmil.katzarkov@univie.ac.at}
\thanks{The first author was partially supported by a Clay Research 
Fellowship and by NSF grant DMS-1308179. The second author was partially 
supported by NSF grants DMS-1264662 and DMS-1406274 and by a Simons
Fellowship. The third author was partially supported by NSF grants 
DMS-1201475 and DMS-1265230, FWF grant P24572-N25, and ERC grant GEMIS}

\begin{document}
\begin{abstract} We consider mirror symmetry for (essentially arbitrary)
hypersurfaces 
in (possibly noncompact) toric varieties from the perspective of
the Strominger-Yau-Zaslow (SYZ) conjecture. Given a
hypersurface $H$ in a toric variety $V$
we construct a Landau-Ginzburg model which is SYZ mirror to the blowup 
of $V\times\C$ along $H\times 0$, under a positivity assumption. This construction also yields
SYZ mirrors to affine conic bundles, as well as a Landau-Ginzburg model which can be naturally 
viewed as a mirror to~$H$. The main applications concern affine
hypersurfaces of general type, for which our results provide a
geometric basis for various mirror symmetry statements that appear in
the recent literature. We also obtain analogous results for
complete intersections.
\end{abstract}

\maketitle

\section{Introduction}

A number of recent results \cite{KKOY,SeGenus2,Efimov,AAEKO,GKR} 
suggest that the phenomenon of mirror symmetry is not restricted to
Calabi-Yau or Fano manifolds. Indeed, while mirror symmetry
was initially formulated as a duality between Calabi-Yau manifolds,
it was already suggested in the early works of Givental and Batyrev
that Fano manifolds also exhibit mirror symmetry. The counterpart to
the presence of a nontrivial first Chern class is that the mirror
of a compact Fano manifold is not a compact manifold, but rather
a {\em Landau-Ginzburg model}, i.e.\ a (non-compact) K\"ahler manifold
equipped with a holomorphic function called {\em superpotential}.
A physical explanation of this phenomenon and a number of examples
have been given by Hori and Vafa \cite{HV}. From a mathematical point of
view, Hori and Vafa's construction amounts to a toric duality,
and can also be applied to varieties of general type
\cite{Clarke,Katzarkov,KKOY,GKR}. 

The Strominger-Yau-Zaslow (SYZ) conjecture
\cite{SYZ} provides a geometric interpretation of mirror symmetry for Calabi-Yau
manifolds as a duality between (special) Lagrangian torus fibrations. In the
language of Kontsevich's homological mirror symmetry \cite{KoICM}, the
SYZ conjecture reflects the expectation that the mirror can be realized
as a moduli space of certain objects in the Fukaya category of the given
manifold, namely, a family of Lagrangian tori equipped with 
rank 1 local systems. Note that this homological perspective eliminates the requirement
of finding special Lagrangian fibrations, at the cost of privileging one side of mirror symmetry: 
in the Calabi-Yau case, the framework we follow produces a degenerating family $Y^0$ of complex 
manifolds ($B$-side) starting with a Lagrangian torus fibration on a symplectic manifold $X^0$  ($A$-side). 

Outside of the Calabi-Yau situation, homological
mirror symmetry is still expected to hold \cite{KoENS}, but the Lagrangian
tori bound holomorphic discs, which causes their Floer theory to be
obstructed; the mirror superpotential can be interpreted as a weighted 
count of these holomorphic discs \cite{hori,Au1,Au2,FO3toric}. We call such a mirror
a \emph{$B$-side Landau-Ginzburg model}.

In the Calabi-Yau case, mirror symmetry is expected to be involutive; i.e 
when the symplectic form on $X^0$ is in fact a K\"ahler form for some 
degenerating family of complex structures then the mirror $Y$ should be 
equipped with its own K\"ahler form which is mirror to these complex 
structures. Involutivity should hold beyond the Calabi-Yau situation, but requires making sense of a class of potential functions on symplectic manifolds, called \emph{$A$-side Landau-Ginzburg models}, which have well defined Fukaya categories. The idea for such a definition goes back to Kontsevich  \cite{KoENS}, and was studied in great depth by Seidel in \cite{SeBook} in the special case of Lefschetz fibrations. 
\begin{remark} \label{rem:keep_ourselves_honest}
The general theory of  Fukaya  categories $\F(X, W^{\vee})$ of $A$-side 
Landau-Ginzburg models is still under development in different 
contexts \cite{A-Se,A-Au,A-Ga}; we shall specifically point out where 
it is being used in this paper. In fact, we will also need to consider {\em twisted}
versions of $A$-side Landau-Ginzburg models, 
where objects of the Fukaya category carry {\em relatively spin} structures 
with respect to a background class in $H^2(X,\Z/2)$ (rather than
spin structures); see Section \ref{s:backtoH}.
\end{remark}

On manifolds of general type (or more generally, whose first Chern
class cannot be represented by an effective divisor), the SYZ approach
to mirror symmetry seems to fail at first glance due to the lack of a
suitable Lagrangian torus fibration. The idea that allows one to overcome this 
obstacle is to replace the
given manifold with another closely related space which does carry an
appropriate SYZ fibration. Thus, we make the following definition:

\begin{definition} \label{def:SYZ-mirror}
We say that a $B$-side Landau-Ginzburg model $(Y,W)$ is {\em SYZ mirror} to 
a K\"ahler manifold $X$ $($resp.\ an $A$-side Landau-Ginzburg model $(X,
W^{\vee}))$\/ if there exists an open dense subset $X^0$ of $X$, 
and a Lagrangian torus fibration $\pi:X^0\to B$,
such that the following properties hold:
\begin{enumerate}
\item $Y$ is a completion of a moduli space of
unobstructed
torus-like objects of the Fukaya
category $\F(X^0)$ $($resp.\ $\F(X^0, W^{\vee}))$
containing those objects which are supported
on the fibers of $\pi$;
\item the function $W$ restricts to the superpotential induced by the
deformation of $\F(X^0)$ to $\F(X)$ $($resp.\ $\F(X^0, W^{\vee})$ to 
$\F(X, W^{\vee}))$ for these objects.
\end{enumerate}
We say that $(Y,W)$ is a {\em generalized SYZ mirror} of $X$ if (after
shifting $W$ by a suitable additive constant) it is
an SYZ mirror of a (suitably twisted) $A$-side Landau-Ginzburg model with Morse-Bott superpotential,
whose critical locus is isomorphic to~$X$.
\end{definition}
The last part of the definition is motivated by the expectation that the Fukaya category of a Morse-Bott 
superpotential, twisted by a background class which accounts for
the non-triviality of the normal bundle to the critical locus, is equivalent 
(up to an additive constant shift in the curvature term, which accounts for
exceptional curves through the critical locus) to the Fukaya category of the critical locus; see Corollary \ref{cor:functor}
and Proposition \ref{prop:m0shift}.

Definition \ref{def:SYZ-mirror} and the construction of moduli spaces of objects 
of the Fukaya category are clarified in Section~\ref{s:syz} and
Appendix~\ref{sec:moduli-objects-fukay}.  To understand the first
condition in the case of an $A$-side Landau-Ginzburg model, it is useful to note that every object of the Fukaya category $\F(X^0)$ of compact Lagrangians 
also defines an object of $\F(X^0,W^{\vee})$ since the objects of
the latter are Lagrangians satisfying admissibility properties outside a compact set
and such properties trivially hold for compact Lagrangians.
Hence the fibers of $\pi$ automatically define
objects of $\F(X^0,W^{\vee})$; we shall enlarge this space by considering certain non-compact Lagrangians in $X^0$ which can be seen as limits of compact Lagrangians.

\begin{remark}\label{rmk:morestuff}
It is important to note that, even in the absence of superpotentials, the 
assertion that $Y^0$ is SYZ mirror to $X^0$ does not imply that the Fukaya 
category of $X^0$ is equivalent to the derived category of
$Y^{0}$; indeed, the example of the Kodaira surface mentioned in \cite{A-ICM} shows that 
there may in general be an analytic gerbe on $Y^0$ so that the Fukaya 
category of $X^0$ is in fact mirror to sheaves twisted by this gerbe. 


Beyond the Calabi-Yau situation, a complete statement of homological mirror
symmetry for SYZ mirrors would have to consider further deformations of the
derived category of sheaves by (holomorphic) polyvector fields on $Y$.  The
superpotential $W$ should be thought of as the leading order term of
this deformation corresponding to discs of Maslov index $2$.
\end{remark}

In this paper we use this perspective to study mirror symmetry for
hypersurfaces (and complete intersections) in toric varieties. 
If $H$ is a smooth hypersurface in
a toric variety $V$, then one simple way to construct a closely related
K\"ahler manifold with effective first Chern class 
is to blow up the product $V\times\C$ along the codimension~2
submanifold $H\times 0$. By a result of Bondal and Orlov \cite{BO},
the derived category of coherent sheaves of the resulting manifold $X$
admits a semi-orthogonal decomposition into subcategories equivalent to
$D^b\mathrm{Coh}(H)$ and $D^b\mathrm{Coh}(V\times\C)$; and ideas similar to those of
\cite{Smith} can be used to study the Fukaya category of $X$, as we explain in
Section~\ref{s:backtoH} (cf.~Corollary \ref{cor:functor}).
Thus, finding a mirror to $X$ is, for many purposes, as good as finding a
mirror to $H$. 
Accordingly, our main results concern SYZ mirror symmetry for $X$ and,
by a slight modification of the
construction, for $H$. Along the way
we also obtain descriptions of SYZ mirrors to various related spaces.
These results provide a geometric foundation for mirror
constructions that have appeared in the recent literature
\cite{Clarke,Katzarkov,KKOY,SeGenus2,Sheridan,AAEKO,GKR}. 

We focus primarily on the case where $V$ is affine, and other cases which
can be handled with the same techniques.
The general case requires more
subtle arguments in enumerative geometry, which 
should be the subject of further investigation.


\subsection{Statement of the results}
Our main result can be formulated as follows (see \S \ref{s:notations} for
the details of the notations).

Let $H=f^{-1}(0)$ be a smooth nearly tropical hypersurface (cf.\ \S
\ref{ss:tropdeg}) in a (possibly noncompact) toric 
variety $V$ of dimension $n$, and let $X$ be the blow-up of $V\times \C$ 
along $H\times 0$, equipped with an 
$S^1$-invariant K\"ahler form $\omega_\epsilon$ for which the fibers of 
the exceptional divisor have sufficiently small area $\epsilon>0$
(cf.\ \S \ref{ss:blowup}).

Let $Y$ be the toric variety defined by the polytope $\{(\xi,\eta)\in
\R^n\times\R\,|\,\eta\ge \varphi(\xi)\}$, where $\varphi$ is the tropicalization
of $f$. Let $w_0=-T^{\epsilon}+T^{\epsilon}v_0\in \O(Y)$, where $T$ is the
Novikov parameter and $v_0$ is
the toric monomial with weight $(0,\dots,0,1)$, and set $Y^0=Y\setminus
w_0^{-1}(0)$. Finally, let $W_0=w_0+w_1+\dots+w_r\in \O(Y)$ be the 
{\em leading-order superpotential} of Definition \ref{def:LGmirror},
namely the sum of $w_0$ and one toric 
monomial $w_i$ ($1\le i\le r$) for each irreducible toric divisor of $V$ 
(see Definition \ref{def:LGmirror}). We assume:

\begin{assumption}\label{ass:affinecase}
$c_1(V)\cdot C>\max(0,H\cdot C)$ for every rational curve $C\simeq \PP^1$ in $V$.
\end{assumption}

\noindent This includes the case where $V$ is an affine toric variety as an important special case.
Under this assumption, our main result is the following:

\begin{theorem}\label{thm:main}
Under Assumption \ref{ass:affinecase}, 
the $B$-side Landau-Ginzburg model $(Y^0,W_0)$ \/is SYZ mirror to $X$.
\end{theorem}

In the general case, the mirror of $X$ differs from $(Y^0,W_0)$
by a correction term which is of higher order with respect to the 
Novikov parameter (see Remark \ref{rmk:borderline}).

Equipping $X$ with an appropriate
superpotential, given by the affine coordinate of the $\C$ factor,
yields an $A$-side Landau-Ginzburg model whose singularities are of Morse-Bott type.
Up to twisting by a class in $H^2(X,\Z/2)$, this $A$-side Landau-Ginzburg model
can be viewed as a stabilization of the sigma model with target~$H$.

\begin{theorem}\label{cor:main}
Assume $V$ is affine, and
let $W_0^H=-v_0+w_1+\dots+w_r\in \O(Y)$ (see~Definition \ref{def:LGmirror}).
Then the $B$-side Landau-Ginzburg model $(Y,W_0^H)$ is a generalized SYZ mirror of\/
$H$.
\end{theorem}

Unlike the other results stated in this introduction, this theorem 
strictly speaking relies on the assumption that Fukaya categories of 
Landau-Ginzburg models satisfy certain properties for which we do not 
provide complete proofs. In Section \ref{s:backtoH}, we give sketches 
of the proofs of these results, and indicate the steps which are missing 
from our argument.

A result similar to Theorem \ref{cor:main} can also be obtained from the perspective of mirror duality 
between toric Landau-Ginzburg models \cite{HV,Clarke,KKOY,GKR}. However, the
toric approach is much less illuminating, because geometrically it works 
at the level of the open toric strata in the relevant toric varieties 
(the total space of $\mathcal{O}(-H)\to V$ on one hand, and $Y$ on the other
hand), whereas the interesting geometric features of these spaces lie 
entirely within the toric divisors. 

Theorem \ref{thm:main} relies on a mirror symmetry statement
for open Calabi-Yau manifolds which is of independent interest. Consider
the conic bundle $$X^0=\{(\mathbf{x},y,z)\in V^0\times \C^2\,|\,
yz=f(\mathbf{x})\}$$ over the open stratum $V^0\simeq (\C^*)^n$ of $V$,
where $f$ is again the defining equation of the hypersurface $H$.
The conic bundle $X^0$ sits as an open dense subset inside $X$, 
see Remark \ref{rmk:conicbundle}. Then we have:

\begin{theorem}\label{thm:conicbundle}
The open Calabi-Yau manifold $Y^0$ is SYZ mirror to $X^0$.
\end{theorem}

In the above statements, and in most of this paper, we view $X$ or $X^0$ 
as a symplectic manifold, and construct the SYZ mirror $Y^0$
(with a superpotential) as an algebraic moduli space of objects in the
Fukaya category of $X$ or $X^0$. This is the same direction considered
e.g.\ in \cite{SeGenus2,Efimov,AAEKO}. However, one can also work in the
opposite direction, starting from the symplectic geometry of $Y^0$ and 
showing that it admits $X^0$ (now viewed as a complex manifold) as an 
SYZ mirror. For completeness we describe this converse construction
in Section \ref{s:converse} (see Theorem \ref{thm:converse});
similar results have also been obtained independently by Chan, Lau and Leung
\cite{ChanLauLeung}.
\medskip

The methods we use apply in more general settings as well. In particular,
the assumption that $V$ be a toric variety is not strictly necessary -- it
is enough that SYZ mirror symmetry for $V$ be sufficiently well understood.
As an illustration, in Section~\ref{ss:completeint} 
we derive analogues of Theorems \ref{thm:main}--\ref{thm:conicbundle} for 
{\em complete intersections}.

\subsection{A reader's guide}
The rest of this paper is organized as follows. 

First we briefly review (in Section \ref{s:syz}) 
the SYZ approach to mirror symmetry, 
following~\cite{Au1,Au2}. Then in Section \ref{s:notations} we 
introduce notation and describe the protagonists of our main results, namely
the spaces $X$ and $Y$ and the superpotential~$W_0$. 

In Section \ref{s:fibrations} we construct
a Lagrangian torus fibration on $X^0$, similar to those previously considered 
by Gross \cite{GrossTop,GrossSlag} and by Casta\~no-Bernard and Matessi
\cite{CBM1,CBM2}. In Section \ref{s:opencase} we study the Lagrangian Floer
theory of the torus fibers, which we use to prove Theorem
\ref{thm:conicbundle}.
In Section \ref{s:generalcase} we consider the 
partial compactification of $X^0$ to $X$, and prove Theorem \ref{thm:main}.
Theorem \ref{cor:main} is then proved in Section
\ref{s:backtoH}.

In Section \ref{s:converse} we briefly consider the converse construction,
namely we start from a Lagrangian torus fibration on $Y^0$ and recover $X^0$ as its
SYZ mirror. 

Finally, some examples illustrating the main results are given in
Section \ref{s:examples}, 
while Sections~\ref{s:generalizations} and \ref{ss:completeint} discusses various
generalizations, including to hypersurfaces in abelian varieties
(Theorem \ref{thm:mainabelian}) and complete intersections in toric
varieties (Theorem \ref{thm:mainci}).

\subsection*{Acknowledgements} We would like to thank Paul Seidel and
Mark Gross for a number of illuminating discussions; Patrick 
Clarke and Helge Ruddat for explanations of their work; 
Anton Kapustin, Maxim Kontsevich, Dima Orlov and Tony Pantev for useful
conversations; and the anonymous referees for their numerous comments on
earlier versions of this manuscript.

The first author was supported by a Clay Research Fellowship
and by NSF grant DMS-1308179. The second author was partially 
supported by NSF grants DMS-1264662 and DMS-1406274 and by a
Simons Fellowship. The third author 
was partially supported by NSF grants 
DMS-1201475 and DMS-1265230, FWF grant P24572-N25, and ERC grant GEMIS.

\section{Review of SYZ mirror symmetry}\label{s:syz}

In this section, we briefly review SYZ mirror symmetry for K\"ahler 
manifolds with effective anticanonical class; the reader is referred to 
\cite{Au1,Au2} for basic ideas about SYZ, and to Appendix \ref{sec:moduli-objects-fukay}
for technical details.

\subsection{Lagrangian torus fibrations and SYZ mirrors}\label{ss:syz1}

In first approximation, the Strominger-Yau-Zaslow conjecture \cite{SYZ} states 
that mirror pairs of Calabi-Yau manifolds carry mutually dual Lagrangian 
torus fibrations (up to ``instanton corrections'').
A reformulation of this statement in the language of homological mirror 
symmetry \cite{KoICM} is that a mirror of a Calabi-Yau manifold can be
constructed as a moduli space of suitable objects in its
Fukaya category (namely, the fibers of an SYZ fibration,
equipped with rank 1 local systems); and vice versa. In Appendix \ref{sec:moduli-objects-fukay},
we explain how ideas of Fukaya \cite{Fcyclic} yield a precise construction of such a mirror space from local rigid analytic charts glued via the equivalence relation which identifies objects that are quasi-isomorphic in the Fukaya category.

We consider an open Calabi-Yau manifold of the
form $X^0=X\setminus D$, where $(X,\omega,J)$ is a K\"ahler manifold of complex
dimension $n$ and $D\subset X$
is an anticanonical divisor (reduced, with normal crossing singularities).
$X^0$ can be equipped with a holomorphic $n$-form $\Omega$ (with simple
poles along $D$), namely the inverse of the defining section of~$D$.
The restriction of $\Omega$ to an oriented Lagrangian submanifold 
$L\subset X^0$ is a nowhere vanishing complex-valued $n$-form on $L$;
the complex argument of this $n$-form determines the {\em phase} function
$\arg(\Omega_{|L}):L\to S^1$. Recall that $L$ is said to be {\em special
Lagrangian} if $\arg(\Omega_{|L})$ is constant; a weaker condition is to
require the vanishing of the {\em Maslov class} of $L$ in $X^0$, i.e.\ we 
require the existence of a lift of $\arg(\Omega_{|L})$ to a real-valued function.
(The choice of such a real lift then makes $L$ a {\em graded Lagrangian},
and yields $\Z$-gradings on Floer complexes.)

The main input of the construction of the SYZ mirror of the open Calabi-Yau
manifold $X^0$ is a Lagrangian torus
fibration $\pi:X^0\to B$ (with appropriate singularities) whose fibers have
trivial Maslov class. (Physical considerations suggest that one 
should expect the fibers of $\pi$ to be
special Lagrangian, but such fibrations are hard to produce.)

The base $B$ of the Lagrangian torus fibration $\pi$ carries a natural 
real affine structure (with singularities along the locus $B^{sing}$
of singular fibers), i.e.\ $B\setminus B^{sing}$ can be
covered by a set of coordinate charts with transition
functions in $GL(n,\Z)\ltimes \R^n$. A smooth fiber $L_0=\pi^{-1}(b_0)$
and a collection of loops $\gamma_1,\dots,\gamma_n$ forming a basis of 
$H_1(L_0,\Z)$ determine an
affine chart centered at $b_0$ in the following manner: given $b\in
B\setminus B^{sing}$ close enough to $b_0$, we can isotope $L_0$ to
$L=\pi^{-1}(b)$ among fibers of $\pi$. Under such an isotopy, each loop $\gamma_i$ traces a
cylinder $\Gamma_i$ with boundary in $L_0\cup L$; the affine coordinates
associated to $b$ are then the symplectic areas 
$(\int_{\Gamma_1}\omega,\dots,\int_{\Gamma_n}\omega)$.

In the examples we will consider, ``most'' fibers of $\pi$ do not bound  nonconstant holomorphic discs
in $X^0$; we call such Lagrangians \emph{tautologically unobstructed.}
Recall that a  (graded, spin) Lagrangian submanifold 
$L$ of $X^0$ together with a  unitary rank one local system $\nabla$ 
determines an object $(L,\nabla)$ of the Fukaya category $\F(X^0)$
\cite{FO3book} whenever certain counts of holomorphic discs cancel; this condition evidently holds if there are no non-constant discs. Thus, given an open subset 
$U\subset B\setminus B^{sing}$ such that all the fibers in $\pi^{-1}(U)$
are tautologically unobstructed, the moduli space of objects $(L,\nabla)$ where 
$L\subset \pi^{-1}(U)$ is a fiber of $\pi$ and $\nabla$ is a unitary 
rank 1 local system on $L$ yields an open subset $U^\vee\subset Y^0$ of the SYZ
mirror of $X^0$. 

A word is in order about the choice of coefficient field. A careful
definition of Floer homology involves working over the
{\em Novikov field}\/ (here over complex numbers),
\begin{equation}\label{eq:novikov}
\Lambda=\left\{\sum_{i=0}^\infty c_i T^{\lambda_i}\,\Big|\,c_i\in\C,\ 
\lambda_i\in\R,\ \lambda_i\to +\infty\right\}.
\end{equation}
Recall that the {\em valuation} of a non-zero element of $\Lambda$ is the
smallest exponent $\lambda_i$ that appears with a non-zero coefficient;
the above-mentioned local systems are required to have holonomy in
the multiplicative subgroup 
$$U_{\Lambda} =\textstyle \left\{c_0+\sum c_iT^{\lambda_i}\in\Lambda\,\big|\,
c_0\neq 0\ \text{and}\ \lambda_i>0\right\}$$ 
of {\em unitary elements} (or units) of the 
Novikov field, i.e.\ elements whose valuation is zero. The local system
$\nabla\in \hom(\pi_1(L),U_{\Lambda})$ enters into the definition of 
Floer-theoretic operations by contributing
holonomy terms to the weights of holomorphic discs: a rigid holomorphic 
disc $u$ with boundary on Lagrangians $(L_i,\nabla_i)$ is counted with a 
weight 
\begin{equation}\label{eq:weight}
T^{\omega(u)} \mathrm{hol}(\partial u),
\end{equation} 
where $\omega(u)$ is the symplectic area of the disc $u$, and 
$\mathrm{hol}(\partial u)\in U_{\Lambda}$ is the total holonomy of the 
local systems $\nabla_i$ along its boundary. (Thus, local systems are
conceptually an exponentiated variant of the ``bounding cochains'' used 
by Fukaya et al \cite{FO3book,FO3toric}). Gromov compactness ensures
that all structure constants of Floer-theoretic operations are well-defined
elements of $\Lambda$. 

Thus, in general the SYZ mirror of $X^0$ is naturally
an analytic space defined over $\Lambda$.
However, it is often possible to obtain a complex mirror
by treating the Novikov parameter $T$ as a numerical parameter
$T=e^{-2\pi t}$ with $t>0$ sufficiently large; 
of course it is necessary to assume the convergence of all the power series
encountered. The local systems are then taken to be unitary in the usual sense, i.e.\ $\nabla\in
\hom(\pi_1(L),S^1)$, and the weight of a rigid holomorphic disc, still
given by \eqref{eq:weight}, becomes a complex number.
The complex manifolds obtained by varying the parameter $t$
are then understood to be mirrors to the family of K\"ahler manifolds
$(X^0,t\omega)$.

To provide a unified treatment, we denote by $\K$ the coefficient
field ($\Lambda$ or $\C$), by $U_{\K}$ the subgroup of unitary
elements (either $U_{\Lambda}$ or $S^1$), and by $\mathrm{val}:\K\to \R$ the
valuation (in the case of complex numbers, $\mathrm{val}(z)=-\frac{1}{2\pi t}
\log |z|$).

Consider as above a contractible open subset $U\subset B\setminus B^{sing}$
above which all fibers of $\pi$ are tautologically unobstructed, a reference fiber
$L_0=\pi^{-1}(b_0)\subset \pi^{-1}(U)$, and
a basis $\gamma_1,\dots,\gamma_n$ of $H_1(L_0,\Z)$.
A fiber $L=\pi^{-1}(b)\subset \pi^{-1}(U)$ and a local system
$\nabla\in \hom(\pi_1(L),U_{\K})$ determine a point of the mirror,
$(L,\nabla)\in U^\vee\subset Y^0$.
Identifying implicitly $H_1(L,\Z)$ with $H_1(L_0,\Z)$, the local system
$\nabla$ is determined by its holonomies along the loops $\gamma_1,\dots,
\gamma_n$, while the fiber $L$ is determined by the symplectic areas of
the cylinders $\Gamma_1,\dots,\Gamma_n$. This yields natural coordinates
on $U^\vee\subset Y^0$, identifying it with an open subset of $(\K^*)^n$ via
\begin{equation}
\label{eq:localchart}
(L,\nabla)\mapsto (z_1,\dots,z_n)=\left(T^{\int_{\Gamma_1}\!\omega}\,\nabla(\gamma_1),\dots,
T^{\int_{\Gamma_n}\!\omega}\,\nabla(\gamma_n)\right).
\end{equation}
One feature of Floer theory that is conveniently captured
by this formula is the fact that, in the absence of instanton corrections,
the non-Hamiltonian isotopy from $L_0$ to $L$ is
formally equivalent to equipping $L_0$ with a {\em non-unitary} local
system for which $\mathrm{val}(\nabla(\gamma_i))=\int_{\Gamma_i}\omega$.

The various regions of $B$ over which the fibers are
tautologically unobstructed are separated by {\em walls} (real hypersurfaces in $B$, or
thickenings of real hypersurfaces) of {\em potentially obstructed} fibers (i.e. which bound non-constant holomorphic discs), across which the 
local charts of the mirror (as given by \eqref{eq:localchart}) need to be 
glued together in an appropriate manner to account for ``instanton
corrections''. 

The discussion preceding Equation \eqref{eq:analytic_space_glue_all_tori} makes precise
the idea that we can embed the moduli space of Lagrangians equipped with
unitary local systems in an analytic space obtained by gluing coordinate charts coming 
from non-unitary systems. This will be the first step in the construction of
the mirror manifold as a completion of the moduli space of Lagrangians.


Consider a potentially obstructed fiber $L=\pi^{-1}(b)$ of $\pi$, where $b\in
B\setminus B^{sing}$ lies in a wall that separates two tautologically unobstructed
chambers. By deforming this fiber to a nearby chamber, we obtain a bounding cochain (with respect to the Floer differential) for the moduli space of holomorphic discs with boundary on $L$.  While local systems on $L$ define objects of $\F(X^0)$, the quasi-isomorphism type of such objects depends
on the choice of bounding cochain, which in our case amounts to a choice of this isotopy.  In the special situation we are considering, we use this argument to prove in Lemma \ref{lem:conv-wall-cross}  that any unitary local system on $L$ can be represented by a non-unitary local system on a fiber lying in a tautologically unobstructed chamber. This implies that the space obtained by gluing the mirrors of the chambers contains the analytic space corresponding to all unitary local systems on smooth fibers of $\pi$.


The gluing maps for the mirrors of nearby chambers are given by wall-crossing
formulae, with instanton corrections accounting for the disc bubbling 
phenomena that occur as a Lagrangian submanifold is isotoped across a wall
of potentially obstructed Lagrangians (see \cite{Au1} for an informal discussion, and
Appendix \ref{sec:general-theory} for the relation with the invariance proof of Floer
cohomology in this setting \cite{FO3book,Fcyclic}). Specifically, consider a Lagrangian isotopy
$\{L_t\}_{t\in [0,1]}$ whose end points are tautologically unobstructed and lie in adjacent chambers. 
Assume that all nonconstant holomorphic discs bounded by the Lagrangians $L_t$
in $X^0$ represent a single relative homotopy class $\beta\in \pi_2(X^0,L_t)$
(we implicitly identify these groups with each other by means of the
isotopy), or its multiples (for non-simple discs). The weight
associated to the class $\beta$ defines a regular function 
$$z_\beta=T^{\omega(\beta)}
\nabla(\partial\beta)\in \O(U_i^\vee),$$ in fact a monomial in the
coordinates $(z_1,\dots,z_n)$ of \eqref{eq:localchart}. In this situation,
assuming its transversality, the
moduli space $\mathcal{M}_{1}(\{L_t\},\beta)$ of all holomorphic discs 
in the class $\beta$ bounded by $L_t$ as $t$ varies from $0$ to $1$, with
one boundary marked point, is a closed \hbox{$(n-1)$}-dimensional manifold,
oriented if we fix a spin structure on $L_t$. Thus,
evaluation at the boundary marked point (combined with identification of the
submanifolds $L_t$ via the isotopy) yields a cycle 
$C_\beta=\mathrm{ev}_*[\mathcal{M}_1(\{L_t\},\beta)]\in H_{n-1}(L_t)$.
The instanton corrections to the gluing of the local coordinate charts
\eqref{eq:localchart} are then of the form
\begin{equation}\label{eq:instcorr}
z_i\mapsto (h(z_\beta))^{C_\beta \cdot \gamma_i} z_i,
\end{equation}
where $h(z_\beta)=1+ z_{\beta} + \dots\in \mathbb{Q}[[z_\beta]]$ is a power 
series recording the (virtual) contributions of multiple covers of the discs
in the class $\beta$. In practice, we shall only use the weaker property that 
these transformations are of the form
\begin{equation}
z_i\mapsto h_i(z_\beta) z_i,
\end{equation}
where $h_i(z_\beta) \in 1+  z_\beta  \mathbb{Q}[[z_\beta]]$.

In the examples we consider in this paper, there are only finitely many
walls in $B$, and the above considerations are sufficient to construct the
SYZ mirror of $X^0$ out of instanton-corrected gluings of local charts. In general, 
intersections between walls lead, via a ``scattering'' 
phenomenon, to an infinite number of higher-order
instanton corrections; it is conjectured that these Floer-theoretic corrections can be determined using the machinery 
developed by Kontsevich-Soibelman \cite{KS1,KS2} and Gross-Siebert \cite{GS1,GS2}.

\begin{remark}\label{rmk:switchsides}
We have discussed how to construct the analytic space $Y^0$
(``B-model'') from the symplectic geometry of $X^0$ (``A-model'').
When $Y^0$ makes sense as a complex manifold (i.e., assuming convergence),
one also expects it to carry a natural K\"ahler structure for which
the A-model of $Y^0$ is equivalent to the B-model of $X^0$. We will however
not emphasize this feature of mirror symmetry.
\end{remark}

\subsection{The superpotential}\label{ss:syz2}
In the previous section we explained the construction of the SYZ mirror $Y^0$ of an open 
Calabi-Yau manifold $X^0=X\setminus D$, where $D$ is an anticanonical
divisor in a K\"ahler manifold $(X,\omega,J)$, equipped with a Lagrangian
torus fibration $\pi:X^0\to B$. We now turn to mirror symmetry for $X$
itself.

The Fukaya category of $X$ is a deformation of that of $X^0$: the Floer cohomology
of  Lagrangian submanifolds of $X^0$, when viewed as objects of $\F(X)$, is 
deformed by the presence of additional holomorphic discs 
that intersect the divisor $D$. Let $L$ be a Lagrangian fiber of
the SYZ fibration $\pi:X^0\to B$: since the Maslov class of $L$ in $X^0$
vanishes, the Maslov index of a holomorphic disc bounded by $L$ in $X$
is equal to twice its algebraic intersection number with $D$. Following
Fukaya, Oh, Ohta, and Ono \cite{FO3book} we associate to $L$ and
a rank 1 local system $\nabla$ over it the {\em obstruction}
\begin{equation}
\label{eq:m0}
\m_0(L,\nabla)=\sum_{\beta\in \pi_2(X,L)\setminus \{0\}}
z_\beta(L,\nabla)\, \mathrm{ev}_*[\mathcal{M}_1(L,\beta)]\in C^*(L;\K),
\end{equation}
where $z_\beta(L,\nabla)=T^{\omega(\beta)}\nabla(\partial\beta)$ is the weight
associated to the class $\beta$, and $\mathcal{M}_1(L,\beta)$ is the
moduli space of holomorphic discs with one boundary marked point in $(X,L)$ 
representing the class $\beta$. In the absence of bubbling, one can achieve regularity, 
and $[\mathcal{M}_1(L,\beta)]$ can be defined as the fundamental class of the manifold $\mathcal{M}_1(L,\beta)$.  To consider a more general situation, we appeal to the work of Fukaya, Oh, Ohta, and Ono who define  such a potential  for Lagrangian fibers in toric manifolds in \cite{FO3toric}. While the examples we consider are not toric, their construction applies more generally whenever the moduli spaces of stable holomorphic discs with non-positive Maslov
index contribute trivially to the Floer differential. The situation is therefore
simplest when the divisor $D$ is nef, or more generally when the following condition holds:

\begin{assumption}\label{ass:nef}
Every rational curve $C\simeq \PP^1$ in $X$ has non-negative intersection number 
$D\cdot C\ge 0$.
\end{assumption}

Consider first the case of a Lagrangian submanifold $L$ which is tautologically unobstructed in $X^0$. 
By positivity of intersections,
the minimal Maslov index of a non-constant holomorphic disc with
boundary on $L$ is~$2$ (when \hbox{$\beta\cdot D=1$}). 
Gromov compactness implies that the
chain $\mathrm{ev}_*[\mathcal{M}_1(L,\beta)]$ is actually a cycle, of dimension
$n-2+\mu(\beta)=n$, i.e.\ a scalar multiple $n(L,\beta)[L]$ of the
fundamental class of $L$; whereas the evaluation chains for $\mu(\beta)>2$ have
dimension greater than $n$ and we discard them. Thus $(L,\nabla)$ is
{\em weakly unobstructed}, i.e.\ $$\m_0(L,\nabla)=W(L,\nabla)\, e_{L}$$
is a multiple of the unit in $H^{0}(L, \K) $, which is Poincar\'e dual to 
the fundamental class of $L$. More generally, Assumption \ref{ass:nef}
excludes discs of negative Maslov index, while the vanishing of the contribution
of discs of Maslov index $0$ is explained in Appendix~\ref{sec:conv-wall-cross}.

Given an open subset $U\subset B\setminus B^{sing}$ over which the
fibers of $\pi$ are tautologically unobstructed in $X^0$, the coordinate chart
$U^\vee\subset Y^0$ considered in the previous section now parametrizes weakly 
unobstructed objects $(L=\pi^{-1}(b),\nabla)$ of $\F(X)$, and the
{\em superpotential} 
\begin{equation}\label{eq:superpot}
W(L,\nabla)=\sum_{\substack{\beta\in \pi_2(X,L)\\ \beta\cdot D=1}}
n(L,\beta) \,z_\beta(L,\nabla)
\end{equation}
is a regular function on $U^\vee$. The superpotential represents a
curvature term in Floer theory: the differential on the Floer complex
of a pair of weakly unobstructed objects $(L,\nabla)$ and $(L',\nabla')$ 
squares to $(W(L',\nabla')-W(L,\nabla))\,\mathrm{id}$. In particular, the
family Floer cohomology \cite{Ffamily} of an unobstructed Lagrangian 
submanifold of $X$  with the fibers 
of the SYZ fibration over $U$ is expected to yield no longer an object of the 
derived category of coherent sheaves over $U^\vee$ 
but rather a {\em matrix factorization} of the superpotential $W$.

In order to construct the mirror of $X$ globally, 
we again have to
account for instanton corrections across the walls of potentially obstructed fibers
of $\pi$. As before, these corrections are needed in order to account for the bubbling
of holomorphic discs of Maslov index~0 as one crosses a wall, and encode
weighted counts of such discs. Under Assumption~\ref{ass:nef}, 
positivity of intersection implies that all the holomorphic discs of 
Maslov index 0 are contained in $X^0$; therefore the instanton
corrections are exactly the same for $X$ as for $X^0$, i.e.\ the
moduli space of objects of $\F(X)$ that we construct out of the fibers of
$\pi$ is again $Y^0$ (the SYZ mirror of $X^0$).

A key feature of the instanton corrections is that the superpotential
defined by \eqref{eq:superpot} naturally glues to a regular function on $Y^0$;
this is because, by construction, the gluing via wall-crossing transformations identifies quasi-isomorphic
objects of $\F(X)$, for which the obstructions $\m_0$ have to match, as explained
in Corollary \ref{cor:potential_matched_by_gluing}.
Thus, the mirror of $X$ is the $B$-side Landau-Ginzburg model $(Y^0,W)$, where
$Y^0$ is the SYZ mirror of $X^0$ and $W\in \O(Y^0)$ is given by \eqref{eq:superpot}.
(However, see Remark \ref{rmk:morestuff}).

\begin{remark}\label{rmk:instcorrtrick}
The regularity of the superpotential $W$ is a useful feature for the 
construction  of the SYZ mirror of $X^0$. Namely,
rather than directly computing the instanton corrections by studying the
enumerative geometry of 
holomorphic discs in $X^0$, it is often easier to determine them 
indirectly, by considering either $X$ or some other partial compactification 
of $X^0$ (satisfying Assumption \ref{ass:nef}), computing the mirror
superpotential in each chamber of $B\setminus B^{sing}$, and
matching the expressions on either side of a wall via a coordinate change of
the form \eqref{eq:instcorr}. 
\end{remark}

When Assumption \ref{ass:nef} fails, the instanton corrections to the
SYZ mirror of $X$ might differ from those for $X^0$ (hence the difference
between the mirrors might be more subtle than simply equipping $Y^0$ with a
superpotential). However, this only happens if the (virtual) counts of 
Maslov index 0 discs bounded by potentially obstructed fibers of $\pi$ in $X$ differ
from the corresponding counts in $X^0$. Fukaya-Oh-Ohta-Ono have shown that
this issue never arises for toric varieties \cite[Corollary 11.5]{FO3toric}.
In that case, the deformation of the
Fukaya category which occurs upon (partially) compactifying $X^0$ to $X$ 
(due to the presence of additional holomorphic discs) is accurately
reflected by the deformation of the mirror B-model given by
the superpotential $W$ (i.e., considering matrix factorizations
rather than the usual derived category).

Unfortunately, the argument of \cite{FO3toric} does not adapt immediately to
our setting; thus for the time being we only consider settings in which Assumption
\ref{ass:nef} holds.  This will be the subject of further investigation.


\medskip

The situation is in fact symmetric: just as partially compactifying $X^0$
to $X$ is mirror to equipping $Y^0$ with a superpotential, equipping $X^0$
or $X$ with a superpotential is mirror to partially compactifying $Y^0$.
One way to justify this claim would be to switch to the other direction of
mirror symmetry, reconstructing $X^0$ as an SYZ mirror of $Y^0$ equipped
with a suitable K\"ahler structure (cf.\ Remark \ref{rmk:switchsides}). 
However, in simple cases this statement can also be understood directly.
The following example will be nearly sufficient for our purposes
(in Section \ref{s:backtoH} we will revisit and generalize it):

\begin{example}\label{ex:FScompact}
Let $X^0=\C^*$, whose mirror $Y^0\simeq \K^*$
parametrizes objects $(L,\nabla)$ of $\F(X^0)$, where $L$ is a simple 
closed curve enclosing the origin (up to Hamiltonian isotopy) and 
$\nabla$ is a unitary rank 1 local system on $L$. The natural coordinate
on $Y^0$, as given by \eqref{eq:localchart}, tends to zero as
the area enclosed by $L$ tends to infinity. Equipping $X^0$ with the
superpotential $W(x)=x$, the Fukaya category $\F(X^0,W)$ 
also contains ``admissible'' non-compact Lagrangian submanifolds, i.e.\
properly embedded Lagrangians whose image under $W$ is only allowed to tend 
to infinity in the direction of the positive real axis. Denote by $L_\infty$ a 
properly embedded arc which connects $+\infty$ to itself by passing around
the origin (and encloses an infinite amount of area). An easy calculation
in $\F(X^0,W)$ shows that $\mathrm{HF}^*(L_\infty,L_\infty)\simeq H^*(S^1;\K)$; so $L_\infty$ behaves
Floer cohomologically like a torus. In particular, $L_\infty$ admits a one-parameter family of deformations in $\F(X^0,W)$;
these are represented by equipping $L_{\infty}$ with a bounding cochain in $\mathrm{HF}^1(L_\infty,L_\infty) = \K$ of sufficiently large valuation (with our conventions, the valuation of $0$ is $+\infty$). Given a point $cT^{\lambda} \in \K $,  the corresponding Floer differential counts, in addition to the usual strips, triangles with one boundary puncture converging to a time $1$ chord of 
an appropriate Hamiltonian (equal to $\mathrm{Re}(x)$ near $+\infty$)  with ends on $L$ (this is the implementation of the Fukaya category $\F(X^0,W)$ appearing in \cite{SeLef}). 

Except for the case $c=0$, these additional objects of the Fukaya category turn out to be isomorphic to simple closed curves
(enclosing the origin) with rank 1 local systems.  
More precisely, letting
$L_{\lambda}$ be the fiber enclosing an additional amount of area $\lambda \in \R$ compared to
a suitable reference Lagrangian $L_0$,
and $\nabla_{c}$ the local system with holonomy $c$, an easy computation shows that the pairs $(L_{\infty},cT^{\lambda})$ and $(L_{\lambda},\nabla_{c}   )$ represent quasi-isomorphic objects of $\F(\C^*,W)$. Thus, in $\F(\C^*,W)$ the previously considered moduli space of objects
contains an additional point $L_\infty$; this naturally extends the mirror
from $Y^0\simeq \K^*$ to $Y\simeq \K$, and the coordinate coming from identifying bounding cochains on $L_\infty$ with local systems on closed curves
defines an analytic structure near this point.

Alternatively, one can geometrically recover the Lagrangians $L_\lambda$ 
as self-surgeries of the immersed Lagrangian obtained by deforming 
$L_\infty$ to a curve with one self-intersection,
enclosing the same amount of area as $L_\lambda$.
This self-intersection corresponds to a generator in $HF^1(L_\infty,L_\infty)$, 
giving rise to a bounding cochain.
The Floer-theoretic isomorphisms between bounding cochains on 
admissible Lagrangians and embedded Lagrangians then become an instance of 
the surgery formula of \cite{FO3surgery}.

\end{example}

\section{Notations and constructions}\label{s:notations}

\subsection{Hypersurfaces near the tropical limit}\label{ss:tropdeg}

Let $V$ be a (possibly non-compact) toric variety of complex dimension $n$,
defined by a fan $\Sigma_V\subseteq \R^n$. We denote by $\sigma_1,\dots,\sigma_r$ the
primitive integer generators of the rays of $\Sigma_V$.
We consider a family of smooth algebraic hypersurfaces $H_\tau\subset V$ 
(where $\tau\to 0$), transverse to the toric divisors in $V$, and degenerating to 
the ``tropical'' limit. 
Namely, in affine coordinates $\mathbf{x}=(x_1,\dots,x_n)$ 
over the open stratum $V^0\simeq (\C^*)^n\subset V$, $H_\tau$ is defined by an equation of 
the form 
\begin{equation}\label{eq:Ht}
f_\tau=\sum_{\alpha\in A} c_\alpha \tau^{\rho(\alpha)}\mathbf{x}^\alpha=0,
\end{equation} 
where
$A$ is a finite subset of the lattice $\Z^n$ of characters of the
torus $V^0$,  $c_\alpha\in \C^*$ are arbitrary constants, and $\rho:A\to\R$ satisfies a certain convexity 
property. 

More precisely, $f_\tau$ is a section of a certain line bundle
$\mathcal{L}$ over $V$, determined by a convex
piecewise linear function $\lambda:\Sigma_V\to\R$ with integer linear slopes.
(Note that $\mathcal{L}$ need not be ample; however the convexity assumption forces it
to be nef.) The polytope $P$ associated to $\mathcal{L}$ is the set 
of all $v\in\R^n$ such that $\langle v,\cdot\rangle+\lambda$ takes
everywhere non-negative values; more concretely,
$P=\{v\in\R^n\,|\,\langle \sigma_i,v\rangle+\lambda(\sigma_i)\ge 0\ 
\forall 1\le i\le r\}$. It is a classical fact that the integer points of 
$P$ give a basis of the space of sections of $\mathcal{L}$.
The condition that $H_\tau$ be transverse to each
toric stratum of $V$ is then equivalent to the requirement that 
$A\subseteq P\cap \Z^n$ intersects nontrivially the closure of each face 
of $P$ (i.e., in the compact case, $A$ should contain every vertex of $P$).

Consider a polyhedral decomposition $\mathcal{P}$ of the convex hull
$\mathrm{Conv}(A)\subseteq \R^n$, whose set of vertices is exactly
$\mathcal{P}^{(0)}=A$. We will mostly consider the case where 
the decomposition $\mathcal{P}$ is {\em regular}, i.e.\ every 
cell of $\mathcal{P}$ is congruent under the
action of $GL(n,\Z)$ to a standard simplex.
We say that $\rho:A\to\R$ is {\em adapted}\/ to the polyhedral
decomposition $\mathcal{P}$ if it is the restriction to $A$ of a
convex piecewise linear function
$\bar\rho:\mathrm{Conv}(A)\to\R$ whose maximal domains of
linearity are exactly the cells of $\mathcal{P}$.

\begin{definition}\label{def:maxdegeneration}
The family of hypersurfaces $H_\tau\subset V$ has a {\em maximal degeneration}
for\/ $\tau\to 0$ if it is given by equations of the form \eqref{eq:Ht} where
$\rho$ is adapted to a regular polyhedral decomposition $\mathcal{P}$
of\/ $\mathrm{Conv}(A)$.
\end{definition}

The logarithm map
$\Log_\tau:\mathbf{x}=(x_1,\dots,x_n)\mapsto \frac{1}{|\log \tau|}(\log |x_1|,\dots,
\log |x_n|)$ maps $H_\tau$ to its {\em amoeba}\/ $\Pi_\tau=\Log_\tau(H_\tau\cap
V^0)$; it is known \cite{Mikhalkin,Rullgard} that,
for $\tau\to 0$, the amoeba $\Pi_\tau\subset \R^n$ converges to the {\it tropical
hypersurface} $\Pi_0\subset\R^n$ defined by the tropical polynomial 
\begin{equation}\label{eq:Pi0}
\varphi(\xi)=\max\,\{\langle \alpha,\xi\rangle-\rho(\alpha)\,|\,\alpha\in A\}
\end{equation}
(namely, $\Pi_0$ is the set of points where the maximum is achieved more
than once). Combinatorially, $\Pi_0$ is the dual cell complex of 
$\mathcal{P}$; in particular the connected components of
$\R^n\setminus\Pi_0$ can be naturally labelled by the
elements of $\mathcal{P}^{(0)}=A$, according to which term
achieves the maximum in \eqref{eq:Pi0}.

\begin{example}\label{ex:genus2}
The toric variety $V=\PP^1\times\PP^1$ is defined by the fan
$\Sigma\subseteq \R^2$ whose rays are generated by 
$\sigma_1=(1,0)$, $\sigma_2=(0,1)$, $\sigma_3=(-1,0)$, $\sigma_4=(0,-1)$.
The piecewise linear function $\lambda:\Sigma\to\R$ with
$\lambda(\sigma_1)=\lambda(\sigma_2)=0$, $\lambda(\sigma_3)=3$,
and $\lambda(\sigma_4)=2$ defines the line bundle 
$\mathcal{L}=\mathcal{O}_{\PP^1\times\PP^1}(3,2)$, whose associated polytope 
is $P=\{(v_1,v_2)\in\R^2:\ 0\le v_1\le 3,\ 0\le v_2\le 2\}$. Let $A=
P\cap \Z^2$. The regular decomposition of $P$ shown in 
Figure~\ref{fig:Pi0-genus2} (left) is induced by the function 
$\rho:A\to\R$ whose values are given in the figure. The corresponding
tropical hypersurface $\Pi_0\subseteq \R^2$ is shown in Figure~\ref{fig:Pi0-genus2} (right); $\Pi_0$ is the limit of the amoebas of a 
maximally degenerating family of
smooth genus 2 curves $H_\tau\subset V$ as $\tau\to 0$.

When the toric variety $V$ is non-compact, $P$ is unbounded, and
the convex hull of $A$ is only a proper subset of $P$. For instance,
Figure \ref{fig:Pi0-genus2} also represents
a maximally degenerating family of smooth genus 2 curves in 
$V^0\simeq (\C^*)^2$
(where now $P=\R^2$).
\end{example}

\begin{figure}
\setlength{\unitlength}{1cm}\psset{unit=\unitlength}
\begin{picture}(3.5,2.5)(-0.25,-0.25)
\psframe(0,0)(3,2)
\psline(1,0)(1,2) \psline(2,0)(2,2)
\psline(0,1)(3,1) \psline(0,2)(2,0) \psline(1,2)(3,0)
\psline(0,1)(1,0) \psline(2,2)(3,1)
\put(0.1,0.1){\small 5}
\put(1.1,0.1){\small 2}
\put(2.1,0.1){\small 1}
\put(3.1,0.1){\small 2}
\put(0.1,1.1){\small 2}
\put(1.1,1.1){\small 0}
\put(2.1,1.1){\small 0}
\put(3.1,1.1){\small 2}
\put(0.1,2.1){\small 2}
\put(1.1,2.1){\small 1}
\put(2.1,2.1){\small 2}
\put(3.1,2.1){\small 5}
\end{picture}
\qquad
\setlength{\unitlength}{3.5mm}\psset{unit=\unitlength}
\begin{picture}(8,8)(-4,-4)
\psline(-4,-3)(-3,-3)
\psline(-4,0)(-2,0)
\psline(-3,-4)(-3,-3)
\psline(-1,-4)(-1,-2)
\psline(1,-4)(1,-1)
\psline(4,3)(3,3)
\psline(4,0)(2,0)
\psline(3,4)(3,3)
\psline(1,4)(1,2)
\psline(-1,4)(-1,1)
\pspolygon(-2,-2)(-1,-2)(0,-1)(0,1)(-1,1)(-2,0)
\psline(-3,-3)(-2,-2)
\pspolygon(2,2)(1,2)(0,1)(0,-1)(1,-1)(2,0)
\psline(3,3)(2,2)
\end{picture}
\caption{A regular decomposition of the polytope for
$\mathcal{O}_{\PP^1\times\PP^1}(3,2)$, and the corresponding tropical
hypersurface.}
\label{fig:Pi0-genus2}
\end{figure}
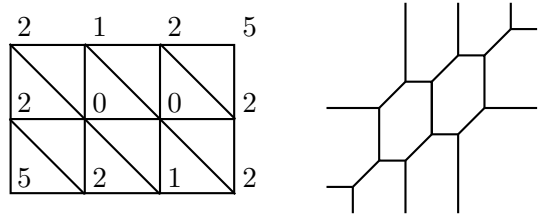

We now turn to the symplectic geometry of the situation we just considered.
Assume that $V$ is equipped with a complete toric K\"ahler metric, with
K\"ahler form $\omega_V$. The torus $T^n=(S^1)^n$ acts on
$(V,\omega_V)$ by Hamiltonian diffeomorphisms; we denote by
$\mu_V:V\to\R^n$ the corresponding moment map. It is well-known that
the image of $\mu_V$ is a convex 
polytope $\Delta_V\subset \R^n$, dual to the fan $\Sigma_V$. The
preimage of the interior of $\Delta_V$ is the open stratum $V^0\subset V$;
over $V^0$ the logarithm map $\Log_\tau$ and the moment map $\mu_V$ are
related by some diffeomorphism $g_\tau:\R^n\stackrel{\sim}{\to}\mathrm{int}(\Delta_V)$.

For a fixed K\"ahler form $\omega_V$, the
diffeomorphism $g_\tau$ gets rescaled by a factor of $|\log\tau|$ as $\tau$
varies;
in particular, the moment map images 
$\mu_V(H_\tau)=\overline{g_\tau(\Pi_\tau)}
\subseteq \Delta_V$ of a degenerating family of hypersurfaces collapse
towards the boundary of $\Delta_V$ as $\tau\to 0$. 
This can be avoided by considering a
varying family of K\"ahler forms $\omega_{V,\tau}$, obtained from the
given $\omega_V$ by symplectic inflation along all the toric divisors of
$V$, followed by a rescaling so that $[\omega_{V,\tau}]=
[\omega_V]$ is independent of~$\tau$. (To be more
concrete, one could e.g.\ consider a family of toric K\"ahler forms 
which are multiples of
the standard complete K\"ahler metric of
$(\C^*)^n$ over increasingly large open subsets of $V^0$.)
\medskip

Throughout this paper, we will consider smooth hypersurfaces that are close 
enough to the tropical limit, namely hypersurfaces of the form considered
above with $\tau$ sufficiently close to $0$. The key requirement we have
for ``closeness'' to the tropical limit is that the amoeba should lie in a 
sufficiently small neighborhood of the tropical hypersurface $\Pi_0$, so
that the complements have the same combinatorics. Since we consider a
single hypersurface rather than the whole family, we will omit $\tau$ from the notation.

\begin{definition}\label{def:neartrop}
A smooth hypersurface $H=f^{-1}(0)$ in a toric variety $V$ is {\em nearly 
tropical} if it is a member of a maximally degenerating family of 
hypersurfaces as above, with the property that
the amoeba $\Pi=\Log(H)\subset \R^n$ is entirely
contained inside a neighborhood of the tropical hypersurface
$\Pi_0$ which retracts onto $\Pi_0$.
\end{definition}

In particular, each element $\alpha\in A$ determines a non-empty open
component of $\R^n\setminus \Pi$; we will (abusively) refer
to it as the component over which the monomial of $f$ with weight $\alpha$
dominates.

We equip $V$ with a toric K\"ahler form $\omega_V$ of the form
discussed above, and denote by $\mu_V$ and $\Delta_V$ the moment map and its
image.  Let $\delta>0$ be a constant such that  a standard symplectic tubular neighborhood $U_H$ of $H$ of
size $\delta$ embeds into $V$ and  the complement
of the moment map image $\mu_V(U_H)$  has a non-empty 
component for each element of $A$ (i.e.\ for each monomial in $f$).


\begin{remark}
The assumption that the degeneration is maximal is made purely for 
convenience, and to ensure that the toric variety $Y$ constructed in \S
\ref{ss:mirror} below is smooth. However, all of our arguments work equally
well in the case of non-maximal degenerations. 
\end{remark}

\subsection{Blowing up}\label{ss:blowup}

Our main goal is to study SYZ mirror symmetry for the blow-up $X$ of
$V\times\C$ along $H\times 0$, equipped with a suitable K\"ahler form. 

Recalling that the defining equation $f$ of $H$ is a section of a line
bundle $\mathcal{L}\to V$, the normal bundle to $H\times 0$ in $V\times\C$
is the restriction of $\mathcal{L}\oplus \mathcal{O}$, and we can construct 
explicitly $X$ as a hypersurface in the total space of the 
$\PP^1$-bundle $\PP(\mathcal{L}\oplus \mathcal{O})
\to V\times \C$. Namely, the defining section of $H\times 0$ projectivizes to 
a section $\mathbf{s}(\mathbf{x},y)=(f(\mathbf{x}):y)$ of
$\PP(\mathcal{L}\oplus \mathcal{O})$
over the complement of $H\times 0$; and $X$ is the closure of the graph of
$\mathbf{s}$. In other terms,
\begin{equation}\label{eq:Xblowup}
X=\{(\mathbf{x},y,(u:v))\in \PP(\mathcal{L}\oplus\mathcal{O})\,|\,
f(\mathbf{x})v=yu\}.
\end{equation}
In this description it is clear that the projection $p:X\to V\times\C$
is a biholomorphism outside of the exceptional divisor $E=p^{-1}(H\times
0)$.

The $S^1$-action on $V\times \C$ by rotation of
the $\C$ factor preserves $H\times 0$ and hence lifts to an $S^1$-action on 
$X$. This action preserves the exceptional divisor $E$, and acts by rotation
in the standard manner on each fiber of the $\PP^1$-bundle $p_{|E}:E\to
H\times 0$. In coordinates, we can write this action in the form:
\begin{equation}\label{eq:s1action}
e^{i\theta}\cdot
(\mathbf{x},y,(u:v))=(\mathbf{x},e^{i\theta}y,(u:e^{i\theta}v)).
\end{equation}
Thus, the fixed
point set of the $S^1$-action on $X$ consists of two disjoint strata: the proper
transform $\tilde{V}$ of $V\times 0$ (corresponding to $y=0,\ v=0$ in the
above description), and the section $\tilde{H}$ of $p$ 
over $H\times 0$ given by the line subbundle $\mathcal{O}$ of the normal
bundle (i.e., the point $(0:1)$ in each fiber of $p_{|E}$).

The open stratum $V^0\times \C^*$ of the toric variety $V\times\C$ carries a
holomorphic $(n+1)$-form 
$\Omega_{V\times\C}=i^{n+1}\prod_j d\log x_j\wedge d\log y$, which has simple poles
along the toric divisor $D_{V\times\C}=(V\times 0)\cup (D_V\times \C)$ (where
$D_V=V\setminus V^0$ is the union of the toric divisors in $V$). The
pullback $\Omega=p^*(\Omega_{V\times\C})$ has
simple poles along the proper transform of $D_{V\times\C}$, namely the
anticanonical divisor $D=\tilde{V}\cup p^{-1}(D_V\times\C)$. The complement
$X^0=X\setminus D$, equipped with the $S^1$-invariant holomorphic $(n+1)$-form $\Omega$, is 
an open Calabi-Yau manifold.

\begin{remark}\label{rmk:conicbundle}
$X\setminus \tilde{V}$ corresponds to $v\neq 0$ in \eqref{eq:Xblowup},
so it is isomorphic to an affine conic bundle over $V$, namely the
hypersurface in the total space of $\mathcal{O}\oplus \mathcal{L}$
given~by
\begin{equation}\label{eq:conicbundle}
\{(\mathbf{x},y,z)\in \mathcal{O}\oplus \mathcal{L}\,|\,f(\mathbf{x})=yz\}.
\end{equation}
Further removing the fibers over $D_V$, we conclude that $X^0$ is a conic
bundle over the open stratum $V^0\simeq (\C^*)^n$, given again by the
equation $\{f(\mathbf{x})=yz\}$.
\end{remark}

We equip $X$ with an 
$S^1$-invariant K\"ahler form $\omega_\epsilon$ for which the fibers of 
the exceptional divisor have a sufficiently small area $\epsilon>0$.
Specifically, we require that $\epsilon\in (0,\delta/2)$, where $\delta$ is
the size of the standard tubular neighborhood of $H$ that embeds in
$(V,\omega_V)$. The most natural way to construct such a K\"ahler form
would be to equip $\mathcal{L}$ with a Hermitian metric, which determines a
K\"ahler form on $\PP(\mathcal{L}\oplus \mathcal{O})$ and, by restriction, on
$X$; on the complement of $E$ the resulting K\"ahler form is given by
\begin{equation}\label{eq:omegaeps0}
p^*\omega_{V\times\C}+\frac{i\epsilon}{2\pi}\partial\bar{\partial}
\log(|f(\mathbf{x})|^2+|y|^2),
\end{equation}
where $\omega_{V\times\C}$ is the product K\"ahler form on $V\times\C$
induced by the toric K\"ahler form $\omega_V$ on $V$ and the standard area
form of $\C$.

However, from a symplectic perspective the blowup operation
amounts to deleting from $V\times\C$
a standard symplectic tubular neighborhood of $H\times 0$ and
collapsing its boundary (an $S^3$-bundle over $H$) onto $E$ by the Hopf map.
Thus, $X$ and $V\times\C$ are symplectomorphic away from neighborhoods of 
$E$ and $H\times 0$; to take full advantage of this, we will choose 
$\omega_\epsilon$ in such a way that the projection
$p:X\to V\times\C$ is a symplectomorphism away from a neighborhood of the exceptional divisor.
Namely, we set
\begin{equation}\label{eq:omegaeps}
\omega_\epsilon=p^*\omega_{V\times\C}+\frac{i\epsilon}{2\pi}\partial
\bar{\partial}\left(\chi(\mathbf{x},y)\log(|f(\mathbf{x})|^2+|y|^2)\right),
\end{equation} 
where $\chi$ is a suitably chosen
$S^1$-invariant smooth cut-off function supported in a tubular neighborhood of
$H\times 0$, with $\chi=1$ near $H\times 0$. 
It is clear
that \eqref{eq:omegaeps} defines a K\"ahler form provided $\epsilon$ is small enough;
specifically, $\epsilon$ needs to be such that a standard symplectic
neighborhood of size $\epsilon$ of $H\times 0$ can be embedded
($S^1$-equivariantly) into the support of $\chi$. 
For simplicity, we assume that $\chi$ is chosen so that the
following property holds:

\begin{property}\label{ass:chi}
The support of $\chi$ is contained inside $p^{-1}(U_H\times B_\delta)$,
where $U_H\subset V$ is a standard symplectic $\delta$-neighborhood of $H$ 
and $B_\delta\subset\C$ is the disc of radius $\delta$.
\end{property}

\begin{remark}
$\omega_\epsilon$ lies in the same cohomology class
$[\omega_\epsilon]=p^*[\omega_{V\times\C}]-\epsilon[E]$ as the K\"ahler form
defined by \eqref{eq:omegaeps0}, and is equivariantly symplectomorphic to it.
\end{remark}

\subsection{The mirror $B$-side Landau-Ginzburg model}\label{ss:mirror}
Using the same notations as in the previous section, we now describe
a $B$-side Landau-Ginzburg model which we claim is SYZ mirror to $X$ (with
the K\"ahler form $\omega_\epsilon$, and relatively to the anticanonical
divisor~$D$).

Recall that the hypersurface $H\subset X$ has a defining equation of the
form \eqref{eq:Ht}, involving toric monomials whose weights range over a
finite subset $A\subset \Z^n$, forming the vertices of a polyhedral
complex $\mathcal{P}$ (cf.\ Definition \ref{def:maxdegeneration}).

We denote by $Y$ the (noncompact) $(n+1)$-dimensional toric variety defined
by the fan $\Sigma_Y=\R_{\ge 0}\cdot(\mathcal{P}\times \{1\})\subseteq 
\R^{n+1}=\R^n\oplus\R$. Namely, the integer generators of the rays of 
$\Sigma_Y$ are the vectors of the form $(-\alpha,1)$, $\alpha\in A$, and the vectors
$(-\alpha_1,1),\dots,(-\alpha_k,1)$ span a cone of $\Sigma_Y$ if and only if
$\alpha_1,\dots,\alpha_k$ span a cell of $\mathcal{P}$.

Dually, $Y$ can be described by a (noncompact) polytope $\Delta_Y\subseteq
\R^{n+1}$, defined in terms of the tropical polynomial $\varphi:\R^n\to\R$
associated to $H$ (cf.\ \eqref{eq:Pi0}) by
\begin{equation}
\label{eq:DeltaY}
\Delta_Y=\{(\xi,\eta)\in\R^n\oplus\R\,|\,\eta\ge \varphi(\xi)\}.
\end{equation}

\begin{remark}
The polytope $\Delta_Y$ also determines a K\"ahler class $[\omega_Y]$ on
$Y$. While in this paper we focus on the A-model of $X$ and the B-model of
$Y$, it can be shown that the family of complex structures on $X$ obtained
by blowing up $V\times \C$ along the maximally degenerating family 
$H_\tau\times 0$ (cf.\ \S \ref{ss:tropdeg}) corresponds to a family of
K\"ahler forms asymptotic to $|\log\tau|[\omega_Y]$ as $\tau\to 0$.
\end{remark}

\begin{remark}
Even though deforming the hypersurface $H$ inside $V$ does not modify
the symplectic geometry of $X$, the topology of $Y$ depends on the chosen
polyhedral decomposition $\mathcal{P}$ (i.e., on
the combinatorial type of the tropical hypersurface defined by $\varphi$). 
However, the various possibilities for $Y$ are related to each other by 
crepant birational transformations, and hence are
expected to yield equivalent B-models. (The A-model of $Y$, on the other 
hand, is affected by these birational transformations and does depend on the tropical
polynomial $\varphi$, as explained in the previous remark.)
\end{remark}

The facets of $\Delta_Y$ correspond to the maximal domains of linearity of
$\varphi$. Thus the irreducible toric divisors of $Y$ are in one-to-one
correspondence with the connected components of $\R^n\setminus \Pi_0$,
and the combinatorics of the toric strata of $Y$ can be immediately
read off the tropical hypersurface $\Pi_0$ (see Example \ref{ex:genus2part2}
below).
\medskip

It is advantageous for our purposes to introduce a collection of affine
charts on $Y$ indexed by the elements of $A$ (i.e., the facets of $\Delta_Y$,
or equivalently, the connected components of $\R^n\setminus \Pi_0$).

For each $\alpha\in A$, let $Y_\alpha=(\K^*)^n\times\K$,
with coordinates
$\mathbf{v}_\alpha=(v_{\alpha,1},\dots,v_{\alpha,n})\in(\K^*)^n$ and
$v_{\alpha,0}\in\K$ (as before, $\K$ is either $\Lambda$ or $\C$).
Whenever $\alpha,\beta\in A$ are connected by an edge in the polyhedral
decomposition $\mathcal{P}$ (i.e., whenever the corresponding components
of $\R^n\setminus \Pi_0$ share a top-dimensional facet, with primitive
normal vector $\beta-\alpha$), we glue $Y_\alpha$ to $Y_\beta$ by the
coordinate transformations
\begin{equation}\label{eq:toriccharts}
\begin{cases}
v_{\alpha,i}=v_{\beta,0}^{\beta_i-\alpha_i}\, v_{\beta,i}
\qquad (1\le i\le n),
\\ v_{\alpha,0}=v_{\beta,0}.
\end{cases}\end{equation}
These charts cover the complement in $Y$ of the codimension 2 strata
(as $Y_\alpha$ covers the open stratum of $Y$ and the open stratum of
the toric divisor corresponding to~$\alpha$).
In terms of the standard basis of toric monomials indexed by weights in
$\Z^{n+1}$, $v_{\alpha,0}$ is the monomial with weight $(0,\dots,0,1)$,
and for $i\ge 1$ $v_{\alpha,i}$ is the monomial with weight
$(0,\dots,-1,\dots,0,-\alpha_i)$ (the $i$-th entry is $-1$).

Denoting by $T$ the Novikov parameter (treated as
an actual complex parameter when $\K=\C$),
and by $v_0$ the common coordinate $v_{\alpha,0}$ for all charts, we set
\begin{equation}\label{eq:w0}
w_0=-T^{\epsilon}+T^{\epsilon}v_0.
\end{equation}
With this notation, the above coordinate transformations can be rewritten as
$$
v_{\alpha,i}=\left(1+T^{-\epsilon} w_0\right)^{\beta_i-\alpha_i} v_{\beta,i},\qquad 1\le i\le n.
$$
More generally, for $m=(m_1,\dots,m_n)\in\Z^n$ we set $\mathbf{v}_\alpha^m=
v_{\alpha,1}^{m_1}\dots v_{\alpha,n}^{m_n}$. Then
\begin{equation}\label{eq:gluing}
\mathbf{v}_\alpha^m=(1+T^{-\epsilon} w_0)^{\langle \beta-\alpha, m\rangle}
\mathbf{v}_\beta^m.
\end{equation}
We shall see that $w_0$ and the transformations \eqref{eq:gluing} have a
natural interpretation in terms of the enumerative geometry of holomorphic
discs in $X$.

Next, recall from \S \ref{ss:tropdeg} that the inward normal vectors to the facets
of the moment polytope $\Delta_V$ associated to $(V,\omega_V)$ are the
primitive integer generators 
$\sigma_1,\dots,\sigma_r$ of the rays of $\Sigma_V$. Thus, there exist
constants $\varpi_1,\dots,\varpi_r\in\R$ such that
\begin{equation}\label{eq:Delta_V}
\Delta_V=\{u\in \R^n\,|\,\langle \sigma_i,u\rangle + \varpi_i\ge 0
\ \ \forall 1\le i\le r\}.
\end{equation}
Then for $i=1,\dots,r$ we set
\begin{equation}\label{eq:wi}
w_i=T^{\varpi_i} {\mathbf v}_{\alpha_i}^{\sigma_i}
\end{equation}
where $\alpha_i\in A$ is chosen to lie on the facet of $P$ defined by $\sigma_i$, i.e. so that $\langle \sigma_i,\alpha_i\rangle$ is minimal.  Hence, by the conditions imposed in \S \ref{ss:tropdeg},
$\langle \sigma_i,\alpha_i\rangle+\lambda(\sigma_i)=0$, where $\lambda:
\Sigma_V\to\R$ is the piecewise linear function defining
$\mathcal{L}=\mathcal{O}(H)$. 
 By \eqref{eq:gluing}, the choice of $\alpha_i$ satisfying the required condition
is irrelevant: in all cases ${\mathbf v}_{\alpha_i}^{\sigma_i}$ is
simply the toric monomial with weight $(-\sigma_i,\lambda(\sigma_i))\in
\Z^n\oplus \Z$. Moreover, this weight pairs non-negatively with
all the rays of the fan $\Sigma_Y$, therefore $w_i$ defines a 
regular function on $Y$.

With all the notation in place, we can at last make the following
definition, which clarifies the statements of Theorems \ref{thm:main}
and \ref{cor:main}:

\begin{definition}\label{def:LGmirror}
We denote by $Y^0$ the complement of the hypersurface $D_Y=w_0^{-1}(0)$ in
the toric $(n+1)$-fold $Y$, and define the {\em leading-order superpotential} 
\begin{equation}\label{eq:W0}
W_0=w_0+w_1+\dots+w_r=-T^\epsilon+T^\epsilon v_0+
\sum_{i=1}^r T^{\varpi_i}\mathbf{v}^{\sigma_i}_{\alpha_i}\in\mathcal{O}(Y).
\end{equation}
We also define
\begin{equation}\label{eq:W0H}
W_0^H=-v_0+w_1+\dots+w_r=-v_0+
\sum_{i=1}^r T^{\varpi_i}\mathbf{v}^{\sigma_i}_{\alpha_i}\in\mathcal{O}(Y).
\end{equation}
\end{definition}

\begin{remark}
Since there are no convergence issues, we can think of $(Y^0,W_0)$ and $(Y,W_0^H)$
either as $B$-side Landau-Ginzburg models defined over the 
Novikov field or as one-parameter families of complex $B$-side Landau-Ginzburg models defined
over~$\C$. 
\end{remark}

\begin{example}\label{ex:genus2part2}
When $H$ is the genus 2 curve of Example \ref{ex:genus2}, 
the polytope $\Delta_Y$ has 12 facets (2 of them compact and the 10 others 
non-compact), corresponding to the 12 components of $\R^n\setminus \Pi_0$, and
intersecting exactly as pictured on Figure~\ref{fig:Pi0-genus2} right.
The edges of the figure correspond to the configuration of $\PP^1$'s and
$\mathbb{A}\!^1$'s along which the toric divisors of the 3-fold $Y$ intersect.

Label the irreducible toric divisors by $D_{a,b}$ ($0\le a\le 3$, $0\le b\le 2$),
corresponding to the elements $(a,b)\in A$. Then
the leading-order superpotential $W_0$ consists of five terms: $w_0=-T^\epsilon+
T^\epsilon v_0$, where $v_0$ is the toric monomial of weight $(0,0,1)$,
which vanishes with multiplicity 1 on each of the 12 toric divisors; and
up to constant factors, $w_1$ is the toric monomial with weight $(-1,0,0)$, 
which vanishes with multiplicity
$a$ on $D_{a,b}$; $w_2$ is the toric monomial with weight
$(0,-1,0)$, vanishing with multiplicity $b$ on $D_{a,b}$;
$w_3$ is the monomial with weight $(1,0,3)$, with multiplicity $(3-a)$ on
$D_{a,b}$; and $w_4$ is the monomial with weight $(0,1,2)$, with
multiplicity $(2-b)$ on $D_{a,b}$. In particular, the compact divisors $D_{1,1}$
and $D_{2,1}$ are components of the singular fiber
$\{W_0=-T^\epsilon\}\subset Y^0$ (which also has a third, non-compact
component); and similarly for $\{W_0^H=0\}\subset Y$.

(In general the order of vanishing of $w_i$ on a given divisor is equal to
the intersection number with $\Pi_0$ of a semi-infinite ray in the direction of 
$-\sigma_i$ starting from a generic point in the relevant component of $\R^n\setminus \Pi_0$.)

This example does not satisfy Assumption \ref{ass:affinecase}, and in this
case the actual mirror of $X$ differs from $(Y^0,W_0)$ by higher-order 
correction terms. On the other hand, if we consider the genus
2 curve with 10 punctures $H\cap V^0$ in the open toric variety 
$V^0\simeq (\C^*)^2$, which does fall within the scope of Theorem
\ref{thm:main}, the
construction yields the same toric 3-fold $Y$, but now 
we simply have $W_0=w_0$ (resp.\ $W_0^H=-v_0$).
\end{example}

\section{Lagrangian torus fibrations on blowups of toric varieties}
\label{s:fibrations}

As in \S \ref{ss:blowup}, we consider a smooth nearly tropical
hypersurface $H=f^{-1}(0)$ in a toric variety $V$ of dimension $n$,
and the blow-up $X$ of $V\times\C$ along $H\times 0$, equipped with
the $S^1$-invariant K\"ahler form $\omega_\epsilon$ given by
\eqref{eq:omegaeps}. Our goal in this section is to construct an $S^1$-invariant 
Lagrangian torus fibration $\pi:X^0\to B$
(with appropriate singularities) on the
open Calabi-Yau manifold $X^0=X\setminus D$, where $D$ is the proper
transform of the toric anticanonical divisor of $V\times\C$. (Similar
fibrations have been previously considered by Gross \cite{GrossTop,GrossSlag} 
and by Casta\~no-Bernard and Matessi \cite{CBM1,CBM2}.)
The key observation is that $S^1$-invariance forces the fibers of $\pi$ to
be contained in the level sets of the moment map of the 
$S^1$-action. Thus, we begin by studying the geometry of the reduced
spaces.

\subsection{The reduced spaces}\label{ss:redspaces}
The $S^1$-action \eqref{eq:s1action} on $X$ is Hamiltonian with respect
to the K\"ahler form $\omega_\epsilon$ given by \eqref{eq:omegaeps}, 
and its moment map $\mu_X:X\to\R$
can be determined explicitly. 
Outside of the exceptional divisor, we
identify $X$ with $V\times\C$ via the projection~$p$, and observe that
$\mu_X(\mathbf{x},y)=\int_{D(\x,y)} \omega_\epsilon$, where
$D(\x,y)$ is a disc bounded by the orbit of $(\mathbf{x},y)$, namely
the total transform of $\{\mathbf{x}\}\times D^2(|y|)\subset
V\times\C$. (We normalize $\mu_X$ so that it takes
the constant value $0$ over the proper transform of $V\times 0$; also,
our convention differs from the usual one by a factor of $2\pi$.)

Hence, for given $\x$ the quantity $\mu_X(\x,y)$ 
is a strictly increasing function of $|y|$. Moreover,
applying Stokes' theorem we find that
\begin{equation}\label{eq:muX}
\mu_X(\x,y)=\pi |y|^2+\frac{\epsilon}{2}|y|\frac{\partial}{\partial
|y|}\left(\chi(\x,y)\log(|f(\x)|^2+|y|^2)\right).
\end{equation}
In the regions where $\chi$ is constant  this simplifies to: 
\begin{equation}\label{eq:muX_explicit}
\mu_X(\x,y)=\begin{cases}\pi |y|^2+\epsilon\, \dfrac{|y|^2}{
|f(\x)|^2+|y|^2}&\mbox{where $\chi\equiv 1$ (near $E$),}\\
\pi |y|^2&\mbox{where $\chi\equiv 0$ (away from $E$).}\end{cases}
\end{equation}
(Note that the first expression extends naturally to a smooth function over $E$.)

The critical points of $\mu_X$ are the fixed points of the $S^1$-action.
Besides $\tilde{V}=\mu_X^{-1}(0)$, the fixed points occur along $\tilde{H}$,
which lies in the level set $\mu_X^{-1}(\epsilon)$; in
particular, all the other level sets of $\mu_X$ are smooth.
Since for any given $\x$ the moment map $\mu_X$ is a strictly increasing 
function of $|y|$, each level set of $\mu_X$ intersects $p^{-1}(\{\x\}\times\C)$ along
a single $S^1$-orbit. Hence,
for $\lambda>0$, the natural 
projection to $V$ (obtained by composing $p$ with projection to 
the first factor) yields a natural identification of the reduced 
space $X_{red,\lambda}=\mu_X^{-1}(\lambda)/S^1$ with $V$. 

For $\lambda\gg \epsilon$, $\mu_X^{-1}(\lambda)$ is disjoint from the
support of the cut-off function $\chi$, and the reduced K\"ahler
form $\omega_{red,\lambda}$ on $X_{red,\lambda}\cong V$ coincides with
the toric K\"ahler form $\omega_V$. As $\lambda$ becomes closer to
$\epsilon$, $\omega_{red,\lambda}$ differs from $\omega_V$ near $H$ but
remains cohomologous to it. At the critical level $\lambda=\epsilon$,
the reduced form $\omega_{red,\epsilon}$ is singular along $H$ (but
its singularities are fairly mild, see Lemma \ref{l:omegared}).
Finally, for $\lambda<\epsilon$ the
K\"ahler form $\omega_{red,\lambda}$ differs from $\omega_V$ in a
tubular neighborhood of $H$, inside which the normal direction to $H$ has
been symplectically {\em deflated}. In particular, one easily checks that
\begin{equation}
\label{eq:redclass}
[\omega_{red,\lambda}]=[\omega_V]-\max(0,\epsilon-\lambda)[H].
\end{equation}

Our goal is to exploit the toric structure of $V$ to construct families of
Lagrangian tori in $X_{red,\lambda}$. The K\"ahler form 
$\omega_{red,\lambda}$ on $X_{red,\lambda}\cong V$ is not $T^n$-invariant 
near $H$; in fact it isn't even smooth along $H$ for $\lambda=\epsilon$.
However, there exist (smooth) toric K\"ahler forms $\omega'_{V,\lambda}$,
depending piecewise smoothly on $\lambda$, with $[\omega'_{V,\lambda}]=
[\omega_{red,\lambda}]$; see \eqref{eq:omega'Vlambda} for an explicit
construction.
The following result will be proved in Appendix \ref{s:reducedappendix}.

\begin{lemma}\label{l:rectifyomegared}
There exists a family of homeomorphisms $(\phi_{\lambda})_{\lambda\in
\R_+}$ of\/ $V$ such that: 
\begin{enumerate}
\item $\phi_\lambda$ preserves the toric divisor
$D_V\subset V$;
\item the restriction of $\phi_\lambda$ to $V^0$ is a 
diffeomorphism for $\lambda\neq \epsilon$, and a diffeomorphism outside of $H$
for $\lambda=\epsilon$;
\item $\phi_\lambda$ intertwines the reduced
K\"ahler form $\omega_{red,\lambda}$ and the toric K\"ahler form
$\omega'_{V,\lambda}$;
\item $\phi_\lambda=\mathrm{id}$ at every point whose $T^n$-orbit is
disjoint from the support of $\chi$;
\item $\phi_\lambda$ depends on $\lambda$ in a continuous manner, and
smoothly except at $\lambda=\epsilon$.
\end{enumerate}
\end{lemma}

\noindent
The diffeomorphism (singular along $H$ for $\lambda=\epsilon$) $\phi_\lambda$ given by Lemma \ref{l:rectifyomegared}
yields a preferred Lagrangian torus fibration on the open stratum
$X_{red,\lambda}^0=(\mu_X^{-1}(\lambda)\cap X^0)/S^1$ of
$X_{red,\lambda}$ (naturally identified with $V^0$ under the canonical
identification $X_{red,\lambda}\cong V$), namely the preimage by
$\phi_\lambda$ of the standard fibration of
$(V^0,\omega'_{V,\lambda})$ by $T^n$-orbits:

\begin{definition}
We denote by $\pi_\lambda:X_{red,\lambda}^0\to \R^n$ the composition
$\pi_\lambda=\Log\circ \phi_\lambda$, where $\Log:V^0 \cong (\C^*)^n \to \R^n$ is
the logarithm map $(x_1,\dots,x_n)\mapsto
\frac{1}{|\log\tau|}(\log |x_1|,\dots,\log |x_n|)$,
and $\phi_\lambda:(X_{red,\lambda},\omega_{red,\lambda})\to
(V,\omega'_{V,\lambda})$ is as in Lemma \ref{l:rectifyomegared}.
\end{definition}

\begin{remark}
By construction, the natural affine structure (see \S \ref{ss:syz1}) on the
base of the Lagrangian torus fibration $\pi_\lambda$ identifies it with
the interior of the moment polytope $\Delta_{V,\lambda}$ associated to the cohomology class
$[\omega'_{V,\lambda}]=[\omega_{red,\lambda}]\in H^2(V,\R)$.
\end{remark}

\subsection{A Lagrangian torus fibration on $X^0$}\label{ss:Lagr-torus-fibration}
We now assemble the Lagrangian torus fibrations $\pi_\lambda$ on the reduced
spaces into a (singular) Lagrangian torus fibration on~$X^0$:

\begin{definition} \label{def:Lagr-torus-fibration}
We denote by $\pi:X^0\to B=\R^n\times \R_+$ the map which sends the point
$x\in \mu_X^{-1}(\lambda)\cap X^0$ to $\pi(x)=(\pi_\lambda(\bar{x}),\lambda)$,
where $\bar{x}\in X_{red,\lambda}^0$ is the $S^1$-orbit of $x$.
\end{definition}

The map $\pi$ is continuous, and smooth away from $\lambda=\epsilon$.
The fiber of $\pi$ above $(\xi,\lambda)\in B$ is obtained by lifting
the Lagrangian torus $\pi_\lambda^{-1}(\xi)\subset X_{red,\lambda}$ to
$\mu_X^{-1}(\lambda)$ and ``spinning'' it by the $S^1$-action. 

Away from the fixed points of  the $S^1$-action, $\mu_X^{-1}(\lambda)$ 
is a coisotropic manifold with isotropic foliation given by the
$S^1$-orbits. Moreover, the $S^1$-bundle $\mu_X^{-1}(\lambda)\to X_{red,\lambda}$
is topologically trivial for $\lambda>\epsilon$ (setting $y\in\R_+$ gives 
a global section), trivial over the complement of $H$ for
$\lambda=\epsilon$, and the circle bundle associated to the line bundle
$\O(-H)$ for $\lambda<\epsilon$; in any case, its restriction to
a fiber of $\pi_\lambda$ is topologically trivial.
The fibers of $\pi_\lambda$ are smooth Lagrangian tori 
(outside of $H$ when $\lambda=\epsilon$, which corresponds precisely to the 
$S^1$-fixed points); therefore, we conclude that the fibers of $\pi$ are 
smooth Lagrangian tori unless they contain fixed points of the $S^1$-action.

The only fixed points
occur for $\lambda=\epsilon$, when $\mu_X^{-1}(\lambda)$ contains the
stratum of fixed points $\tilde{H}$. The identification of the
reduced space with $V$ maps $\tilde{H}$ to the 
hypersurface $H$, so the singular fibers map to
\begin{equation} \label{eq:Bsing}
B^{sing}=\Pi'\times \{\epsilon\}\subset B,
\end{equation}
where $\Pi'=\pi_\epsilon(H\cap V^0)\subset \R^n$ is essentially the amoeba of the
hypersurface $H$ (up to the fact that $\pi_\epsilon$ differs from the
logarithm map by $\phi_\epsilon$).
The fibers above the points of $B^{sing}$ differ from the regular fibers
in that, where a smooth fiber $\pi^{-1}(\xi,\lambda)\simeq T^{n+1}$ 
is a trivial 
$S^1$-bundle over $\pi_\lambda^{-1}(\xi)\simeq T^n\subset V^0$, for
$\lambda=\epsilon$ some of the $S^1$ fibers (namely those which lie over
points of $H$) are collapsed to points.

Because the fibration $\pi$ has non-trivial monodromy around $B^{sing}$,
the only globally defined affine coordinate on $B$ is the last
coordinate $\lambda$ (the moment map of the $S^1$-action); other affine
coordinates are only defined over subsets of
$B\setminus B^{sing}$, i.e.\ in the complement of certain cuts. 
Our preferred choice for such a description relates
the affine structure on $B$ to the moment polytope $\Delta_V\times\R_+$ of
$V\times\C$. Namely, away from a tubular neighborhood of
$\Pi'\times (0,\epsilon)$ the Lagrangian torus fibration $\pi$ coincides
with the standard toric fibration on $V\times\C$:

\begin{proposition}\label{prop:pistandard}
Outside of the support of $\chi$ (a tubular neighborhood of the exceptional
divisor $E$), the K\"ahler form $\omega_{\epsilon}$ is equal to
$p^*\omega_{V\times\C}$, and the moment map of the $S^1$-action is the
standard one $\mu_X(\x,y)=\pi|y|^2$. Moreover, outside of
$\pi(\mathrm{supp}\,\chi)$, the fibers of the Lagrangian
fibration $\pi$ are standard product tori, i.e.\ they are the
preimages by $p$ of the orbits of the $T^{n+1}$-action in $V\times\C$.
\end{proposition}

\proof
The first statement follows immediately from formulas \eqref{eq:omegaeps} and 
\eqref{eq:muX}. The second one is then a direct consequence of
the manner in which $\pi$ was constructed and condition (3) in 
Lemma \ref{l:rectifyomegared}.
\endproof

Recall that the support of $\chi$ is constrained by Property \ref{ass:chi}.
Thus, the fibration $\pi$ is standard (coincides with the standard toric
fibration on $V\times\C$) over a large subset $B^{std}=(\R^n\times \R_+)\setminus
(\Log(U_H)\times (0,\delta))$ of $B$. Since
$\omega_\epsilon=p^*\omega_{V\times\C}$ over $\pi^{-1}(B^{std})$, 
we conclude that over $B^{std}$ the affine structure of $B$
agrees with that for the standard toric fibration of $V\times \C$, i.e.\
as an affine manifold $B^{std}$ can be naturally identified with the complement of
$\mu_V(U_H)\times (0,\delta)$ inside $\mathrm{int}(\Delta_V)\times\R_+$.

This description of the affine structure on $B\setminus B^{sing}$ can be
extended from $B^{std}$ to the complement of a set of codimension 1 cuts.
Recall from \S \ref{ss:syz1} that the affine coordinates of
$b\in B\setminus B^{sing}$ relative to some reference point $b_0$ are
given by the symplectic areas of certain relative 2-cycles
$(\Gamma_1,\dots,\Gamma_{n+1})$ 
with boundary on $\pi^{-1}(b)\cup\pi^{-1}(b_0)$; the above identification
of $B^{std}$ with a subset of $\Delta_V\times\R_+$ arises from taking
the boundaries of $\Gamma_i$ to be (homologous to) orbits of the various
$S^1$ factors of the $T^{n+1}$-action on $V\times\C$.

When $b$ and $b_0$ have the same last coordinate $\lambda>\epsilon$,
we can choose $\Gamma_1,\dots,\Gamma_n$ to be contained in
$\mu_X^{-1}(\lambda)$, and obtained as the lifts
of relative 2-cycles $\Gamma_{i,red}$ in $X_{red,\lambda}$ with boundary 
on fibers of $\pi_\lambda$; we can fix
such lifts by requiring that $y\in \R_+$ on $\Gamma_i$. Since
$\int_{\Gamma_{i}}\omega_\epsilon=\int_{\Gamma_{i,red}}\omega_{red,\lambda}$,
the affine structure on the level 
set $\R^n\times\{\lambda\}
\subset B$ is the same as that on the base of the fibration $\pi_\lambda$
on the reduced space $X_{red,\lambda}$, which can be identified via the
diffeomorphism 
$\phi_\lambda$ with the standard toric fibration on
$(V,\omega'_{V,\lambda})$. For $\lambda>\epsilon$ we have
$[\omega_{red,\lambda}]=[\omega'_{V,\lambda}]=[\omega_V]$, 
so the base is naturally identified with
the interior of the moment polytope $\Delta_V$; moreover, this
identification is consistent with our previous description of the affine
structure over $B^{std}$, since in that region the various K\"ahler forms 
agree pointwise.

In other terms, over $\R^n\times (\epsilon,\infty)\subset B$, the affine
structure is globally a product $\mathrm{int}(\Delta_V)\times (\epsilon,\infty)$ of
the affine structure on the moment polytope of $(V,\omega_V)$ and the
interval $(\epsilon,\infty)$, in a manner that extends the previous
description over $B^{std}$.

For $\lambda<\epsilon$, the affine structure on $\R^n\times\{\lambda\}\subset
B$ can described similarly, by choosing relative 2-cycles $\Gamma_{i,red}$
in $X_{red,\lambda}$ with boundary on fibers of $\pi_\lambda$ 
and lifting them to relative 2-cycles $\Gamma'_i$ in $\mu_X^{-1}(\lambda)$.
Since the lifts may
intersect the exceptional divisor $E$, we cannot require $y\in \R_+$
as in the case $\lambda>\epsilon$. Instead, we use the monomial 
$\x^{\alpha_0}$ for some $\alpha_0\in A$ to fix a trivialization of
$\mathcal{L}=\O(H)$ over $V^0$, and choose the
lifts so that $\x^{-\alpha_0} z=\x^{-\alpha_0}f(\x)/y\in \R_+$ on $\Gamma'_i$.
Since $\int_{\Gamma'_i}\omega_\epsilon=\int_{\Gamma_{i,red}}
\omega_{red,\lambda}$, the affine structure on the level 
set $\R^n\times\{\lambda\}
\subset B$ is again identical to that on the base of the fibration $\pi_\lambda$
on $X_{red,\lambda}$, or equivalently via
$\phi_\lambda$, the standard toric fibration on
$(V,\omega'_{V,\lambda})$. Thus, the affine structure identifies
$\R^n\times\{\lambda\}\subset B$ with the interior of the moment
polytope $\Delta_{V,\lambda}$ associated to the K\"ahler class
$[\omega'_{V,\lambda}]=[\omega_{red,\lambda}]=[\omega_V]-\max(0,\epsilon-\lambda)[H]$.
However, this description is no longer consistent with that previously given
for $B^{std}$, because the boundary of $\Gamma'_i$ does
not represent the expected homology class in $\pi^{-1}(b)$.

Specifically, assume $b_0$ and $b\in (\R^n\setminus
\Log(U_H))\times\{\lambda\}$ lie in the connected components corresponding
to $\alpha_0$ and $\alpha\in A$ respectively. Then the boundary of
$\Gamma'_i$ in $\pi^{-1}(b_0)$ does represent the homology class of the
orbit of the $i$-th $S^1$-factor, while the boundary in $\pi^{-1}(b)$
differs from it by $\alpha_i-\alpha_{0,i}$ times the orbit of the last
$S^1$-factor. Moreover,
$$\int_{\Gamma_{i,red}}\omega_V-\int_{\Gamma_{i,red}}
\omega_{red,\lambda}=
(\epsilon-\lambda)(\Gamma_{i,red}\cdot H)=
(\epsilon-\lambda)(\alpha_i-\alpha_{0,i}).$$
This formula also gives the difference between the $\omega_{\epsilon}$-areas of the relative cycles $\Gamma'_i$
and the relative cycles $\Gamma_i\subset
\pi^{-1}(B^{std})$ previously used to determine affine coordinates over
$B^{std}$. Hence, the affine coordinates determined by the 
relative cycles $\Gamma'_i$ differ from those constructed previously over
$B^{std}$ by a shear 
\begin{equation}\label{eq:shear_affine}
(\zeta_1,\dots,\zeta_n,\lambda)\mapsto
\bigl(\zeta_1+(\epsilon-\lambda)(\alpha_1-\alpha_{0,1}),\dots,
\zeta_n+(\epsilon-\lambda)(\alpha_n-\alpha_{0,n}),\lambda\bigr)
\end{equation}
or more succinctly, $(\zeta,\lambda)\mapsto
\bigl(\zeta+(\epsilon-\lambda)(\alpha-\alpha_0),\lambda\bigr)$.

More globally, over $\R^n\times (0,\epsilon)\subset B$ the affine structure can be
identified (using the relative cycles $\Gamma'_i$ to define 
coordinates) with a piece of the moment polytope for the total space of the 
line bundle $\mathcal{O}(-H)$ over $V$ (equipped
with a toric K\"ahler form in the class $[\omega_V]-\epsilon[H]$),
consistent with the fact that the normal bundle to $\tilde{V}$
inside $X$ is $\mathcal{O}(-H)$; but this description is not consistent
with the one we have given over~$B^{std}$. 

\begin{figure}
\setlength{\unitlength}{2cm}
\begin{picture}(3.2,1.9)(-1.2,-2)
\psset{unit=\unitlength}
\newgray{ltgray}{0.8}
\pspolygon[fillstyle=solid,fillcolor=ltgray,linestyle=none](0,-1)(0,-2)(1,-2)
\psline(-1,-2)(2,-2)
\psline(-1,-2)(-1,-0.2)
\psline(2,-2)(2,-0.2)
\put(0,-1){\makebox(0,0)[cc]{\small $\times$}}
\psline[linestyle=dashed,dash=0.1 0.1](0,-1)(0,-2)
\psline[linestyle=dashed,dash=0.1 0.1](0,-1)(1,-2)
\psarc{<->}(0.3,0){1.65}{-108}{-72}
\psline{->}(-0.5,-1.7)(-0.5,-1.4)
\psline{->}(-0.5,-1.7)(-0.2,-1.7)
\psline{->}(1.1,-1.7)(1.4,-1.7)
\psline{->}(1.1,-1.7)(0.8,-1.4)
\psline{<->}(-1.12,-2)(-1.12,-1) \put(-1.3,-1.5){\small $\epsilon$}
\end{picture}
\qquad\qquad
\begin{picture}(3.3,1.9)(-0.2,0.2)
\psset{unit=\unitlength}
\newgray{vltgray}{0.95}
\newgray{ltgray}{0.9}
\newgray{dkgray}{0.85}
\pspolygon[linestyle=none,fillstyle=solid,fillcolor=vltgray](0,0.8)(1,1.8)(2.9,1.8)(1.9,0.8)
\psline[linestyle=dashed,dash=0.1 0.1](0,0.2)(0.5,0.7)
\pscustom[fillstyle=solid,fillcolor=gray,linewidth=0.3pt,curvature=.5 .1 0]{
  \psline(0.49,1.29)(0.9,1.27)
  \pscurve(0.9,1.27)(1.3,1.18)(0.98,0.8)
  \pscurve[liftpen=1](1.02,0.8)(1.4,1.1)(1.65,1.3)(2.9,1.8)
  \pscurve[liftpen=1](2.8,1.8)(1.55,1.4)(0.51,1.31)
}
\psline[fillstyle=solid,fillcolor=dkgray,linestyle=dashed,dash=0.06 0.06,linewidth=0.3pt](1,0.2)(1,0.8)(1.6,0.2)
\psline(0,0.2)(2,0.2)
\psline(0,0.2)(0,2)
\psline[linewidth=0.5pt](1,1.8)(0,0.8)(1.9,0.8)
\psline{<->}(-0.12,0.2)(-0.12,0.8) \put(-0.3,0.42){\small $\epsilon$}
\end{picture}
\caption{The base of the Lagrangian torus fibration $\pi:X^0\to B$.
Left: $H=\{\text{point}\}\subset
\CP^1$. Right: $H=\{x_1+x_2=1\}\subset \C^2$.}
\label{fig:affineB}
\end{figure}
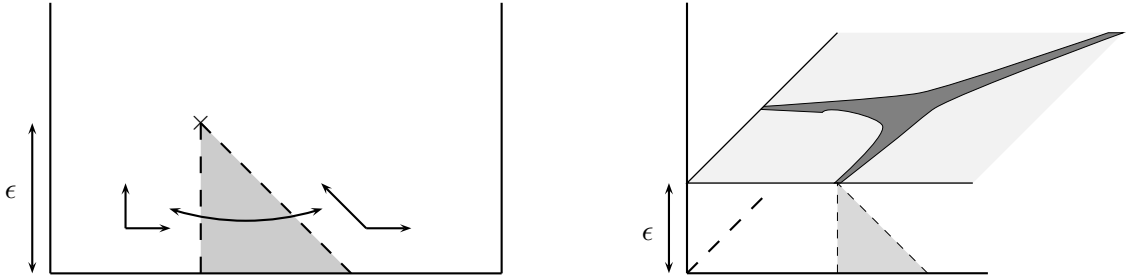

On the other hand, the shears \eqref{eq:shear_affine} map 
the complement of the amoeba of $H$ in
$\Delta_{V,\lambda}$ to the complement of a standard
$(\epsilon-\lambda)$-neighborhood of the amoeba of $H$ in $\Delta_V$.
Thus, making cuts along the projection of the exceptional divisor, 
we can extend the affine coordinates previously described over $B^{std}$, 
and identify the affine structure on $B\setminus (\Pi'\times (0,\epsilon])$ 
with an open subset of $\mathrm{int}(\Delta_V)\times\R_+$, obtained by 
deleting an $(\epsilon-\lambda)$-neighborhood of the amoeba of $H$ from
$\mathrm{int}(\Delta_V)\times\{\lambda\}$ for all $\lambda\in (0,\epsilon]$.

This is the picture of $B$ that we choose to emphasize, depicting it
as the complement of a set of ``triangular'' cuts inside $\Delta_V\times\R_+$; see
Figure \ref{fig:affineB}.

\begin{remark}
While the fibration we construct is merely Lagrangian, it is very reasonable
to conjecture that in fact $X^0$ carries an $S^1$-invariant {\em special Lagrangian} fibration over
$B$. The holomorphic $(n+1)$-form $\Omega=p^*\Omega_{V\times\C}$ on $X^0$ is
$S^1$-invariant, and induces a holomorphic $n$-form on the reduced space
$X_{red,\lambda}^0$, which turns out to coincide with the standard
toric form $\Omega_V=i^n\prod_j d\log x_j$. Modifying the construction of
the fibration $\pi_\lambda:X_{red,\lambda}^0\to \R^n$ so that its fibers 
are special Lagrangian with respect to $\Omega_V$ would then be sufficient
to ensure that the fibers of $\pi$ are special Lagrangian with respect to~$\Omega$.
In dimension 1 this is easy to accomplish by elementary methods.
In higher dimensions, making $\pi_\lambda$ special Lagrangian requires
the use of analysis, as the deformation of product tori in $V^0$
(which are special Lagrangian with respect to $\omega'_{V,\lambda}$ and 
$\Omega_V$) to tori which are special Lagrangian for $\omega_{red,\lambda}$
and $\Omega_V$ is governed by a first-order elliptic PDE \cite{McLean}
(see also \cite[\S 9]{joycenotes} or \cite[Prop.~2.5]{Au1}). If one were to
argue as in the proof of Lemma \ref{l:rectifyomegared} (cf.\ Appendix
\ref{s:reducedappendix}), the 1-forms 
used to construct $\phi_\lambda$ should be chosen not only to satisfy the usual condition
for Moser's trick, but also to be co-closed with respect to a suitable
rescaling of the K\"ahler metric induced by $\omega_{t,\lambda}$.
When $V=(\C^*)^n$ this does not seem to pose any major
difficulties, but in general it is not obvious that one can ensure
the appropriate behavior along the toric divisors.
\end{remark}

\section{SYZ mirror symmetry for $X^0$} \label{s:opencase}

In this section we apply the procedure described in \S \ref{s:syz} to
the Lagrangian torus fibration $\pi:X^0\to B$ of \S \ref{s:fibrations} in order
to construct the SYZ mirror to the open Calabi-Yau manifold $X^0$ and 
prove Theorem \ref{thm:conicbundle}. The key observation is that, 
by Proposition \ref{prop:pistandard}, most fibers
of $\pi$ are mapped under the projection $p$ to standard product 
tori in the toric variety $V\times\C$; therefore, the holomorphic discs
bounded by these fibers can be understood by reducing to the toric case,
which is well understood (see e.g.\ \cite{cho-oh}).

\begin{proposition}\label{prop:unobstructed}
The fibers of $\pi:X^0\to B$ which bound holomorphic discs in $X^{0}$ are those which intersect
the subset $p^{-1}(H\times\C)$. 

Moreover, the simple holomorphic discs in $X^0$
bounded by such a fiber contained in $\mu_X^{-1}(\lambda)$ have Maslov 
index 0 and symplectic area $|\lambda-\epsilon|$, and their boundary represents
the homology class of an $S^1$-orbit if $\lambda>\epsilon$ and its negative
otherwise.
\end{proposition}

\proof
Let $L\subset X^0$ be a smooth fiber of $\pi$, contained in $\mu_X^{-1}(\lambda)$
for some $\lambda\in \R_+$, and let $u:(D^2,\partial D^2)\to (X^0,L)$ be a 
holomorphic disc with boundary in $L$. Denote by $L'$ the projection of
$L$ to $V$ (i.e., the image of $L$ by the composition $p_V$ of $p$ and
the projection to the first factor). The restriction of $p_V$ to $\mu_X^{-1}(\lambda)$ 
coincides with the quotient map to the reduced space
$X_{red,\lambda}\simeq V$; thus, $L'$ is in fact a fiber of
$\pi_\lambda$, i.e.\ a Lagrangian torus in $(V^0,\omega_{red,\lambda})$,
smoothly isotopic to a product torus inside $V^0\simeq (\C^*)^n$.

Since the relative homotopy group $\pi_2(V^0,L')\simeq
\pi_2((\C^*)^n,(S^1)^n)$ vanishes, the holomorphic disc
$p_V\circ u:(D^2,\partial D^2)\to (V^0,L')$ is necessarily constant.
Hence the image of the disc $u$ is contained inside a fiber $p_V^{-1}(\x)$
for some $\x\in V^0$.

If $\x\not\in H$, then $p_V^{-1}(\x)\cap X^0=p^{-1}(\{\x\}\times\C^*)\simeq
\C^*$, inside which $p_V^{-1}(\x)\cap L$ is
a circle centered at the origin (an orbit of the $S^1$-action). The
maximum principle then implies that the map $u$ is necessarily constant.

On the other hand, when $\x\in H$, $p_V^{-1}(\x)\cap X^0$ is the union of
two affine lines intersecting transversely at one point: the proper transform of $\{\x\}\times\C$, 
and the fiber of $E$ above $\x$ (minus the point where it intersects
$\tilde{V}$). Now, $p_V^{-1}(\x)\cap L$ is again an $S^1$-orbit, i.e.\ a circle
inside one of these two components (depending on whether $\lambda>\epsilon$ 
or $\lambda<\epsilon$); either way, $p_V^{-1}(\x)\cap L$
bounds exactly one non-constant embedded holomorphic disc in $X^0$ (and all of its
multiple covers). The result follows.
\endproof

Denote by $B^{reg}\subset B$ the set of those fibers of 
$\pi$ which do not intersect $p^{-1}(U_H\times\C)$.  From Property \ref{ass:chi} and
Propositions \ref{prop:pistandard} and \ref{prop:unobstructed}, we
deduce:

\begin{corollary}
The fibers of $\pi$ above the points of $B^{reg}$ are tautologically unobstructed
in $X^0$, and project under $p$ to standard product tori in $V^0\times\C$.
\end{corollary}

With respect to the affine structure,
$B^{reg}=(\R^n\setminus \Log(U_H))\times\R_+$ is naturally isomorphic to
$(\Delta_V\setminus \mu_V(U_H))\times \R_+$.  
\begin{definition} \label{def:chamber}
The \emph{chamber} $U_\alpha$ is the connected component  of $B^{reg}$ 
over which the monomial of weight $\alpha$ dominates all other
monomials in the defining equation of $H$. 
\end{definition}

\begin{remark}\label{rmk:Bregretracts}
By construction, the complement of $\Log(U_H)$ is a deformation retract of
the complement of the amoeba of $H$ inside $\R^n$; so the set of tautologically
unobstructed fibers of $\pi$ retracts onto $B^{reg}=\bigsqcup U_\alpha$.
\end{remark}

As explained in \S \ref{ss:syz1},
$U_\alpha$ determines an affine coordinate chart
$U_\alpha^\vee$ for the SYZ mirror of $X^0$, with coordinates of the
form~\eqref{eq:localchart}.

Specifically, fix a reference point $b^0\in U_\alpha$, and
observe that, since $L^0=\pi^{-1}(b^0)$ is the lift of an orbit of the 
$T^{n+1}$-action on $V\times\C$, its first homology carries a preferred 
basis $(\gamma_1,\dots,\gamma_n,\gamma_0)$ consisting of orbits of the various $S^1$ factors.
Consider $b\in U_\alpha$, with coordinates $(\zeta_1,\dots,\zeta_n,\lambda)$
(here we identify $U_\alpha\subset B^{reg}$ with a subset of the moment
polytope  $\Delta_V\times\R_+\subset \R^{n+1}$ for the $T^{n+1}$-action on
$V\times\C$), and denote by $(\zeta_1^0,\dots,\zeta_n^0,\lambda^0)$ the
coordinates of $b^0$. Then the valuations of the coordinates given by
\eqref{eq:localchart}, i.e., the areas of the cylinders
$\Gamma_1,\dots,\Gamma_n,\Gamma_0$ bounded by
$L^0$ and $L=\pi^{-1}(b)$, are $\zeta_1-\zeta_1^0,\dots$, $\zeta_n-\zeta_n^0$,
and $\lambda-\lambda^0$ respectively.  In order to eliminate the dependence
on the choice of $L^0$, we rescale each coordinate by a suitable power of
$T$, and equip $U_\alpha^\vee$ with the coordinate system
\begin{equation}\label{eq:chartalpha}
(L,\nabla)\mapsto (v_{\alpha,1},\dots,v_{\alpha,n},w_{\alpha,0})=
\left(T^{\zeta_{1}}\/\nabla(\gamma_1),\,\dots,\,T^{\zeta_n}\/\nabla(\gamma_n),
\,T^{\lambda\,}\nabla(\gamma_0)\right).
\end{equation}
(Compare with \eqref{eq:localchart}, noting that 
$\zeta_i=\zeta_i^0+\int_{\Gamma_i}\omega_\epsilon$ and
$\lambda=\lambda^0+\int_{\Gamma_0}\omega_\epsilon$.)

As in \S \ref{ss:mirror}, we set $\mathbf{v}_\alpha=(v_{\alpha,1},\dots,v_{\alpha,n})$,
and for $m\in \Z^n$ we write $\mathbf{v}_\alpha^m=
v_{\alpha,1}^{m_1}\dots v_{\alpha,n}^{m_n}$. Moreover, 
 we write $w_0$ for $w_{\alpha,0}$; this is \emph{a priori}
ambiguous, but we shall see shortly 
that the gluings between the charts preserve the last
coordinate.
\medskip

The ``naive'' gluings between these coordinate charts (i.e., those which
describe the geometry of the space of $(L,\nabla)$ up to Hamiltonian 
isotopy without accounting for instanton corrections) are governed by the
global affine structure of $B\setminus B^{sing}$. Their description is
instructive, even though it is not necessary for our argument.

For $\lambda>\epsilon$ the affine structure is globally that of 
$\Delta_V\times (\epsilon,\infty)$. Therefore, \eqref{eq:chartalpha} makes sense and is
consistent with \eqref{eq:localchart} even when $b$ does not lie in
$U_\alpha$; thus, for $\lambda>\epsilon$ the naive gluing is the identity
map ($\v_\alpha=\v_\beta$, and $w_{\alpha,0}=w_{\beta,0}$).

On the other hand, for $\lambda\in (0,\epsilon)$ we argue as 
in \S \ref{ss:Lagr-torus-fibration} (cf.\ equation \eqref{eq:shear_affine} and the
preceding discussion).
When $b=(\zeta_1,\dots,\zeta_n,\lambda)$ lies in a different chamber
$U_\beta$ from that  containing the reference point $b^0$ (i.e.,
$U_\alpha$), the
intersection number of a cylinder $\Gamma'_i$ constructed as previously
with the exceptional
divisor $E$ is equal to $\beta_i-\alpha_i$, and its symplectic area 
differs from $\zeta_i-\zeta_i^0$
by $(\beta_i-\alpha_i)(\epsilon-\lambda)$. Moreover, due to the monodromy
of the fibration, the bases of first homology used in $U_\alpha$ and
$U_\beta$ differ by $\gamma_i\mapsto \gamma_i+(\beta_i-\alpha_i)\gamma_0$
for $i=1,\dots,n$.
Thus, for $\lambda<\epsilon$ the naive gluing between
the charts $U_\alpha^\vee$ and $U_\beta^\vee$ corresponds to setting
$$v_{\alpha,i}=T^{-(\beta_i-\alpha_i)(\epsilon-\lambda)}\nabla(\gamma_0)^{\beta_i-\alpha_i}
v_{\beta,i}=(T^{-\epsilon} w_0)^{\beta_i-\alpha_i} v_{\beta,i},\qquad
1\le i\le n.$$

The naive gluing formulas for the two cases ($\lambda>\epsilon$ and
$\lambda<\epsilon$) are inconsistent. As seen in \S \ref{ss:syz1},
this is not unexpected: the actual gluing between the coordinate 
charts $\{U_\alpha^\vee\}_{\alpha\in A}$ differs from these formulas by instanton corrections 
which account for the bubbling of holomorphic discs as $L$ is isotoped
across a wall of potentially obstructed fibers.

Given a potentially obstructed fiber $L\subset \mu_X^{-1}(\lambda)$, the
simple holomorphic discs bounded by $L$ are classified by Proposition
\ref{prop:unobstructed}.
For $\lambda>\epsilon$, the symplectic area of these discs is 
$\lambda-\epsilon$, and their boundary loop represents the class 
$\gamma_0\in H_1(L)$ (the orbit of the $S^1$-action), so the corresponding 
weight is $T^{\lambda-\epsilon}\nabla(\gamma_0)\ (=T^{-\epsilon}w_0)$; while for 
$\lambda<\epsilon$ the symplectic area is $\epsilon-\lambda$ and the
boundary loop represents $-\gamma_0$, so the weight is
$T^{\epsilon-\lambda}\nabla(\gamma_0)^{-1}\ (=T^\epsilon w_0^{-1})$. As explained in \S
\ref{ss:syz1}, we therefore expect the instanton corrections to the
gluings to be given by power series in $(T^{-\epsilon} w_0)^{\pm 1}$.

While the direct calculation of the multiple cover contributions to the
instanton corrections would require sophisticated machinery, Remark
\ref{rmk:instcorrtrick} provides a way to do so by purely elementary
techniques. Namely, we study the manner in which counts of Maslov index
2 discs in partial compactifications of $X^0$ vary between chambers.
The reader is referred to Example 3.1.2 of \cite{Au2} for a simple
motivating example (corresponding to the case where $H=\{\mathrm{point}\}$
in $V=\C$).
\medskip

Recall that a point of $U_\alpha^\vee$ corresponds to a pair $(L,\nabla)$,
where $L=\pi^{-1}(b)$ is the fiber of $\pi$ above some point 
$b\in U_\alpha$, and $\nabla$ is a unitary rank 1 local system on~$L$.
Given a partial compactification $X'$ of $X^0$ (satisfying Assumption
\ref{ass:nef}), $(L,\nabla)$ is a weakly unobstructed object
of $\F(X')$, i.e.\ $\m_0(L,\nabla)=W_{X'}(L,\nabla)\,e_L$,
where $W_{X'}(L,\nabla)$ is a weighted count of Maslov index 2
holomorphic discs bounded by $L$ in $X'$. Varying $(L,\nabla)$, these
weighted counts define regular functions on each chart $U_\alpha^\vee$,
and by Corollary \ref{cor:potential_matched_by_gluing}, they glue into a global regular
function on the SYZ mirror of $X^0$.

We first use this idea to verify that the coordinate $w_0=w_{\alpha,0}$ is
preserved by the gluing maps, by interpreting it as a weighted count of discs in the 
partial compactification $X^0_+$ of $X^0$ obtained
by adding  the open stratum $\tilde{V}{}^0$ of the divisor $\tilde{V}$.

\begin{lemma}\label{l:countw0}
Let $X^0_+=p^{-1}(V^0\times\C)=X^0\cup \tilde{V}{}^0\subset X$.
Then any point $(L,\nabla)$ of $U_\alpha^\vee$ defines a weakly unobstructed
object of $\F(X^0_+)$, with 
\begin{equation}\label{eq:WX0+}
W_{X^0_+}(L,\nabla)=w_{\alpha,0}.
\end{equation}
\end{lemma}

\proof 
Let $u:(D^2,\partial D^2) \to (X^0_+,L)$ be a holomorphic
disc in $X^0_+$ with boundary on $L$ whose Maslov index is 2.
The image of $u$ by the projection $p$ is a holomorphic disc in
$V^0\times\C\simeq (\C^*)^n\times\C$ with boundary on the product torus
$p(L)=S^1(r_1)\times\dots\times S^1(r_0)$. Thus, the first $n$ 
components of $p\circ u$ are constant by the maximum
principle, and we can write $p\circ u(z)=(x_1,\dots,x_n,r_0\gamma(z))$,
where $|x_1|=r_1$, \dots, $|x_n|=r_n$, and $\gamma:D^2\to\C$ maps the unit circle
to itself. Moreover, the Maslov index of $u$ is twice its intersection
number with $\tilde{V}$. Therefore $\gamma$ is a degree 1 map of the
unit disc to itself, i.e.\ a biholomorphism; so the choice of 
$(x_1,\dots,x_n)$ determines $u$ uniquely up to reparametrization.

We conclude that each point of $L$ lies on the boundary of a unique
Maslov index~2 holomorphic disc in $X^0_+$, namely the preimage by $p$
of a disc $\{\x\}\times D^2(r_0)$. These discs are easily 
seen to be regular, by reduction to the toric case \cite{cho-oh}; their symplectic area 
is $\lambda$ (by definition of the moment map $\mu_X$, see the beginning of 
\S \ref{ss:redspaces}), and their boundary represents the homology class
$\gamma_0\in H_1(L)$ (the orbit of the $S^1$-action on $X$).
Thus, their weight is $T^{\omega(u)}\nabla(\partial
u)=T^\lambda\nabla(\gamma_0)=w_{\alpha,0}$, which completes the proof.
\endproof

Lemma \ref{l:countw0} implies that the local coordinates $w_{\alpha,0}\in
\O(U_\alpha^\vee)$ glue to a globally defined regular function $w_0$ on
the mirror of $X^0$ (hence we drop $\alpha$ from the notation).

Next, we consider monomials in the remaining coordinates $\v_{\alpha}$. 
First, let $\sigma\in\Z^n$ be a primitive generator of a ray of the fan $\Sigma_V$,
and denote by $D^0_\sigma$ the open stratum of the corresponding toric 
divisor in $V$. We will presently see that the monomial $\v_{\alpha}^\sigma$ is
related to a weighted count of discs in the partial compactification 
$X'_\sigma$ of $X^0$ obtained by adding  $p^{-1}(D^0_\sigma\times\C)$:
\begin{equation}
  X'_\sigma=p^{-1}((V^0\cup D^0_\sigma)\times\C)\setminus
\tilde{V}\subset X.
\end{equation}
Let $\varpi\in\R$ be the constant such that the corresponding facet of 
$\Delta_V$ has equation $\langle \sigma,u\rangle+\varpi=0$, and let
$\alpha_{min}\in A$ be such that $\langle \sigma,\alpha_{min}\rangle$ 
is minimal.
\begin{lemma}\label{l:Wsigma}
Any point $(L,\nabla)$ of $U_\alpha^\vee$ $(\alpha\in A)$ defines
a weakly unobstructed object of $\F(X'_\sigma)$, with
\begin{equation}\label{eq:Wsigma}
W_{X'_\sigma}(L,\nabla)=(1+T^{-\epsilon}w_0)^{\langle \alpha-\alpha_{min},
\sigma\rangle} T^\varpi \v_\alpha^\sigma.
\end{equation}
\end{lemma}

\proof 
After performing dual monomial changes of coordinates on $V^0$ and on 
$U_\alpha^\vee$ (i.e., replacing the coordinates $(x_1,\dots,x_n)$
by $(\x^{\tau_1},\dots,\x^{\tau_n})$ where $\langle \sigma,\tau_i\rangle=
\delta_{i,1}$, and $(v_{\alpha,1},\dots,v_{\alpha,n})$ by
$(\v_\alpha^\sigma,\dots)$), we can reduce to the case where $\sigma=(1,0,\dots,0)$,
and $V^0\cup D^0_\sigma\simeq \C\times (\C^*)^{n-1}$.

With this understood, let $u:(D^2,\partial D^2)\to (X'_\sigma,L)$ be
a Maslov index 2 holomorphic disc with boundary on $L$. The composition of
$u$ with the projection $p$ is a holomorphic disc in $(V^0\cup
D^0_\sigma)\times\C\simeq \C\times (\C^*)^{n-1}\times\C$ with boundary
on the product torus $p(L)=S^1(r_1)\times\dots\times S^1(r_0)$.
Thus, all the components of $p\circ u$ except for the first and last ones
are constant by the maximum principle. Moreover, since the Maslov index
of $u$ is twice its intersection number with $D^0_\sigma$, the first
component of $p\circ u$ has a single zero, i.e.\ it is a biholomorphism
from $D^2$ to the disc of radius $r_1$. 
Therefore, up to reparametrization we
have $p\circ u(z)=(r_1z,x_2,\dots,x_n,r_0\gamma(z))$, where
$|x_2|=r_2,\dots,$ $|x_n|=r_n$, and $\gamma:D^2\to \C$ maps the unit
circle to itself.

A further constraint is given by the requirement that the image of $u$
be disjoint from $\tilde{V}$ (the proper transform of $V\times 0$). Thus, the last component $\gamma(z)$ is
allowed to vanish only when $(r_1z,x_2,\dots,x_n)\in H$, and its vanishing
order at such points is constrained as well. We claim that the intersection
number $k$ of the disc $\mathbb{D}=D^2(r_1)\times \{(x_2,\dots,x_n)\}$ with $H$ is equal
to $\langle \alpha-\alpha_{min},\sigma\rangle$. Indeed, with respect to the
chosen trivialization of $\O(H)$ over $V^0$, near $p_V(L)$ the dominating term 
in the defining section of $H$ is the monomial $\x^\alpha$, whose values
over the circle $S^1(r_1)\times\{(x_2,\dots,x_n)\}$ wind $\alpha_1=\langle
\alpha,\sigma\rangle$ times around the origin; whereas near $D_\sigma^0$
(i.e., in the chambers which are unbounded in the direction of $-\sigma$)
the dominating terms have winding number $\langle
\alpha_{min},\sigma\rangle$.  Comparing these winding numbers we obtain
that $k=\langle \alpha-\alpha_{min},\sigma\rangle$.

Assume first that $(x_2,\dots,x_n)$ are generic, in the sense that
$\mathbb{D}$ intersects $H$ transversely at $k$ distinct
points $(r_1a_i,x_2,\dots,x_n)$, $i=1,\dots,k$ (with $a_i\in D^2$). 
The condition that $u$ avoids $\tilde{V}$ implies that $\gamma$ 
is allowed to have at most simple zeroes at $a_1,\dots,a_k$.
Denote by $I\subseteq \{1,\dots,k\}$ the set of those $a_i$ at which
$\gamma$ does have a zero, and let
$$\gamma_I(z)=\prod_{i\in I} \frac{z-a_i}{1-\bar{a}_iz}.$$
Then $\gamma_I$ maps the unit circle to itself, and its zeroes in the
disc are the same as those of $\gamma$, so that $\gamma_I^{-1}\gamma$ is
a holomorphic function on the unit disc, without zeroes, and mapping the 
unit circle to itself, i.e.\ a constant map. Thus
$\gamma(z)=e^{i\theta}\gamma_I(z)$, and
\begin{equation}\label{eq:blaschke}
p\circ u(z)=(r_1z,x_2,\dots,x_n,r_0e^{i\theta}\gamma_I(z))
\end{equation} for some $I\subseteq \{1,\dots,k\}$ and
$e^{i\theta}\in S^1$. We conclude that there are $2^k$ holomorphic discs
of Maslov index 2 in $(X'_\sigma,L)$ whose boundary passes through a
given generic point of~$L$. It is not hard to check that these discs are
all regular, using e.g.\ the same argument as in the proof of Lemma 7 in
\cite{AuR6}. Succinctly: observing that $u$ does not intersect $\tilde{H}$, projection to $V$ decomposes (via a short exact
sequence) the Cauchy-Riemann operator for $u$ into 
a $\bar\partial$ operator on the 
trivial holomorphic line bundle with trivial real boundary condition 
(along the fibers of the projection), and the $\bar\partial$ operator for
the ``standard'' disc $\mathbb{D}$ in $\C\times (\C^*)^{n-1}$ (which itself splits into a direct sum of
line bundles and is easily checked to be surjective); this implies surjectivity.

When the disc $\mathbb{D}$ is not transverse to $H$,
we can argue in exactly the same manner, except that $a_1,\dots,a_k\in D^2$
are no longer distinct; and $\gamma$ may have a multiple zero at $a_i$ as 
long as its order of vanishing does not exceed the multiplicity of $(r_1a_i,
x_2,\dots,x_n)$
as an intersection of $\mathbb{D}$ with $H$. We still conclude that $p\circ
u$ is of the form \eqref{eq:blaschke}. These discs are not all distinct 
(or regular), but we can argue by continuity as follows. There are 
diffeomorphisms arbitrarily $C^\infty$-close to identity which fix a
neighborhood of $H$ and map $S^1(r_1)\times \{(x_2,\dots,x_n)\}$ to a nearby circle 
$S^1(r'_1)\times\{(x'_2,\dots,x'_n)\}$ contained in a generic fiber. The moduli space of
holomorphic discs with respect to the pullback of the standard complex
structure by such a diffeomorphism is canonically identified with the moduli
space of holomorphic discs for the standard complex structure with boundary
on the nearby generic fiber. This provides an explicit regularization of
the moduli space, and
we conclude that the enumeration of holomorphic discs is as in the
transverse case (i.e., discs which can be written in the form 
\eqref{eq:blaschke} in more than one way should be counted with a multiplicity equal
to the number of such expressions.)

All that remains is to calculate the weights \eqref{eq:weight} associated
to the holomorphic discs we have identified. 
Denote by $(\zeta_1,\dots,\zeta_n,\lambda)$ the affine coordinates of 
$\pi(L)\in U_\alpha$ introduced above, and consider a disc given by
\eqref{eq:blaschke} with $|I|=\ell\in\{0,\dots,k\}$. Then the relative
homology class represented by $p\circ u(D^2)$ in 
$\C\times(\C^*)^{n-1}\times\C\subset V\times\C$ is equal to
$[D^2(r_1)\times\{pt\}]+\ell [\{pt\}\times
D^2(r_0)]$. By elementary toric geometry, the symplectic area of the disc 
$D^2(r_1)\times\{pt\}$ with respect to the toric K\"ahler form
$\omega_{V\times\C}$ is equal to $\langle\sigma,\mu_V\rangle+\varpi=
\zeta_1+\varpi$, while that of $\{pt\}\times D^2(r_0)$ is
equal to $\lambda$. Thus, the symplectic area of the disc $p\circ u(D^2)$
with respect to $\omega_{V\times\C}$ is
$\zeta_1+\varpi+\ell\lambda$. The disc we are interested in, $u(D^2)\subset
X'_\sigma$, is the proper transform of $p\circ u(D^2)$ under the blowup map;
since its intersection number with the exceptional divisor $E$ is equal to 
$|I|=\ell$, we conclude that
\begin{equation}\label{eq:areaofdisc}
\textstyle\int_{D^2} u^*\omega_\epsilon=\Bigl(\int_{D^2} (p\circ u)^*\omega_{V\times\C}
\Bigr)-\ell\epsilon =\zeta_1+\varpi+\ell(\lambda-\epsilon).
\end{equation}
On the other hand, the degree of $\gamma_{I|S^1}:S^1\to S^1$ is equal to
$|I|=\ell$, so in $H_1(L,\Z)$ we have $[u(S^1)]=\gamma_1+\ell \gamma_0$.
Thus the weight of $u$ is
$$T^{\omega_\epsilon(u)}\nabla(\partial
u)=T^{\zeta_1+\varpi+\ell(\lambda-\epsilon)}\nabla(\gamma_1)\nabla(\gamma_0)^\ell
=(T^{-\epsilon} w_0)^\ell T^\varpi v_{\alpha,1}.$$
Summing over the $\binom{k}{\ell}$ families of discs with $|I|=\ell$ for
each $\ell=0,\dots,k$, we find that $$W_{X'_\sigma}(L,\nabla)=
\sum_{\ell=0}^k \tbinom{k}{\ell} (T^{-\epsilon}w_0)^\ell\,T^\varpi
v_{\alpha,1}=(1+T^{-\epsilon}w_0)^k T^\varpi v_{\alpha,1}.$$
\vskip-1.5em
\endproof

Next we extend Lemma \ref{l:Wsigma} to the case of general
monomials in the coordinates $\mathbf{v}_\alpha$.
Let $\sigma$ be {\em any}\/ primitive element of $\Z^n$, and denote again by
$\alpha_{min}$ an element of $A$
such that $\langle \alpha_{min},\sigma\rangle$ is minimal.
Denote by $V'_\sigma=V^0\cup D^0_\sigma$ 
the toric partial compactification of $V^0$ obtained
by adding a single toric divisor $D^0_\sigma$ in the direction of the ray
$-\sigma$. The hypersurface $H^0$ admits a natural partial compactification
$H'_\sigma$ inside~$V'_\sigma$. 

We claim that $H'_\sigma$ is smooth for
$\tau$ sufficiently small in \eqref{eq:Ht}. Indeed, rescaling $f_\tau$ by
a factor of $\x^{-\alpha_{min}}$ if necessary, we can assume without loss of
generality that
$\langle \alpha_{min},\sigma\rangle=0$. Then $f_\tau$ extends to a regular function on
$V'_\sigma$, whose restriction to $D^0_\sigma$ is again a maximally
degenerating family of Laurent polynomials, associated to the regular polyhedral
decomposition $\mathcal{P}\cap \sigma^\perp$ of the convex hull of
$A\cap \sigma^\perp$. This implies that for sufficiently small $\tau$
the restriction of $f_\tau$ to $D^0_\sigma$ vanishes transversely; the
smoothness of $H'_\sigma$ follows.

By blowing up $V'_\sigma\times \C$ along
$H'_\sigma\times 0$ and removing the proper transform of $V'_\sigma\times
0$, we obtain a partial compactification $X'_\sigma$ of $X^0$.
While $X'_\sigma$ does not necessarily embed into $X$, we can equip 
$V'_\sigma$ (resp.\ $X'_\sigma$) with a toric (resp.\ $S^1$-invariant) 
K\"ahler form which agrees with $\omega_V$ (resp.\ $\omega_\epsilon$) 
everywhere outside of an arbitrarily small neighborhood of the 
compactification divisor. 

Denote by $L\subset X^0$ a smooth fiber of $\pi$ which lies in the region where
the K\"ahler forms agree (so that $L$ is Lagrangian in $X'_\sigma$ as well).

\begin{lemma}\label{l:Wsigma2}
The Maslov index 0 holomorphic discs bounded by $L$ inside $X'_\sigma$
are all contained in $X^0$ and described by Proposition \ref{prop:unobstructed}.

Moreover, if $L$ is tautologically unobstructed in $X^0$ and 
lies over the chamber $U_\alpha$, then
the points $(L,\nabla)\in U_\alpha^\vee$ define
weakly unobstructed objects of $\F(X'_\sigma)$, with 
\begin{equation}\label{eq:Wsigma2}
W_{X'_\sigma}(L,\nabla)=(1+T^{-\epsilon}w_0)^{\langle \alpha-\alpha_{min},
\sigma\rangle} T^\varpi \v_\alpha^\sigma
\end{equation}
for some $\varpi\in\R$.
\end{lemma}

\proof The Maslov index of a disc in $X'_\sigma$ with boundary
on $L$ is twice its intersection number with the compactification divisor,
and Assumption \ref{ass:nef} is satisfied (in fact $X'_\sigma$ is affine).
Thus all Maslov index 0 holomorphic discs are contained in the open stratum
$X^0$, and Proposition \ref{prop:unobstructed} holds. (Since
$L$ lies away from the compactification divisor, the symplectic
area of these discs remains the same as for $\omega_\epsilon$.)

Thus, whenever $L$ lies over a chamber $U_\alpha$ it does not bound any 
holomorphic discs of Maslov index zero or less in $X'_\sigma$, and the 
Maslov index 2 discs can be classified exactly
as in the proof of Lemma \ref{l:Wsigma}. The only difference is that,
since we evaluate the symplectic areas of these discs with respect to
the K\"ahler form on $X'_\sigma$ rather than $X$, the constant 
term $\varpi$ in the area formula \eqref{eq:areaofdisc} now depends
on the choice of the toric K\"ahler form on $V'_\sigma$ near 
the compactification divisor.
\endproof

By Remark \ref{rmk:instcorrtrick} (see also Corollary
\ref{cor:potential_matched_by_gluing}), the expressions
\eqref{eq:Wsigma2} determine globally defined regular functions 
on the mirror of $X^0$.  Thus, we can use Lemma \ref{l:Wsigma2} to
determine the wall-crossing transformations between the affine
charts of the mirror. 

Consider two adjacent chambers
$U_\alpha$ and $U_\beta$ separated by a wall of potentially obstructed fibers of $\pi$, 
i.e.\ assume that $\alpha,\beta\in A$ are connected by an edge in the
polyhedral decomposition $\mathcal{P}$. Then we have:

\begin{proposition}\label{prop:gluing}
The instanton-corrected gluing between the coordinate charts 
$U_\alpha^\vee$ and $U_\beta^\vee$ preserves the coordinate $w_0$, and
matches the remaining coordinates via
\begin{equation}\label{eq:gluingsigma}
\v_\alpha^\sigma=(1+T^{-\epsilon} w_0)^{\langle \beta-\alpha,\sigma
\rangle}\v_\beta^\sigma
\qquad \text{for all } \sigma\in\Z^n.
\end{equation}
\end{proposition}

\proof
Let $\{L_t\}_{t\in [0,1]}$ be a path among smooth fibers of $\pi$, 
with $L_0$ and $L_1$ tautologically
unobstructed and lying over the chambers $U_\alpha$ and $U_\beta$
respectively. We consider the partial compactifications $X^0_+$ and
$X'_\sigma$ of $X^0$ introduced in Lemmas \ref{l:countw0}--\ref{l:Wsigma2};
in the case of $X'_\sigma$ we choose the K\"ahler form to agree with
$\omega_\epsilon$ over a large open subset which contains the path $L_t$,
so as to be able to apply Lemma \ref{l:Wsigma2}. 

Since these partial compactifications satisfy
Assumption \ref{ass:nef}, the moduli spaces of Maslov index 0 holomorphic 
discs bounded by the Lagrangians $L_t$ in $X^0_+$, $X'_\sigma$, and $X^0$
are the same, and the corresponding wall-crossing transformations 
are identical (see Appendix \ref{sec:moduli-objects-fukay}).
Noting that the expressions \eqref{eq:WX0+} and \eqref{eq:Wsigma2} are
manifestly convergent over the whole completions $(\K^*)^{n+1}$ of
$U_\alpha^\vee$ and $U_\beta^\vee$, we appeal to Lemma
\ref{l:potential_matched_by_gluing}, and conclude that these expressions
for the superpotentials $W_{X^0_+}$
and $W_{X'_\sigma}$ over the chambers $U_\alpha^\vee$ and $U_\beta^\vee$
match under the wall-crossing transformation. Thus $w_0$ is preserved,
and for primitive $\sigma\in\Z^n$ the monomials $\v_\alpha^\sigma$ and 
$\v_\beta^\sigma$ are related by \eqref{eq:gluingsigma}. (The case of
non-primitive $\sigma$ follows obviously from the primitive case.)
\endproof

This completes the proof of Theorem \ref{thm:conicbundle}. Indeed,
the instanton-corrected gluing maps \eqref{eq:gluingsigma} coincide 
with the coordinate change formulas \eqref{eq:gluing} between the affine charts 
for the toric variety $Y$ introduced
in \S \ref{ss:mirror}. Therefore, the SYZ mirror of $X^0$ embeds inside
$Y$, by identifying the completion of the local chart $U_\alpha^\vee$ 
with the subset of $Y_\alpha$ where $w_0$ is non-zero. 
It follows that the SYZ mirror of $X^0$ is the subset of $Y$ where $w_0$ is non-zero,
namely $Y^0$.

\section{Proof of Theorem \ref{thm:main}} \label{s:generalcase}

\noindent
We now turn to the proof of Theorem \ref{thm:main}. We begin with an
elementary observation:

\begin{lemma}\label{l:fano}
If Assumption \ref{ass:affinecase} holds, then every rational curve 
$C\simeq \PP^1$ in $X$ 
satisfies $D\cdot C=c_1(X)\cdot C>0$; so in particular Assumption 
\ref{ass:nef} holds.
\end{lemma}

\proof 
$c_1(X)=p_V^*c_1(V)-[E]$, where $p_V$ is the projection to $V$ and
$E=p^{-1}(H\times 0)$ is the exceptional divisor. Consider a rational curve
$C$ in $X$ (i.e., the image of a nonconstant holomorphic map from $\PP^1$
to $X$), and denote by $C'=p_V(C)$ the rational curve in $V$ obtained by
projecting $C$ to $V$.
Applying the maximum principle to the projection to the last
coordinate $y\in \C$, we conclude that $C$ is contained either in
$p^{-1}(V\times 0)=\tilde{V}\cup E$, or in $p^{-1}(V\times \{y\})$ for
some nonzero value of $y$. 

When $C\subset p^{-1}(V\times \{y\})$ for $y\neq 0$, the curve
$C$ is disjoint from $E$ and its projection $C'$ is nonconstant, 
so $c_1(X)\cdot [C]=c_1(V)\cdot [C']>0$ by Assumption
\ref{ass:affinecase}. 

When $C$ is contained in $\tilde{V}$, the curve $C'$ is again nonconstant,
and since the normal bundle of $\tilde{V}$ in $X$ is $\mathcal{O}(-H)$, we
have $c_1(X)\cdot [C]=c_1(V)\cdot [C']-[H]\cdot [C']$, which is positive by 
Assumption \ref{ass:affinecase}.

Finally, we consider the case where $C$ is contained in $E$ but not
in $\tilde{V}$. Then $$c_1(X)\cdot [C]=[D]\cdot [C]=[\tilde{V}]
\cdot[C]+[p^{-1}(D_V)]\cdot[C]=[\tilde{V}]\cdot [C]+c_1(V)\cdot [C'].$$
The first term is non-negative by positivity of intersection; and by
Assumption \ref{ass:affinecase} the
second one is positive unless $C'$ is a constant curve, and non-negative
in any case. However $C'$ is constant only when $C$ is
(a cover of) a fiber of the $\PP^1$-bundle $p_{|E}:E\to H\times 0$;
in that case $[\tilde{V}]\cdot[C]>0$, so $c_1(X)\cdot [C]>0$ in all cases.
\endproof

As explained in \S \ref{ss:syz2}, this implies that the tautologically unobstructed fibers
of $\pi:X^0\to B$ remain weakly unobstructed in $X$, and that the SYZ mirror
of $X$ is just $Y^0$ (the SYZ mirror of $X^0$) equipped with a superpotential 
$W_0$ which counts Maslov index 2 holomorphic discs bounded by the fibers of $\pi$.
Indeed, the conclusion of Lemma \ref{l:fano} implies that any component which is a sphere contributes at least 2 to the Maslov index of a stable
genus 0 holomorphic curve bounded by a fiber of $\pi$. Thus, Maslov index 0
configurations are just discs contained in $X^0$, and Maslov index 2
configurations are discs intersecting $D$ transversely in a single
point.

Observe that each Maslov index 2 holomorphic disc intersects exactly
one of the components of the divisor $D$. Thus, the superpotential $W_0$ can be
expressed as a sum over the components of $D=\tilde{V}\cup
p^{-1}(D_V\times\C)$, in which
each term counts those discs which intersect a particular component.
It turns out that the necessary calculations have been carried out in
the preceding section: Lemma \ref{l:countw0} describes the contribution
from discs which only hit $\tilde{V}$, and Lemma
\ref{l:Wsigma} describes the contributions from discs which hit the
various components of $p^{-1}(D_V\times\C)$. Summing these, and
using the notations of \S \ref{ss:mirror}, we obtain
that, for any point $(L,\nabla)$ of $U_\alpha^\vee$ ($\alpha\in A$),
$$W_0(L,\nabla)=w_{\alpha,0}+\sum_{i=1}^r
(1+T^{-\epsilon}w_0)^{\langle \alpha-\alpha_i,\sigma_i\rangle}
T^{\varpi_i}\mathbf{v}_\alpha^{\sigma_i}=w_0+\sum_{i=1}^r w_i.$$
Hence $W_0$ is precisely the leading-order superpotential \eqref{eq:W0}.
This completes the proof of Theorem \ref{thm:main}.

\begin{remark}\label{rmk:signs}
In the proofs of Lemmas \ref{l:countw0} and \ref{l:Wsigma} we have not discussed in any detail the
orientations of moduli spaces of discs, which determine the signs of 
the various terms appearing in the superpotential. The fact that those
are all positive follows from two ingredients.

The first is that, for a standard product torus in a toric variety, equipped with the standard
spin structure, the contributions of the various families of Maslov index~2
holomorphic discs to the superpotential are all positive. See \cite{Cho-spin}
for a detailed calculation in the case of the Clifford torus.
The fact that all the signs are the same is not surprising, since
a monomial change of variables can be used to reduce to a single example,
namely the family of discs $D^2\times \{pt\}$ bounded by a product torus 
in $\C\times (\C^*)^n$ equipped with the standard spin structure.
The same argument also applies to the discs in Lemma \ref{l:countw0} since
those can also be reduced to the toric case.

The second ingredient is a comparison of the orientations of moduli spaces
of discs in $V$ and their lifts to $X$ (as in Lemma \ref{l:Wsigma}). 
A short calculation shows that, for the standard spin structure, the orientation of the moduli 
space of lifted discs in $X$ agrees with that induced by the orientation of
the moduli space of discs in $V$ and the natural orientation of the orbits
of the $S^1$-action. See the proof of Corollary 8 in \cite{AuR6} for a
similar argument. The positivity of the signs in Lemma \ref{l:Wsigma}
follows.
\end{remark}

\begin{remark}\label{rmk:borderline}
When Assumption \ref{ass:affinecase} does not hold, the
SYZ mirror of $X$ differs from $(Y^0,W_0)$, since the enumerative geometry
of discs is modified by the presence of stable genus 0
configurations of total Maslov index 0 or 2.
A borderline case that remains fairly easy is when the strict inequality in
Assumption \ref{ass:affinecase} is relaxed to
$$c_1(V)\cdot C\ge \max(0,H\cdot C).$$
(This includes the situation where $H$ is a Calabi-Yau hypersurface in a toric
Fano variety as an important special case.)

In this case, Assumption \ref{ass:nef} still holds, so the mirror of $X$
remains $Y^0$; the only modification is that the superpotential should
also count the contributions of configurations consisting of a Maslov index
2 disc together with one or more rational curves satisfying $c_1(X)\cdot
C=0$. Thus, we now have $$W=(1+c_0)w_0+(1+c_1)w_1+\dots+(1+c_r)w_r,$$
where $c_0,\dots,c_r\in \Lambda$ are constants (determined by the genus 0
Gromov-Witten theory of $X$), with $\mathrm{val}(c_i)>0$.
\end{remark}



\section{From the blowup $X$ to the hypersurface $H$} \label{s:backtoH}

The goal of this section is to prove Theorem \ref{cor:main}. As a first step, we establish:

\begin{theorem}\label{thm:FScompact}
Under Assumption \ref{ass:affinecase}, the $B$-side Landau-Ginzburg model 
$(Y,W_0)$ is SYZ mirror to the $A$-side Landau-Ginzburg
model $(X, W^\vee=y)$ $($with the K\"ahler form $\omega_\epsilon)$.
\end{theorem}
\noindent
(Recall that $y$ is the coordinate on the second factor of $V\times\C$.)

\proof[Sketch of proof]
This result follows from Theorem \ref{thm:main} by the same considerations
as in Example \ref{ex:FScompact}. Specifically, equipping $X$ with the
superpotential $W^\vee=y$ enlarges its Fukaya category by adding
admissible non-compact Lagrangian submanifolds, i.e., properly
embedded Lagrangian submanifolds of $X$ whose image under $W^\vee$ is
only allowed to tend to infinity in the direction of the positive real
axis; in other terms, the $y$ coordinate is allowed to be unbounded, but
only in the positive real direction.  

Let $a_0\subset \C$ be a properly
embedded arc which connects $+\infty$ to itself by passing around the
origin, encloses an infinite amount of area, and stays away from
the projection to $\C$ of the support of the cut-off function $\chi$ used
to construct $\omega_\epsilon$.
Then we can supplement the family of Lagrangian tori in $X^0$
constructed in \S \ref{s:fibrations} by considering product Lagrangians
of the form $L=p^{-1}(L'\times a_0)$, where $L'$ is
an orbit of the $T^n$-action on $V$. Indeed, by Proposition
\ref{prop:pistandard}, away from the exceptional divisor 
the fibers of $\pi:X^0\to B$ are lifts to $X$ of
product tori $L'\times S^1(r)\subset V\times\C$. For large enough $r$,
the circles $S^1(r)$ can be deformed by Hamiltonian isotopies in $\C$ to
simple closed curves that approximate $a_0$ as $r\to \infty$; moreover,
the induced isotopies preserve the tautological unobstructedness in $X^0$ of the
fibers of $\pi$ which do not intersect $p^{-1}(H\times\C)$. In this sense,
$p^{-1}(L'\times a_0)$ is naturally a limit of the tori
$p^{-1}(L'\times S^1(r))$ as $r\to \infty$.  The analytic structure near this point
is obtained by equation (\ref{eq:localchart}), which naturally extends 
as in Example \ref{ex:FScompact}.

To be more specific, let $L'=\mu_V^{-1}(\zeta_1,\dots,\zeta_n)$ for $(\zeta_1,\dots,\zeta_n)$
a point in the component of $\Delta_V\setminus \mu_V(U_H)$ corresponding
to the weight $\alpha\in A$, and equip $L=p^{-1}(L'\times a_0)$ with a
local system $\nabla\in \hom(\pi_1(L),U_{\K})$. The maximum principle
implies that any holomorphic disc bounded by $L$ in $X^0$ must 
be contained inside a fiber of the projection to $V$ (see the proof of 
Proposition \ref{prop:unobstructed}). Thus $L$ is tautologically unobstructed in $X^0$,
and $(L,\nabla)$ defines an object of the Fukaya category $\F(X^0,W^\vee)$,
and a point in some partial compactification of the coordinate chart
$U_\alpha^\vee$ considered in \S \ref{s:opencase}. 
Denoting by $\gamma_1,\dots,\gamma_n$ the standard basis of $H_1(L)\simeq
H_1(L')$ given by the various $S^1$ factors, in the coordinate chart
\eqref{eq:chartalpha} the object $(L,\nabla)$ corresponds to
$$(v_{\alpha,1},\dots,v_{\alpha,n},w_{\alpha,0})=
\left(T^{\zeta_1}\nabla(\gamma_1),\dots,T^{\zeta_n}\nabla(\gamma_n),0\right).$$
Thus, equipping $X^0$ with the superpotential $W^\vee$ extends
the moduli space of objects under consideration from
$Y^0=Y\setminus w_0^{-1}(0)$ to $Y$.

Under Assumption \ref{ass:affinecase}, $(L,\nabla)$ remains a weakly unobstructed
object of the Fukaya category $\F(X,W^\vee)$. We now study the families of
Maslov index 2 holomorphic discs bounded by $L$ in $X$, in order to
determine the corresponding value of the superpotential and show that it
agrees with \eqref{eq:W0}. Under projection to the $y$ coordinate, any
holomorphic disc $u:(D^2,\partial D^2)\to (X,L)$ maps to a holomorphic
disc in $\C$ with boundary on the arc $a_0$, which is necessarily constant;
hence the image of $u$ is contained inside $p^{-1}(V\times \{y\})$ for some
$y\in a_0$. Moreover, inside the toric variety $p^{-1}(V\times \{y\})\simeq V$ 
the holomorphic disc $u$ has boundary
on the product torus~$L'$. 

Thus, the holomorphic discs bounded by $L$
in $X$ can be determined by reduction to the toric case of $(V,L')$.
For each toric divisor of $V$ there is a family of Maslov index 2
discs which intersect it transversely at a single
point and are disjoint from all the other toric divisors;
these discs are all regular, and exactly one of them passes
through each point of $L$ \cite{cho-oh}. The discs which intersect the toric
divisor corresponding to a facet of $\Delta_V$ with equation
$\langle \sigma,\cdot\rangle+\varpi=0$ have area 
$\langle \sigma,\zeta\rangle+\varpi$ and weight $T^\varpi
\mathbf{v}_\alpha^\sigma$. Summing over all facets of $\Delta_V$, we
conclude that
\begin{equation}\label{eq:WFS}
W_0(L,\nabla)=\sum_{i=1}^r T^{\varpi_i}\mathbf{v}_\alpha^{\sigma_i}.
\end{equation}
Moreover, because $w_0=0$ at the point $(L,\nabla)$, the coordinate
transformations \eqref{eq:gluing} simplify to
$\mathbf{v}_{\alpha_i}^{\sigma_i}=\mathbf{v}_{\alpha}^{\sigma_i}$.
Thus the expression \eqref{eq:WFS} agrees with \eqref{eq:W0}.
\endproof
\begin{remark}
In order to fill the details of this sketch, we would need a sufficient development of Fukaya categories of $A$-side Landau-Ginzburg models in order to verify the existence of the analytic charts at infnity. The most straightforward way to do this is to introduce non-compact Lagrangians which are mirror to the powers of an ample line bundle on $Y$, and check that (i) these Lagrangians generate the Fukaya category and (ii) when $r$ is sufficiently large, the product Lagrangian $L'\times S^1(r)\subset V\times\C$ defines a module over the Floer cochains of this generating family  which is equivalent to the one associated to the product of $L'$ with an admissible arc in $\C$ equipped with a bounding cochain which is a multiple of a degree $1$ generator coming from a self-intersection at infinity.
\end{remark}

Our next observation is that $W^\vee:X\to\C$ has a particularly simple
structure. The following statement is a direct consequence of the
construction:

\begin{proposition}\label{prop:morsebott}
$W^\vee=y:X\to\C$ is a Morse-Bott fibration, with $0$ as its only
critical value; in fact the singular fiber ${W^\vee}^{-1}(0)=\tilde{V}\cup E\subset X$
has normal crossing singularities along $\mathrm{crit}(W^\vee)=
\tilde{V}\cap E\simeq H$.
\end{proposition}

\begin{remark}\label{rmk:deflate}
However, the K\"ahler form on $\mathrm{crit}(W^\vee)\simeq H$ is not that induced by
$\omega_V$, but rather that induced by the restriction of $\omega_\epsilon$,
which represents the cohomology class $[\omega_V]-\epsilon[H]$.
To compensate for this, in the proof of Theorem \ref{cor:main} we will actually replace
$[\omega_V]$ by $[\omega_V]+\epsilon[H]$.
\end{remark}

\noindent
Proposition \ref{prop:morsebott} allows us to relate the Fukaya category of $(X,W^\vee)$ to that of $H$,
using the ideas developed by Seidel in \cite{SeBook}, adapted to the
Morse-Bott case (see \cite{WW}).

\begin{remark}
Strictly speaking, the literature does not include any definition of the Fukaya category of a superpotential without assuming that it is a Lefschetz fibration.   The difficulty resides not in defining the morphisms and the compositions, but in defining the higher order products in a coherent way. These technical problems were resolved by Seidel in \cite{SeLef}, by introducing a method of defining Fukaya categories of Lefschetz fibration that generalizes in a straightforward way to the Morse-Bott case we are considering. This construction will be revisited in \cite{A-Se}.  As the reader will see, in the only example where we shall study such a Fukaya category, the precise nature of the construction of higher products will not enter.
\end{remark}

Outside of its critical locus, the Morse-Bott fibration $W^\vee$ carries 
a natural horizontal distribution given by the $\omega_\epsilon$-orthogonal
to the fiber. Parallel transport with respect to this distribution induces
symplectomorphisms between the smooth fibers; in fact, parallel transport
along the real direction is given by (a rescaling of) the Hamiltonian 
flow generated by $\mathrm{Im}\,W^\vee$, or equivalently, the gradient flow
of $\mathrm{Re}\,W^\vee$ (for the K\"ahler metric).

Given a Lagrangian submanifold $\ell\subset \mathrm{crit}(W^\vee)\simeq H$, 
parallel transport by the positive gradient flow of $\mathrm{Re}\,W^\vee$ 
yields an admissible Lagrangian {\em thimble} $L_\ell\subset X$
(topologically a disc bundle over~$\ell$). Moreover, any local system
$\nabla$ on $\ell$ induces by pullback a local system $\tilde\nabla$ on
$L_\ell$. However, there is a subtlety related to the nontriviality of
the normal bundle to $H$ inside $X$:

\begin{lemma}\label{l:twisting}
The thimble $L_\ell$ is naturally diffeomorphic to the restriction of
the complex line bundle $\mathcal{L}=\O(H)$ to $\ell\subset H$.
\end{lemma}

\proof
First note that, for the Lefschetz fibration $f(x,y)=xy$ on $\C^2$
equipped with its standard K\"ahler form, the thimble associated
to the critical point at the origin is $\{(x,\bar{x}),\ x\in\C\}\subset \C^2$.
Indeed, parallel transport preserves
the quantity $|x|^2-|y|^2$, so that the thimble consists of the points
$(x,y)$ where $|x|=|y|$ and $xy\in \R_{\ge 0}$, i.e.\ $y=\bar{x}$.
In particular, the thimble projects diffeomorphically onto either of the two
$\C$ factors (the two projections induce opposite orientations).

Now we consider the Morse-Bott fibration $W^\vee:X\to\C$. The normal bundle
to the critical locus $\mathrm{crit}\,W^\vee=\tilde{V}\cap E\simeq H$ is isomorphic 
to $\mathcal{L}\oplus \mathcal{L}^{-1}$ (where
$\mathcal{L}$ is the normal bundle to $H$ inside $\tilde{V}$, while 
$\mathcal{L}^{-1}$ is its normal bundle inside $E$). Moreover,
$W^\vee$ is locally given by the 
product of the fiber coordinates on the two line subbundles. 
The local calculation then shows that, by projecting to either subbundle,
a neighborhood of $\ell$ in $L_\ell$ can be identified diffeomorphically
with a neighborhood of the zero section in either $\mathcal{L}_{|\ell}$
or $\mathcal{L}^{-1}_{|\ell}$.
\endproof

\noindent
Lemma \ref{l:twisting} implies that, even when $\ell\subset H$ is spin, $L_\ell\subset X$ need 
not be spin; indeed, $w_2(TL_\ell)=w_2(T\ell)+w_2(\mathcal{L}_{|\ell})$.
Rather, $L_\ell$ is {\em relatively spin}, i.e.\ its second Stiefel-Whitney
class is the restriction of the {\em background class} $s\in H^2(X,\Z/2)$
Poincar\'e dual to $[\tilde{V}]$ (or equivalently to $[E]$).
Hence, applying the thimble construction to an object of the Fukaya 
category $\F(H)$ does not determine an object
of $\F(X,W^\vee)$, but rather an object of the $s$-twisted Fukaya category
$\F_s(X,W^\vee)$ (we shall verify in Proposition \ref{prop:m0shift} that thimbles are
indeed weakly unobstructed objects of this category). 

\begin{remark}\label{rmk:unobstr_thimbles}
While it has not appeared in the literature, the notion of weak 
unobstructedness of an admissible Lagrangian $L$ is a straightforward 
generalization of the case of closed Lagrangians. 
There is a Floer-theoretic $A_\infty$-structure on the ordinary 
cohomology of $L$, and a natural $A_\infty$-homomorphism from the
ordinary cohomology of $L$ equipped with this $A_\infty$-structure
to the endomorphisms of $L$ as an object of the Fukaya category 
of the potential. This homomorphism is not necessarily an isomorphism,
but it is always unital and preserves the curvature $\m_0$. 
We say that $L$ is \emph{weakly unobstructed} if the curvature 
is a multiple of the unit in $H^0(L)$. 
In the case of thimbles, radial parallel transport allows one to 
lift Maurer-Cartan elements and bounding cochains from an arbitrarily 
small neighborhood of the critical fiber to the total space. This implies
that an admissible thimble which bounds no
holomorphic disc of Maslov index less than $2$ in a neighborhood of
the critical fiber is weakly unobstructed; and the curvature is then the 
product of the unit with the count of Maslov index $2$ discs passing 
through a generic point near the critical fiber.
\end{remark}


\begin{corollary}\label{cor:functor}
Under Assumption \ref{ass:affinecase}, there is a fully faithful $A_\infty$-functor from the Fukaya category $\F(H)$ to
$\F_s(X,W^\vee)$, which at the level of objects maps $(\ell,\nabla)$ to the 
thimble $(L_\ell,\tilde\nabla)$.
\end{corollary}

\proof[Sketch of proof] Let $\ell_1,\ell_2$ be two Lagrangian submanifolds of
$\mathrm{crit}(W^\vee)\simeq H$,
assumed to intersect transversely (otherwise transversality is achieved
by  Hamiltonian perturbations, which may
be needed to achieve regularity of holomorphic discs in any case),
and denote by $L_1,L_2\subset X$ the corresponding thimbles. (For simplicity
we drop the local systems from the notations; we also postpone the
discussion of relatively spin structures until further below). 

Recall that $\hom_{\F_s(X,W^\vee)}(L_1,L_2)$ is defined by
perturbing $L_1,L_2$ to Lagrangians $\tilde{L}_1,\tilde{L}_2$ whose images
under $W^\vee$ are half-lines which intersect transversely and such that
the first one lies above the second one near infinity; so for example,
fixing a small angle $\theta>0$,
we can take $\tilde{L}_1$ (resp.~$\tilde{L}_2$) to be the Lagrangian 
obtained from $\ell_1$ (resp.~$\ell_2$) by the gradient flow of 
$\mathrm{Re}(e^{-i\theta}\,W^\vee)$
(resp.\ $\mathrm{Re}(e^{i\theta}\,W^\vee)$).
(A more general approach would be to perturb the holomorphic curve
equation by a Hamiltonian vector field generated by a suitable rescaling
of the real part of $W^\vee$, instead of perturbing the Lagrangian boundary 
conditions; in our case
the two approaches are equivalent.)

We now observe that $\tilde{L}_1$ and $\tilde{L}_2$ intersect transversely,
with all intersections lying in the singular fiber ${W^\vee}^{-1}(0)$,
and in fact $\tilde{L}_1\cap \tilde{L}_2=\ell_1\cap \ell_2$.
Thus, $\hom_{\F(H)}(\ell_1,\ell_2)$ and $\hom_{\F_s(X,W^\vee)}(L_1,L_2)$ 
are naturally isomorphic. Moreover, the maximum principle applied to the
projection $W^\vee$ implies that all holomorphic discs bounded by the
(perturbed) thimbles in $X$ are contained in $(W^\vee)^{-1}(0)=\tilde{V}
\cup E$ (and hence their boundary lies on $\ell_1\cup \ell_2\subset H
\subset \tilde{V}\cup E$).

After quotienting by a suitable reference section, we can view the defining 
section of $H$ as a meromorphic function on $\tilde{V}$, with $f^{-1}(0)=H$.
Since $f=0$ at the boundary, and since a meromorphic function on the disc
which vanishes at the boundary is everywhere zero, any holomorphic disc 
in $\tilde{V}$ with boundary in $\ell_1\cup \ell_2$ must lie entirely
inside $f^{-1}(0)=H$. By the same argument, any holomorphic disc in $E$ with
boundary in $\ell_1\cup \ell_2$ must stay inside $H$ as well. Finally, Lemma
\ref{l:fano} implies that stable curves with both disc and sphere components
cannot contribute to the Floer differential (since each sphere component 
contributes at least 2 to the total Maslov index).

This implies that the Floer differentials on $\hom_{\F(H)}(\ell_1,\ell_2)$
and $\hom_{\F_s(X,W^\vee)}(L_1,L_2)$ count the same holomorphic discs.
The same argument applies to Floer products and higher structure maps.

To complete the proof it only remains to check that the orientations 
of the relevant moduli spaces of discs agree.
Recall that a relatively spin structure on a Lagrangian submanifold
$L$ with background class $s$ is the same thing as a stable trivialization of the
tangent bundle of $L$ over its 2-skeleton, i.e.\ a trivialization of
$TL_{|L^{(2)}}\oplus E_{|L^{(2)}}$, where $E$ is a vector bundle over the ambient manifold
with $w_2(E)=s$; such a stable trivialization in turn determines
orientations of the moduli spaces of holomorphic discs with boundary on $L$
(see \cite[Chapter 8]{FO3book}, noting that the definition of spin structures
in terms of stable trivializations goes back to Milnor \cite{Milnor}).

In our case, we are considering discs in
$H$ with boundary on Lagrangian submanifolds $\ell_i\subset H$, and the given spin structures on
$\ell_i$ determine orientations of the moduli spaces for the
structure maps in $\F(H)$. If we consider the same holomorphic discs in
the context of the thimbles $L_i\subset X$, the spin structure of $\ell_i$
does not induce a spin structure on
$TL_i\simeq T\ell_i\oplus \mathcal{L}_{|\ell_i}$ (what would be needed
instead is
a relatively spin structure on $\ell_i$ with background class
$w_2(\mathcal{L}_{|H})$). On the other hand, the normal bundle to $H$ inside
$X$, namely $\mathcal{L}\oplus \mathcal{L}^{-1}$, is an $SU(2)$-bundle and
hence has a canonical isotopy class of trivialization over the 2-skeleton.
Thus, the spin structure on $\ell_i$ induces a trivialization of
$TL_i\oplus \mathcal{L}^{-1}$ over the 2-skeleton of $L_i$, i.e.\ a relative
spin structure on $L_i$ with background class $w_2(\mathcal{L}^{-1}_{|L_i})=
s_{|L_i}$. Furthermore, because $w_2(\mathcal{L}\oplus \mathcal{L}^{-1})=0$,
stabilizing by this rank 2 bundle does not affect the orientation of the
moduli space of discs \cite[Proposition 8.1.16]{FO3book}. Hence the
structure maps of $\F(H)$ and $\F_s(X,W^\vee)$ involve the same moduli
spaces of holomorphic discs, oriented in the same manner, which completes
the proof.
\endproof
\begin{remark}
 The reason the above is only a sketch of proof is that the construction of the two Fukaya categories requires choices of perturbations,  and we have not discussed how to arrange for these choices to yield the same answer. A model for such arguments in a related situation is provided by Seidel in \cite[Section (14c)]{SeBook}. 
\end{remark}

Implicit in the statement of Corollary \ref{cor:functor} is the fact that,
if $(\ell,\nabla)$ is weakly unobstructed in $\F(H)$, then $(L_\ell,
\tilde\nabla)$ is weakly unobstructed in $\F_s(X,W^\vee)$. In our setting, the values of the superpotentials for objects of $\F(H)$ and their images in $\F_s(X, W^\vee)$ differ
by an additive constant $\delta$. This constant is easiest to determine if we assume
that $V$ is affine:

\begin{proposition}\label{prop:m0shift}
Under the assumption that $V$ is affine, 
the functor of Corollary~\ref{cor:functor} increases the value of 
the superpotential by 
$\delta=T^\epsilon$.
\end{proposition}

\proof[Sketch of proof]
Consider a weakly unobstructed object $(\ell,\nabla)$ of $\F(H)$ and
the corresponding thimble $L_\ell\subset X$.
Holomorphic discs bounded by $L_\ell$ in $X$
are contained in the level sets of $W^\vee=y$ (by the
maximum principle).  By Remark \ref{rmk:unobstr_thimbles}, we only need
to study the moduli spaces of such discs for small values of $y$.

For $y>0$, the intersection $L_\ell^y$ of $L_\ell$ with
$(W^\vee)^{-1}(y)\simeq V$ is a circle bundle over $\ell$, lying in the
boundary of a standard symplectic tubular neighborhood of size $\epsilon$
of $H$ in $(W^\vee)^{-1}(y)$ equipped with the restriction of $\omega_\epsilon$. 
Indeed, as $y\to 0$, the fibers of $W^\vee$
degenerate to the normal crossing divisor $\tilde{V}\cup E$. Symplectic parallel
transport identifies the standard disc bundle
$E\setminus (\tilde{V}\cap E)\simeq H\times D^2(\epsilon)$ inside
$(W^\vee)^{-1}(0)$ with a standard symplectic neighborhood $U^y$ of $H$ inside
$(W^\vee)^{-1}(y)$ for $y>0$. The boundary of $U^y$ (a trivial $S^1$-bundle
over $H$) consists of all points in $(W^\vee)^{-1}(y)$ whose parallel 
transport converges to $\tilde{V}\cap E\simeq H$ as $y\to 0$, and in particular it contains
$L_\ell^y$.

However, while the restriction of $\omega_\epsilon$ to
$(W^\vee)^{-1}(y)\simeq V$ is cohomologous to $\omega_V$ for all $y>0$ and
agrees with it pointwise for $y$ sufficiently large, the actual forms differ
near $H$ for small $y$. 
Under the identification $(W^\vee)^{-1}(y)\simeq V$,
the neighborhoods $U^y$ are small tubular neighborhoods of $H$,
%
%
increasing in size along a suitably normalized 
gradient flow of $|f|$ as $y$ increases,
and agreeing with a standard $\omega_V$-neighborhood of $H$ of size
$\epsilon$
for $y\gg \epsilon^{1/2}$.

Using that $V$ is affine, $H$ is the vanishing locus of the globally defined holomorphic function
$f$, and the maximum principle applied to $f$ implies that, for small enough
$y$ (or for all $y$ if $\epsilon$ is small enough), all holomorphic discs bounded by $L_\ell^y$
in $V$ lie in a neighborhood ${U'}^y$ of $H$ (possibly larger than $U^y$). 

The complex structure on the neighborhood ${U'}^y$
of $H$ in $V$ is not biholomorphic to the standard product complex structure 
on a domain in $H\times\C$, but agrees with it along
$H$. Thus, for small enough $y$, an arbitrarily $C^\infty$-small perturbation of the 
almost-complex structure on $V$ 
(preserving the holomorphicity of $f$) ensures the existence of a holomorphic
projection map $\pi_H:{U'}^y\to H$, without affecting counts of holomorphic
discs; without loss of generality, we can further
assume that $\pi_H$ maps $L_\ell^y$ to $\ell$ as an $S^1$-bundle, with
$|f|$ constant in the $S^1$ fiber over each point of $\ell$.

Holomorphic discs with boundary on $L^y_\ell$ can then be classified by
using the projection to $H$. The Maslov index of a disc $u:D^2\to
(V,L_\ell^y)$ (with image contained in ${U'}^y$) is the sum of 
the Maslov index of $\pi_H\circ u$ and twice the intersection number of $u$
with $H$. Thus, the weak unobstructedness of $\ell$ in $H$ implies that of
$L_\ell^y$, and there are two types of Maslov index 2 discs to consider:
\begin{itemize}
\item $\pi_H \circ u$ is a Maslov index 2 disc in $H$, and $u$ avoids $H$;
\item $\pi_H \circ u$ is constant, and $u$ intersects $H$ transversely once.
\end{itemize}

In the first case, we observe that, given a point $\hat{p}\in L_\ell^y$, 
each holomorphic disc $v:D^2\to (H,\ell)$ through $p=\pi_H(\hat{p})$ has a 
unique lift $u$ through $\hat{p}$ that avoids $H$.
Indeed, $v$ determines the value of $\log |f|$ along
the boundary of the disc $u$; the (unique) harmonic extension of this function to the
entire disc can be expressed as the real part of some holomorphic function $g$,
unique up to a pure imaginary additive constant. We then find
that necessarily $f\circ u=\exp(g)$ up to some constant
factor which is determined by requiring that the marked point map to 
$\hat{p}$. This, together with $\pi_H\circ u=v$, determines $u$.
Recalling that $L_\ell^y$ lives on the boundary of a standard symplectic
neighborhood of $H$, and using that $u$ is disjoint from $H$, we further observe that the symplectic area of $u$ in
$(W^\vee)^{-1}(y)$ is equal to that of $v$ in $H$, and the holonomy of
$\tilde\nabla$ along the boundary of $u$ equals that of $\nabla$ along the
boundary of $v$. Moreover, the same argument as in the proof of Corollary
\ref{cor:functor} shows that the orientations of the moduli spaces match.
Thus, the total contribution of all these discs
corresponds exactly to the superpotential in $\F(H)$.

In the second case, denoting $\pi_H\circ u=p\in \ell$, by construction
$L_\ell^y$ intersects $\pi_H^{-1}(p)$ in a circle which bounds a disc of
symplectic area $\epsilon$, and $u$ necessarily maps $D^2$ biholomorphically
onto this disc. These small discs of size $\epsilon$ in the normal slices 
to $H$ are regular, and contribute positively to the superpotential: indeed,
their deformation theory splits into that of constant discs in $H$ 
and that of the standard disc in the complex plane with boundary 
on a circle with the trivial spin structure (the triviality of the spin
structure is due to the twist by the background class $s$).
Thus, these discs are responsible 
for the additional term $T^\epsilon$ in the superpotential for $L_\ell$.

For the sake of completeness, we also consider the case $y=0$, where the
intersection of $L_\ell$ with $(W^\vee)^{-1}(0)=\tilde{V}\cup E$ is 
simply $\ell$. The argument in the proof of Corollary \ref{cor:functor}
then shows that holomorphic discs bounded by $\ell$ in $\tilde{V}\cup E$
lie entirely within $H$; however, there is a nontrivial contribution of
Maslov index 2 configurations consisting of a constant disc together with
a rational curve contained in $E$, namely the $\PP^1$ fiber 
of the exceptional divisor over a point of $\ell\subset H$. (These
exceptional spheres are actually the limits of the area $\epsilon$ discs
discussed above as $y\to 0$).
\endproof

\begin{remark}\label{rmk:m0shift}
The assumption that $V$ is affine can be weakened somewhat: for
Proposition \ref{prop:m0shift} to hold it is sufficient to assume that 
the minimal Chern number of
a rational curve contained in $\tilde{V}$ is at least 2. When this
assumption does not hold, the discrepancy $\delta$ between the two
superpotentials includes additional contributions from the enumerative
geometry of rational curves of Chern number 1 in $\tilde{V}$.
\end{remark}

\begin{remark}
The $A_\infty$-functor from $\F(H)$ to $\F_s(X,W^\vee)$ is induced by
a Lagrangian correspondence in the product $H\times X$, namely the set
of all $(p,q)\in H\times X$ such that parallel transport of
$q$ by the gradient flow of $-\mathrm{Re}\,W^\vee$ converges to $p\in
\mathrm{crit}\,W^\vee$. This Lagrangian correspondence is admissible
with respect to $\mathrm{pr}_2^* W^\vee$, and weakly unobstructed with 
$\m_0=\delta$. While the Ma'u-Wehrheim-Woodward construction of $A_\infty$-functors
from Lagrangian correspondences \cite{MWW} has not yet been developed in the
setting considered here, it is certainly the right conceptual framework
in which Corollary \ref{cor:functor} should be understood.
\end{remark}

By analogy with the case of Lefschetz fibrations \cite{SeBook}, it is
expected that the Fukaya category of a Morse-Bott fibration is generated
by thimbles, at least under the assumption that the Fukaya category of
the critical locus admits a resolution of the diagonal. The argument is
expected to be similar to that in \cite{SeBook}, except in the Morse-Bott
case the key ingredient becomes the long
exact sequence for fibered Dehn twists \cite{WW}.
Thus, it is reasonable to expect that the $A_\infty$-functor of
Corollary \ref{cor:functor} is in fact
a quasi-equivalence.

Similar statements are also expected to hold for the {\em wrapped Fukaya
category} of $H$ and the {\em partially wrapped Fukaya category} of
$(X,W^\vee)$ (twisted by $s$); however, this remains speculative, as the latter category 
has not been suitably constructed yet.


In any case, Corollary \ref{cor:functor} and Proposition \ref{prop:m0shift}
motivate the terminology
introduced in Definition \ref{def:SYZ-mirror}.

\proof[Proof of Theorem \ref{cor:main}]
While Theorem \ref{thm:FScompact} provides an SYZ mirror to the
Landau-Ginzburg model $(X,W^\vee)$, in light of the above discussion
several adjustments are necessary in order to arrive at a generalized SYZ
mirror to $H$.
\begin{enumerate}
\item As noted in Remark \ref{rmk:deflate}, the restriction of
$\omega_\epsilon$ to $\mathrm{crit}(W^\vee)$ does not agree with the restriction of $\omega_V$ to $H$.
To remedy this, in our main construction $V$ should be equipped with a K\"ahler
form in the class $[\omega_V]+\epsilon [H]$ rather than $[\omega_V]$.
This ensures that the critical locus of $W^\vee$ is indeed isomorphic to
$H$ equipped with the restriction of the K\"ahler form $\omega_V$.
\item In light of Corollary \ref{cor:functor}, the A-side Landau-Ginzburg model
$(X,W^\vee)$ should be twisted by the background class $s=PD([\tilde{V}])\in 
H^2(X,\Z/2)$. Namely, the tori we consider in our main argument should be
viewed as objects of $\F_s(X,W^\vee)$ rather than $\F(X,W^\vee)$. This
modifies the sign conventions for counting discs and hence 
the mirror superpotential.
\item By Proposition \ref{prop:m0shift}, the additive constant
$\delta=T^\epsilon$ should be subtracted from the superpotential, since
the natural $A_\infty$-functor from $\F(H)$ to $\F_s(X,W^\vee)$ increases
$\m_0$ by that amount.
\end{enumerate}
Thus, the mirror space remains the toric variety $Y$, but the
superpotential is no longer
\begin{equation}\label{eq:W0again}
W_0=w_0+\sum_{i=1}^r T^{\varpi_i}
\mathbf{v}_{\alpha_i}^{\sigma_i};
\end{equation}
we now make explicit how each of the above changes affects the potential.

Replacing $[\omega_V]$ by $[\omega_V]+\epsilon[H]$ amounts to changing the
equations of the facets of the moment polytope $\Delta_V$ from
$\langle \sigma_i,\cdot\rangle+\varpi_i=0$ to
$\langle \sigma_i,\cdot\rangle+\varpi_i+\epsilon \lambda(\sigma_i)=0$
(where $\lambda:\Sigma_V\to\R$ is the piecewise linear function defining
$\mathcal{L}=\O(H)$). Accordingly, each exponent $\varpi_i$ in \eqref{eq:W0again}
should be changed to $\varpi_i+\epsilon\lambda(\sigma_i)$.

Next, we twist by the background class $s=PD([\tilde{V}])$, and view the
tori studied in Section \ref{s:opencase} as objects of $\F_s(X,W^\vee)$
rather than $\F(X,W^\vee)$. Specifically, $s$ lifts to a class
in $H^2(X,L;\Z/2)$ (dual to $[\tilde{V}]\in H_{2n}(X\setminus L)$),
and twisting the standard
spin structure by this lift of $s$ yields a relatively spin structure on $L$.
By \cite[Proposition 8.1.16]{FO3book}, this twist affects the signed
count of holomorphic discs in a given class $\beta\in \pi_2(X,L)$ by a
factor of $(-1)^k$ where $k=\beta\cdot [\tilde{V}]$. Recall
from \S \ref{s:generalcase} that, of the various families of holomorphic
discs that contribute to the superpotential, the only ones
that intersect $\tilde{V}$ are those described by Lemma \ref{l:countw0};
thus the only effect of the twisting by the background class $s$ is to
change the first term of $W_0$ from $w_0$ to $-w_0$.

Finally, we subtract $\delta=T^\epsilon$ from the superpotential, and find that 
the appropriate superpotential to consider on $Y$ is given by
$$W'_0=-T^\epsilon-w_0+\sum_{i=1}^r T^{\varpi_i+\epsilon\lambda(\sigma_i)}
\mathbf{v}_{\alpha_i}^{\sigma_i}=-T^\epsilon v_0+\sum_{i=1}^r
T^{\varpi_i} T^{\epsilon\lambda(\sigma_i)} \mathbf{v}_{\alpha_i}^{\sigma_i}.$$
Finally, recall from \S \ref{ss:mirror} that the weights of the toric monomials
$v_0$ and $\mathbf{v}_{\alpha_i}^{\sigma_i}$ are respectively
$(0,1)$ and $(-\sigma_i,\lambda(\sigma_i))\in \Z^n\oplus \Z$. Therefore, 
a rescaling of the last coordinate  by a factor of $T^\epsilon$
changes $v_0$ to $T^\epsilon v_0$ and $\mathbf{v}_{\alpha_i}^{\sigma_i}$ to
$T^{\epsilon\lambda(\sigma_i)}\mathbf{v}_{\alpha_i}^{\sigma_i}$.
This change of variables eliminates the dependence on $\epsilon$ (as one
would expect for the mirror to $H$) and replaces $W'_0$ by the simpler expression
$$-v_0+\sum_{i=1}^r T^{\varpi_i} \mathbf{v}_{\alpha_i}^{\sigma_i},$$
which is exactly $W_0^H$ (see Definition \ref{def:LGmirror}).
\endproof

\begin{remark}
Another way to produce an $A_\infty$-functor from the Fukaya category 
of $H$ to that of $X$ (more specifically, the idempotent closure
of $\F_s(X)$) is the following construction considered by Ivan Smith in 
\cite[Section 4.5]{Smith}. 

Given a Lagrangian submanifold $\ell\subset H$,
first lift it to the boundary of the $\epsilon$-tubular neighborhood of
$H$ inside $V$, to obtain a Lagrangian submanifold $C_\ell\subset V$ 
which is a circle bundle 
over $\ell$; then, identifying $V$ with the reduced space
$X_{red,\epsilon}=\mu_X^{-1}(\epsilon)/S^1$, lift $C_\ell$ to
$\mu_X^{-1}(\epsilon)$ and ``spin'' it by the $S^1$-action, to obtain
a Lagrangian submanifold $T_\ell\subset X$ which is a $T^2$-bundle over
$\ell$. Then $T_\ell$ formally splits into a direct sum $T_\ell^+\oplus T_\ell^-$;
the $A_\infty$-functor is constructed by mapping $\ell$ to either summand.

The two constructions are equivalent: in $\F_s(X,W^\vee)$ 
the summands $T_\ell^\pm$
are isomorphic to the thimble $L_\ell$ (up to a shift).
One benefit of Smith's construction 
is that, unlike $L_\ell$, the Lagrangian submanifold $T_\ell$ is entirely
contained inside $X^0$, which makes its further study amenable to
$T$-duality arguments involving $X^0$ and $Y^0$.
\end{remark}

\section{The converse construction} \label{s:converse}

As a consequence of Theorem \ref{thm:conicbundle}, the mirror $Y^0$ of $X^0$
can be defined as a variety not only over the Novikov field, but also over the
complex numbers.
In this section, we impose the maximal degeneration condition (cf. Definition \ref{def:maxdegeneration}) which implies that $Y^0$ is smooth. 
We then reverse our viewpoint from the preceding discussion:
treating $T$ as a numerical parameter and equipping $Y^0$ with a K\"ahler form, we shall reconstruct $X^0$ (as an analytic space that also happens to be defined over
complex numbers) as an SYZ mirror.
Along the way, we also obtain another perspective on how compactifying $Y^0$ to the toric variety $Y$ amounts to equipping
$X^0$ with a superpotential. We omit any discussion of $Y$ or $Y^0$ equipped with $A$-side Landau-Ginzburg models, which would require a deeper understanding of the corresponding Fukaya categories. 

(Note: many of the results in this section were
also independently obtained by Chan, Lau and Leung \cite{ChanLauLeung}.)

To begin our construction, observe that $Y^0=Y\setminus w_0^{-1}(0)$ carries
a natural $T^n$-action,
given in the coordinates introduced in \S \ref{ss:mirror} by
$$(e^{i\theta_1},\dots,e^{i\theta_n})\cdot (v_{\alpha,1},\dots,v_{\alpha,n},
v_{\alpha,0})=(e^{i\theta_1}v_{\alpha,1},\dots,e^{i\theta_n}v_{\alpha,n},
v_{\alpha,0}).$$ This torus is a subgroup of the $(n+1)$-dimensional torus
which acts on the toric variety $Y$, namely
the stabilizer of the regular function
$w_0=-T^\epsilon+T^\epsilon v_0$.

We equip $Y^0$ with a $T^n$-invariant K\"ahler form $\omega_Y$. To make
things concrete, take $\omega_Y$ to be
the restriction of a complete toric K\"ahler form on $Y$,
with moment polytope $$\Delta_Y=\{(\xi,\eta)\in \R^n\oplus \R\,|
\,\eta\ge \varphi(\xi)=\max_{\alpha\in A} (\langle \alpha,\xi\rangle-
\rho(\alpha))\}$$ (cf.\ \eqref{eq:DeltaY}).
We denote by $\tilde{\mu}_Y:Y\to \R^{n+1}$ the moment map for the
$T^{n+1}$-action on $Y$, and by
$\mu_Y:Y^0\to \R^n$ the moment map for the $T^n$-action on
$Y^0$. Observing that $\mu_Y$ is obtained from $\tilde\mu_Y$ by restricting
to $Y^0$ and projecting to the first $n$ components, the critical locus
of $\mu_Y$ is the union of all codimension 2 toric strata, and
the set of critical values of $\mu_Y$ 
is precisely the tropical hypersurface $\Pi_0\subset\R^n$ defined by
$\varphi$. Finally, we also equip $Y^0$ with
the $T^n$-invariant holomorphic $(n+1)$-form given in each chart by 
$$\Omega_Y=d\log v_{\alpha,1}\wedge\dots
\wedge d\log v_{\alpha,n}\wedge d\log w_0.$$
Note that this holomorphic volume form scales with $\epsilon$. 
\begin{lemma}\label{l:piY}
The map $\pi_Y=(\mu_Y,|w_0|):Y^0\to B_Y=\R^n\times \R_+$ defines a
$T^n$-invariant special Lagrangian torus fibration on $Y^0$. Moreover,
$\pi_Y^{-1}(\xi,r)$ is singular if and only if $(\xi,r)\in 
\Pi_0\times\{T^\epsilon\}$, and
obstructed if and only if $r=T^\epsilon$.
\end{lemma}

This fibration is analogous to some of the examples considered in 
\cite{GrossTop,GrossSlag,CBM1,CBM2}; see also Example
3.3.1 in \cite{Au2}.

The statement that $\pi_Y^{-1}(\xi,r)$ is special Lagrangian follows
immediately from the observation that $\Omega_Y$ descends to the holomorphic
1-form $d\log w_0$ on the reduced space
$\mu_Y^{-1}(\xi)/T^n\simeq \C^*$; thus the circle $|w_0|=r$ is special
Lagrangian in the reduced space, and its lift to $\mu_Y^{-1}(\xi)$ is
special Lagrangian in $Y^0$.

A useful way to think of these tori is to consider the projection of $Y^0$
to the coordinate $w_0$, whose fibers are all isomorphic to $(\C^*)^n$ except for
$w_0^{-1}(-T^\epsilon)=v_0^{-1}(0)$ which is the union of all toric strata in $Y$.
In this projection, $\pi_Y^{-1}(\xi,r)$ fibers over the circle of
radius $r$ centered at the origin, and intersects each of the fibers
$w_0^{-1}(re^{i\theta})$ in a standard product torus (corresponding to
the level $\xi$ of the moment map).
In particular, $\pi_Y^{-1}(\xi,r)$ is singular precisely when
$r=T^\epsilon$ and $\xi\in \Pi_0$.

By the maximum principle, any holomorphic disc in $Y^0$ bounded by 
$\pi_Y^{-1}(\xi,r)$ must lie entirely within a fiber of the projection
to $w_0$. Since the regular fibers of $w_0$ are isomorphic to $(\C^*)^n$,
inside which product tori do not bound any nonconstant holomorphic discs,
$\pi_Y^{-1}(\xi,r)$ is tautologically unobstructed for $r\neq T^\epsilon$. When
$r=T^{\epsilon}$,
$\pi_Y^{-1}(\xi,r)$ intersects one of the components of 
$w_0^{-1}(-T^\epsilon)$ (i.e.\ one of the toric divisors of $Y$)
in a product torus, which bounds various families of holomorphic discs as well as
configurations consisting of holomorphic discs and rational curves in the
toric strata. This completes the proof of Lemma \ref{l:piY}.
\medskip

The maximum principle applied to $w_0$ also implies that every rational curve in $Y$ 
is contained in $w_0^{-1}(-T^\epsilon)$ (i.e.\ the union of all toric
strata), hence disjoint from the anticanonical divisor $w_0^{-1}(0)$,
and thus satisfies $c_1(Y)\cdot C=0$; in fact
$Y$ is a toric Calabi-Yau variety. So Assumption \ref{ass:nef} holds,
and partially compactifying $Y^0$ to $Y$ does not modify the enumerative 
geometry of Maslov index 0 discs bounded by the fibers of $\pi_Y$.
Hence the SYZ mirror of $Y$ is just the mirror of $Y^0$ equipped with
an appropriate superpotential, and we determine both at the same time.

The wall $r=T^\epsilon$ divides the fibration $\pi_Y:Y^0\to B_Y$ into 
two chambers; accordingly, the SYZ mirror of $Y^0$ (and $Y$) is constructed by
gluing together two coordinate charts $U'$ and $U''$ via a transformation which accounts
for the enumerative geometry of discs bounded by the potentially obstructed fibers
of $\pi_Y$. We now define coordinate systems for both charts and determine
the superpotential (for the mirror of $Y$) in terms of those 
coordinates. For notational consistency and to avoid confusion, we now 
denote by $\tau$ (rather than $T$) the Novikov parameter recording
areas with respect to $\omega_Y$.

We start with the chamber $r>T^\epsilon$, over which the fibers of $\pi_Y$ 
can be deformed into product tori in $Y$ (i.e., orbits of the
$T^{n+1}$-action) by a Hamiltonian isotopy that does not intersect
$w_0^{-1}(-T^\epsilon)$ (from the perspective of the projection to $w_0$,
the isotopy amounts simply to deforming the circle of radius $r$ centered at
$0$ to a circle of the appropriate radius centered at $-T^\epsilon$).

Fix a reference fiber $L^0=\pi_Y^{-1}(\xi^0,r^0)$,
where $\xi^0\in \R^n$ and $r^0>T^\epsilon$, and choose 
a basis $(\gamma_1,\dots,\gamma_n,\gamma'_0)$ of $H_1(L^0,\Z)$, where
$-\gamma_1,\dots,-\gamma_n$ correspond to the factors of the $T^n$-action 
on~$L^0$, and $-\gamma'_0$ corresponds to an orbit of the 
last $S^1$ factor of $T^{n+1}$ acting on a product torus $\tilde\mu_Y^{-1}(
\xi^0,\eta^0)$ which is Hamiltonian isotopic to $L^0$ in $Y$.
(The signs are motivated by consistency with the
notations used for $X^0$.)

A point of the chart $U'$ mirror to the chamber $\{r>T^\epsilon\}$ corresponds
to a pair $(L,\nabla)$, where $L=\pi_Y^{-1}(\xi,r)$ is a fiber of $\pi_Y$
(with $r>T^\epsilon$), Hamiltonian isotopic to a product torus 
$\tilde\mu_Y^{-1}(\xi,\eta)$ in $Y$, and $\nabla\in\hom(\pi_1(L),U_{\K})$.
We rescale the coordinates given by \eqref{eq:localchart} to eliminate the
dependence on the base point $(\xi^0,r^0)$, i.e.\ we identify $U'$ with
a domain in $(\K^*)^{n+1}$ via
\begin{equation}\label{eq:chart1}
(L,\nabla)\mapsto (x'_1,\dots,x'_n,z')=\left(\tau^{-\xi_1}\/\nabla(\gamma_1),
\dots,\tau^{-\xi_n}\/\nabla(\gamma_n),\tau^{-\eta}\,\nabla(\gamma'_0)\right).
\end{equation}
(Compare with \eqref{eq:localchart}, noting that
$-\xi_i=-\xi_i^0+\int_{\Gamma_i}\omega_Y$ and
$-\eta=-\eta^0+\int_{\Gamma'_0}\omega_Y$.)

\begin{lemma}\label{l:WY1}
In the chart $U'$, the superpotential for the mirror to $Y$ is given by
\begin{equation}\label{eq:WY1}
W^\vee(x'_1,\dots,x'_n,z')=\sum_{\alpha\in A} (1+\kappa_\alpha) \tau^{\rho(\alpha)}
{x'_1}^{\alpha_1}\dots {x'_n}^{\alpha_n} {z'}^{-1},
\end{equation}
where $\kappa_\alpha\in \K$ are constants with $\mathrm{val}(\kappa_\alpha)>0$.
\end{lemma}

\proof
Consider a point $(L,\nabla)\in U'$, where $L=\pi_Y^{-1}(\xi,r)$ is
Hamiltonian isotopic to 
the product torus $L'=\tilde\mu_Y^{-1}(\xi,\eta)$ in $Y$.
As explained above, the isotopy can be performed without intersecting 
the toric divisors of $Y$, i.e.\ without wall-crossing; therefore, 
the isotopy provides a cobordism between the moduli spaces of Maslov 
index 2 holomorphic discs bounded by $L$ and $L'$ in $Y$.

It is well-known that the families of Maslov index 2 holomorphic discs 
bounded by the standard product torus $L'$ in the toric manifold $Y$ are
in one-to-one correspondence with the codimension 1 toric strata of $Y$.
Namely, for each codimension 1 stratum, there is a unique family of
holomorphic discs which intersect this stratum transversely at a single
point and do not intersect any of the other strata. Moreover, every
point of $L'$ lies on the boundary of exactly
one disc of each family, and these discs are all regular
\cite{cho-oh} (see also \cite[\S 4]{Au1}).

The toric divisors of $Y$, or equivalently the facets of $\Delta_Y$, are in 
one-to-one correspondence with the elements of $A$.
The symplectic area of a Maslov index 2 holomorphic disc in $(Y,L')$ which
intersects the divisor corresponding to $\alpha\in A$ (and whose class we
denote by $\beta_\alpha$)
is equal to the
distance from the point $(\xi,\eta)$ to that facet of $\Delta_Y$, namely 
$\eta-\langle \alpha,\xi\rangle+\rho(\alpha)$, whilst the boundary of the
disc represents the class $\partial \beta_\alpha=\sum
\alpha_i\gamma_i-\gamma'_0\in H_1(L',\Z)$.
The weight associated to such a disc is therefore
$$z_{\beta_\alpha}(L',\nabla)=\tau^{\eta-\langle \alpha,\xi\rangle+\rho(\alpha)}
\nabla(\gamma_1)^{\alpha_1}\dots\nabla(\gamma_n)^{\alpha_n}
\nabla(\gamma'_0)^{-1}=\tau^{\rho(\alpha)} {x'_1}^{\alpha_1}\dots
{x'_n}^{\alpha_n}z'^{-1}.$$
Using the isotopy between $L$ and $L'$,
we conclude that the contributions of Maslov index 2 holomorphic
discs in $(Y,L)$ to the superpotential $W^\vee$ add up to
$$\sum_{\alpha\in A} z_{\beta_\alpha}(L,\nabla)=\sum_{\alpha\in A} \tau^{\rho(\alpha)} {x'_1}^{\alpha_1}\dots
{x'_n}^{\alpha_n}z'^{-1}.$$

However, the superpotential $W^\vee$ also includes contributions
from (virtual) counts of stable genus 0 configurations of discs and rational 
curves of total Maslov index~2. These configurations
consist of a single Maslov index 2
disc (in one of the above families)
together with one or more rational curves contained in the toric
divisors of $Y$ (representing a total class $C\in H_2(Y,\Z)$). 
The enumerative invariant $n(L,\beta_\alpha+C)$ giving the
(virtual) count of such configurations whose boundary passes through a 
generic point of $L$ can be understood in terms of genus 0 Gromov-Witten 
invariants of suitable partial compactifications of $Y$ (see e.g.\ 
\cite{ChanLauLeung}). However, all that matters to us is the general
form of the corresponding terms of the superpotential. Since
the rational components contribute a multiplicative factor
$\tau^{[\omega_Y]\cdot C}$ to the weight, we obtain that
$$W^\vee=\sum_{\alpha\in A}\,\Bigl(1+\sum_{\substack{C\in H_2(Y,\Z)\\
[\omega_Y]\cdot C>0}} n(L,\beta_\alpha+C)\, \tau^{[\omega_Y]\cdot C}\Bigr)
\,\tau^{\rho(\alpha)} {x'_1}^{\alpha_1}\dots {x'_n}^{\alpha_n}z'^{-1},$$
which is of the expected form \eqref{eq:WY1}.
\endproof

Next we look at the other chart $U''$, which corresponds to the
chamber $r<T^\epsilon$ of the fibration $\pi_Y$. Fix again a reference
fiber $L^0=\pi_Y^{-1}(\xi^0,r^0)$, where $\xi^0\in\R^n$ and
$r^0<T^\epsilon$, and choose a basis $(\gamma_1,\dots,\gamma_n,\gamma''_0)$
of $H_1(L^0,\Z)$, where $-\gamma_1,\dots,-\gamma_n$ correspond to the
factors of the $T^n$-action on $L^0$, and $\gamma''_0$ can be represented
by a loop in $L^0$ over which $w_0$ runs counterclockwise around the circle
of radius $r^0$ while $v_{\alpha,1},\dots,v_{\alpha,n}\in \R_+$ (for some
arbitrary choice of $\alpha$). Note that the fibration $w_0:Y\to\C$ is
trivial over the disc of radius $r^0$; in fact the coordinates
$(w_0,v_{\alpha,1},\dots,v_{\alpha,n})$ (for any $\alpha$) give a biholomorphism from the
subset $\{|w_0|\le r^0\}$ of $Y$ to $D^2(r^0)\times(\C^*)^n$. Then
$\gamma''_0$ can be characterized as the unique element of $H_1(L^0,\Z)$
which arises as the boundary of a section of $w_0:Y\to\C$ over 
the disc of radius $r^0$; we denote by $\beta_0$ the relative homotopy 
class of this section. A point of $U''$ corresponds to a pair
$(L,\nabla)$ where $L=\pi_Y^{-1}(\xi,r)$ is a fiber of $\pi_Y$ (with
$r<T^\epsilon$), and $\nabla\in \hom(\pi_1(L),U_{\K})$. As before, we rescale
the coordinates given by \eqref{eq:localchart} to eliminate the dependence
on the base point $(\xi^0,r^0)$, i.e.\ we identify $U''$ with a domain in
$(\K^*)^{n+1}$ via
\begin{equation}\label{eq:chart2}
(L,\nabla)\mapsto (x''_1,\dots,x''_n,y'')=\left(\tau^{-\xi_1}\nabla(\gamma_1),
\dots,\tau^{-\xi_n}\nabla(\gamma_n),\tau^{[\omega_Y]\cdot\beta_0}\,
\nabla(\gamma''_0)\right).
\end{equation}

\begin{lemma}\label{l:WY2}
In the chart $U''$, the superpotential for the mirror to $Y$ is given by
\begin{equation}\label{eq:WY2}
W^\vee(x''_1,\dots,x''_n,y'')=y''.
\end{equation}
\end{lemma}

\proof
By the maximum principle applied to the projection to $w_0$, any holomorphic
disc bounded by $L=\pi_Y^{-1}(\xi,r)$ in $Y$ must be contained in the subset 
$\{|w_0|\le r\}\subset Y$, which is diffeomorphic to $D^2\times
(\C^*)^n$. Thus, for topological reasons, any holomorphic disc bounded by
$L$ must represent a multiple of the class $\beta_0$. Since
the Maslov index is equal to twice the intersection number with
$w_0^{-1}(0)$, Maslov index 2 discs are holomorphic sections of
$w_0:Y\to\C$ over the disc of radius $r$, representing $\beta_0$.

The formula \eqref{eq:WY2} now follows from the claim that
the number of such sections passing through a given point
of $L$ is $n(L,\beta_0)=1$.
This can be viewed as an enumerative problem
for holomorphic sections of a trivial Lefschetz
fibration with a Lagrangian boundary condition, easily answered by 
applying the powerful methods of \cite[\S 2]{SeLES}. 
An alternative,
more elementary approach is to deform $\omega_Y$ among toric K\"ahler 
forms in its cohomology class to ensure that, for some $\xi^0\in\R^n$,
$\mu_Y^{-1}(\xi^0)$ is given in one of the coordinate charts
$Y_\alpha$ of \S \ref{ss:mirror} by equations of the form
$|v_{\alpha,1}|=\rho_1$, \dots, $|v_{\alpha,n}|=\rho_n$.
(In~fact, many natural choices for $\omega_Y$ cause this property to
hold immediately.) When this property holds, the maximum principle applied to
$v_{\alpha,1},\dots,v_{\alpha,n}$ implies that the holomorphic 
Maslov index 2 discs bounded by $L^0=\pi_Y^{-1}(\xi^0,r^0)$ are given by letting
$w_0$ vary in the disc of radius $r^0$ while the other coordinates
$v_{\alpha,1},\dots,v_{\alpha,n}$ are held constant. All these discs are
regular, and there is precisely one disc passing through each point of
$L^0$. It follows that $n(L^0,\beta_0)=1$. This completes the proof,
since the invariant $n(L^0,\beta_0)$ is not affected by the 
deformation of $\omega_Y$ to the special case we have considered, and 
the value of $n(L,\beta_0)$ is the same for all the fibers of $\pi_Y$ over
the chamber $r<T^\epsilon$.
\endproof

We can now formulate and prove the main result of this section:

\begin{theorem}\label{thm:converse}
The rigid analytic manifold
\begin{equation}\label{eq:XX0}
\mathcal{X}^0=\{(x_1,\dots,x_n,y,z)\in (\K^*)^n\times\K^2\,|\,yz=
\tilde{f}(x_1,\dots,x_n)\},\end{equation} 
where $\tilde{f}(x_1,\dots,x_n)=
\sum\limits_{\alpha\in A} (1+\kappa_\alpha)\tau^{\rho(\alpha)}x_1^{\alpha_1}
\dots x_n^{\alpha_n}$, 
is SYZ mirror to $(Y^0,\omega_Y)$. 

Moreover, the $B$-side Landau-Ginzburg model
$(\mathcal{X}^0,W^\vee=y)$ is SYZ mirror to $(Y,\omega_Y)$.
\end{theorem}

\proof
The two charts $U'$ and $U''$ are glued to each other by a coordinate
transformation which accounts for the Maslov index 0 holomorphic discs
bounded by the potentially obstructed fibers of $\pi_Y$. There are many families of
such discs, all contained in $w_0^{-1}(-T^{\epsilon})=v_0^{-1}(0)$.
However we claim that the first $n$ coordinates of the charts
\eqref{eq:chart1} and \eqref{eq:chart2} are not affected by these instanton
corrections, so that the gluing satisfies $x''_1=x'_1,\dots,x''_n=x'_n$.

One way to argue is based on the observation that all Maslov index 0
configurations are contained in $w_0^{-1}(-T^\epsilon)$.
Consider as in \S \ref{ss:syz1} a Lagrangian isotopy
$\{L_t\}_{t\in [0,1]}$ between fibers of $\pi_Y$ in the two chambers
(with $L_{t_0}$ the only potentially obstructed fiber), and
the cycles $C_\alpha=\mathrm{ev}_*[\mathcal{M}_1(\{L_{t_0}\},\alpha)]
\in H_{n-1}(L_{t_0})$ corresponding to the various classes $\alpha\in
\pi_2(Y,L_t)$ that may contain Maslov index 0 configurations. The fact that
each $C_\alpha$ is supported on $L_{t_0}\cap w_0^{-1}(-T^\epsilon)$
implies readily that $C_\alpha\cdot \gamma_1=\dots=C_\alpha\cdot\gamma_n=0$.
Since the overall gluing transformation is given by a composition of 
elementary transformations of the type \eqref{eq:instcorr}, the first $n$ coordinates are not affected.

By Corollary \ref{cor:potential_matched_by_gluing}, 
a more down-to-earth way to see that the 
gluing preserves $x''_i=x'_i$ ($i=1,\dots,n$) is to
consider the partial
compactification $Y'_i$ of $Y^0$ given by the moment polytope $\Delta_Y\cap
\{\xi_i\le K\}$ for some constant $K\gg 0$ (still removing $w_0^{-1}(0)$
from the resulting toric variety). From the perspective of the
projection $w_0:Y^0\to\C^*$, this simply amounts to a toric partial
compactification of each fiber, where the generic fiber $(\C^*)^n$ 
is partially compactified along the $i$-th factor to $(\C^*)^{n-1}\times
\C$. The Maslov index 2 holomorphic discs bounded by $L=\pi_Y^{-1}(\xi,r)$
inside $Y'_i$ are contained in the fibers of $w_0$ by the maximum principle;
requiring that the boundary of the disc pass through a given point $p\in L$
(where we assume $w_0\neq -T^\epsilon$),
we are reduced to the fiber of $w_0$ containing $p$, which $L$ intersects
in a standard product torus $(S^1)^n\subset (\C^*)^{n-1}\times \C$ (where
the radii of the various $S^1$ factors depend on $\xi$).
Thus, there is
exactly one Maslov index 2 holomorphic disc in $(Y'_i,L)$ 
through a generic point $p\in L$ (namely a disc over which all coordinates
except the $i$-th one are constant). The superpotential is equal to 
the weight of this disc, i.e.\ $\tau^{K-\xi_i}\,\nabla(\gamma_i)$, 
which can be rewritten as $\tau^K x'_i$ if $r>T^\epsilon$, and
$\tau^K x''_i$ if $r<T^\epsilon$. Comparing these two expressions,
we see that the gluing between $U'$ and $U''$ identifies $x'_i=x''_i$.

The gluing transformation between the coordinates $y''$ and $z'$ is
more complicated, but is now determined entirely by a comparison between 
\eqref{eq:WY1} and \eqref{eq:WY2}: since the two formulas for $W^\vee$ must 
glue to a regular function on
the mirror, $y''$ must equal the right-hand side of
\eqref{eq:WY1}, hence
$$y''z'=\sum_{\alpha\in A}(1+\kappa_\alpha)\tau^{\rho(\alpha)}{x'_1}^{\alpha_1}
\dots {x'_n}^{\alpha_n}=\tilde{f}(x'_1,\dots,x'_n).$$
This completes the proof of the theorem.
\endproof

The first part of Theorem \ref{thm:converse} is a statement of SYZ mirror
symmetry in the opposite direction from Theorem~\ref{thm:conicbundle}; the
two results taken together relate the symplectic topology and algebraic geometry of
the spaces $X^0$ and $Y^0$ to each other. 
More precisely, we would like to treat $\tau$ as a fixed complex number and
view the mirror of $(Y^0, \omega_Y)$ as a complex manifold. 
The convergence of the function $\tilde{f}$ depends only on that of the
constants $\kappa_\alpha$, which is unknown in general but holds in practice
for a number of examples (see \cite{ChanLauLeung} and other work by the same authors).
Even when convergence is not an issue, the result reveals the need for care in constructing the mirror map: while our main construction is
essentially independent of the coefficients $c_\alpha$ appearing in
\eqref{eq:Ht} (which do not affect the symplectic geometry of $X^0$),
the direction considered here requires the complex structure of $X^0$ 
to be chosen carefully to match with the K\"ahler class $[\omega_Y]$,
specifically we have to take $c_\alpha=1+\kappa_\alpha$.

The second part of Theorem \ref{thm:converse} gives a mirror symmetric
interpretation of the partial compactification of $Y^0$ to $Y$, in terms
of equipping $X^0$ with the superpotential $W^\vee=y$. From the 
perspective of our main construction (viewing $X^0$ as a symplectic
manifold and $Y^0$ as its SYZ mirror), we saw the same phenomenon
in Section \ref{s:backtoH}.


\section{Examples} \label{s:examples}

\subsection{Hyperplanes and pairs of pants}\label{ss:pants}
We consider as our first example the (higher dimensional) pair of pants 
$H$ defined by the
equation 
\begin{equation}\label{eq:pairofpants}
x_1+\dots+x_n+1=0\end{equation}
in $V=(\C^*)^n$. (The case $n=2$ corresponds
to the ordinary pair of pants; in general $H$ is the complement of $n+1$
hyperplanes in general position in $\CP^{n-1}$.) 

The tropical polynomial corresponding to \eqref{eq:pairofpants} 
is $\varphi(\xi)=\mathrm{max}(\xi_1,\dots,\xi_n,0)$;
the polytope $\Delta_Y$ defined by \eqref{eq:DeltaY} is
equivalent via $(\xi_1,\dots,\xi_n,\eta)\mapsto
(\eta-\xi_1,\dots,\eta-\xi_n,\eta)$ to the orthant $(\R_{\ge
0})^{n+1}\subset\R^{n+1}$. Thus $Y\simeq \C^{n+1}$. In terms of the
coordinates $(z_1,\dots,z_{n+1})$ of $\C^{n+1}$, the monomial $v_0$ is given by
$v_0=z_1\dots z_{n+1}$. Thus, in this example our main results are:

\begin{enumerate}
\item the open Calabi-Yau manifold
$Y^0=\C^{n+1}\setminus \{z_1\dots z_{n+1}=1\}$ is SYZ mirror to
the conic bundle $X^0=\{(x_1,\dots,x_n,y,z)\in (\C^*)^n\times \C^2\,|\,
yz=x_1+\dots+x_n+1\}$;\smallskip
\item the $B$-side Landau-Ginzburg model
$(Y^0,W_0=-T^\epsilon+T^\epsilon\,z_1\dots z_{n+1})$ is SYZ mirror to
the blowup $X$ of $(\C^*)^n\times \C$ along $H\times 0$, where
$$H=\{(x_1,\dots,x_n)\in (\C^*)^n\,|\, x_1+\dots+x_n+1=0\};$$
\item the $B$-side Landau-Ginzburg model $(\C^{n+1},W_0^H=-z_1\dots z_{n+1})$ 
is a generalized SYZ mirror of $H$.
\end{enumerate}

\noindent
The last statement in particular has been verified in the sense of
homological mirror symmetry by Sheridan \cite{Sheridan}; see also
\cite{AAEKO} for a more detailed result in the case $n=2$ (the usual pair 
of pants).

If instead we consider the same equation \eqref{eq:pairofpants} to 
define (in an affine chart) a hyperplane
$H\simeq\CP^{n-1}$ inside $V=\CP^n$, with a K\"ahler form such that
$\int_{\CP^1}\omega_V=A$, then our main result becomes that the
$B$-side Landau-Ginzburg model consisting of 
$Y^0=\C^{n+1}\setminus \{z_1\dots z_{n+1}=1\}$ equipped with the
superpotential $$W_0=-T^\epsilon+
T^\epsilon z_1\dots z_{n+1}+z_1+\dots+z_n+T^{A} z_{n+1}$$ is SYZ mirror to
the blowup $X$ of $\CP^n\times\C$ along $H\times 0\simeq \CP^{n-1}\times 0$.

Even though $\CP^{n-1}$ is not affine, Theorem \ref{cor:main} still holds
for this example if we assume that $n\ge 2$, by Remark \ref{rmk:m0shift}. In
this case, the mirror we obtain for $\CP^{n-1}$ (viewed as a hyperplane
in $\CP^n$) is the $B$-side Landau-Ginzburg model
$$(\C^{n+1}, W_0^H=-z_1\dots z_{n+1}+z_1+\dots+z_n+T^A z_{n+1}).$$
Rewriting the superpotential as $$W_0^H=z_1+\dots+z_n+z_{n+1}(T^A-z_1\dots
z_n)=\tilde{W}(z_1,\dots,z_n)+z_{n+1}\,g(z_1,\dots,z_n)$$ makes it apparent that this $B$-side Landau-Ginzburg model is equivalent
(e.g.\ in the sense of Orlov's generalized Kn\"orrer periodicity \cite{Orlov})
to the $B$-side Landau-Ginzburg model consisting of $g^{-1}(0)=\{(z_1,\dots,z_n)\in
\C^n\,|\,z_1\dots z_n=T^A\}$ equipped with the superpotential $\tilde{W}=
z_1+\dots+z_n$, which is the classical toric mirror of $\CP^{n-1}$.

\subsection{ALE spaces}\label{ss:ALE}
Let $V=\C$, and let $H=\{x_1,\dots,x_{k+1}\}\subset \C^*$ consist of
$k+1$ points, $k\ge 0$, with $|x_1|\ll \dots\ll |x_{k+1}|$ (so that the defining
polynomial of $H$, $f_{k+1}(x)=(x-x_1)\dots(x-x_{k+1})\in \C[x]$, is near the
tropical limit).

The conic bundle $X^0=\{(x,y,z)\in \C^*\times\C^2\,|
\,yz=f_{k+1}(x)\}$ is the complement of the regular conic $x=0$ 
in the $A_k$-Milnor fiber $$X'=\{(x,y,z)\in \C^3\,|\, yz=f_{k+1}(x)\}.$$
In fact, $X'$ is the main space of interest here, rather than its open
subset $X^0$ or its partial compactification $X$ (note that $X'=X\setminus
\tilde{V}$). However the mirror of $X'$
differs from that of $X$ simply by excluding the term $w_0$ (which
accounts for those holomorphic discs that intersect $\tilde{V}$) from the
mirror superpotential. 

The tropical polynomial $\varphi:\R\to \R$ corresponding to $f_{k+1}$ is a 
piecewise linear function whose slope takes the successive integer values
$0,1,\dots,k+1$. Thus the toric variety $Y$ determined by the
polytope $\Delta_Y=\{(\xi,\eta)\in\R^2\,|\,
\eta\ge \varphi(\xi)\}$ is the
resolution of the $A_k$ singularity $\{st=u^{k+1}\}\subset \C^3$.
The $k+2$ edges of $\Delta_Y$ correspond to the
toric strata of $Y$, namely the proper transforms of the
coordinate axes $s=0$ and $t=0$ and the $k$ rational $(-2)$-curves
created by the resolution.
Specifically, $Y$ is covered by $k+1$ affine coordinate charts $U_{\alpha}$
 with coordinates $(s_\alpha=v_{\alpha,1},t_\alpha=v_{\alpha+1,1}^{-1})$,
$0\le \alpha\le k$; denoting the toric coordinate $v_{\alpha,0}$ by $u$, 
equation \eqref{eq:toriccharts} becomes $s_\alpha t_\alpha=u$,
and the regular functions $s=s_0,t=t_k,u\in \O(Y)$ satisfy the relation
$st=u^{k+1}$.

Since $w_0=-T^\epsilon+T^\epsilon v_0=-T^\epsilon+T^\epsilon u$,
the space $Y^0$ is the complement of the curve $u=1$ inside
$Y$. With this understood, our main results become:
\begin{enumerate}
\item
the complement $Y^0$ of the curve $u=1$ in the resolution $Y$ of
the $A_k$ singularity $\{st=u^{k+1}\}\subset\C^3$ is SYZ mirror to
the complement $X^0$ of the curve $x=0$ in the Milnor fiber
$X'=\{(x,y,z)\in \C^3\,|\,yz=f_{k+1}(x)\}$ of the $A_k$ singularity;
\smallskip
\item
the $B$-side Landau-Ginzburg model $(Y^0,W_0=s)$ is SYZ mirror to $X'$;
\smallskip
\item
the Landau-Ginzburg models $(Y,W_0=s)$ and $(X',W^\vee=y)$ are SYZ mirror
to each other.
\end{enumerate}

\noindent
These results show that the oft-stated mirror symmetry relation
between the smoothing and the resolution of the $A_k$ singularity
(or, specializing to the case $k=1$, between the affine quadric
$T^*S^2$ and the total space of the line bundle $\O(-2)\to \PP^1$) 
needs to be corrected either by removing smooth curves from each
side, or by equipping both sides with superpotentials.

One final comment that may be of interest to symplectic geometers is that
$W_0=s$ vanishes to order $k+1$ along the $t$ coordinate axis, and to
orders $1,2,\dots,k$ along the exceptional curves of the resolution. The higher derivatives of the superpotential encode information about the $A_\infty$-products on the Floer cohomology of the Lagrangian torus fiber of the SYZ fibration, and the
high-order vanishing of $W_0$ along the toric divisors of $Y^0$ indicates
that the $A_k$ Milnor fiber contains Lagrangian tori whose Floer 
cohomology is isomorphic to the usual cohomology of $T^2$ as an algebra, 
but carries non-trivial $A_\infty$-operations. (See also \cite{LM} for
related considerations.)

\begin{corollary}
For\/ $\alpha\in\{2,\dots,k+1\}$, let $r\in \R_+$ be such that exactly\/
$\alpha$ of the points $x_1,\dots,x_{k+1}$ satisfy $|x_i|<r$. Then the 
Floer cohomology of the Lagrangian torus 
$T_r=\{(x,y,z)\in X'\,|\ |x|=r,\ |y|=|z|\}$ in the $A_k$ Milnor fiber
$X'$, equipped with a suitable spin structure, is\/ $\mathrm{HF}^*(T_r,T_r)\simeq
H^*(T^2;\Lambda)$, equipped with an $A_\infty$-structure for which the
generators $a,b$ of\/ $\mathrm{HF}^1(T_r,T_r)$ satisfy the relations
$\m_2(a,b)+\m_2(b,a)=0$; $\m_i(a,\dots,a)=0$ for all $i$; 
$\m_i(b,\dots,b)=0$ for $i\le \alpha-1$; and $\m_\alpha(b,\dots,b)\neq 0$.
\end{corollary}
\begin{proof}
The condition $|x|=r$ implies that the torus $T_r$ yields a point in the
chamber $U_{\alpha}$, while the condition that $|y|=|z|$ implies that it lies on the critical
locus of $W_0$: this shows that $T_{r}$ is a critical point of $W_0$ of order $\alpha+1$.   

By a construction which is standard in the toric case (see \cite{cho}), the restriction of
$W_0$ to a chart of $Y$ modeled after a domain in $H^{1}(T_{r},\Lambda^{*})$ (identified with $(\Lambda^*)^2$ by choosing the basis
$(a,b)$) agrees with the map
\begin{equation}
(\exp(\lambda_a),\exp(\lambda_b))  \mapsto  \sum_{k} \m_k(\lambda_a a + \lambda_b b , \ldots, \lambda_a a + \lambda_b b) .
\end{equation}
Choosing $a$ to correspond to the generator which vanishes on loops whose projection to $\C$ is constant, the result follows immediately.
\end{proof}
\subsection{Plane curves}\label{ss:planecurves}
For $p,q\ge 2$, consider a smooth Riemann surface $H$ of genus $g=(p-1)(q-1)$
embedded in $V=\PP^1\times\PP^1$, defined as the zero set of a 
suitably chosen polynomial of bidegree 
$(p,q)$.
(The case of a genus 2 curve of bidegree $(3,2)$ was used in 
\S\ref{s:notations} to illustrate the general construction, see Examples
\ref{ex:genus2} and \ref{ex:genus2part2}.)

Namely, in affine coordinates $f$ is given by
$$f(x_1,x_2)=\sum_{a=0}^p \sum_{b=0}^q c_{a,b} \tau^{\rho(a,b)}
x_1^a x_2^b,$$ where $c_{a,b}\in\C^*$ are arbitrary,
$\rho(a,b)\in \R$ satisfy a suitable convexity condition, and $\tau\ll 1$.
The corresponding tropical
polynomial \begin{equation}\label{eq:tropical_genusg}
\varphi(\xi_1,\xi_2)=\max \{ a\xi_1+b\xi_2-\rho(a,b)\,|\,
0\le a\le p,\ 0\le b\le q\}
\end{equation} defines a tropical curve $\Pi_0\subset \R^2$;
see Figure \ref{fig:Pi0-genus2}.
We also denote by $H'$, resp.\ $H^0$, the genus $g$ curves with $p+q$ (resp.\
$2(p+q)$)
punctures obtained by intersecting $H$ with the affine subset
$V'=\C^2\subset V$, resp.\ $V^0=(\C^*)^2$.

The polytope $\Delta_Y=\{(\xi_1,\xi_2,\eta)\,|\,\eta\ge
\varphi(\xi_1,\xi_2)\}$  has $(p+1)(q+1)$ facets, corresponding to the 
regions where a particular term in \eqref{eq:tropical_genusg} realizes
the maximum. Thus the 3-fold $Y$ has $(p+1)(q+1)$ irreducible toric divisors $D_{a,b}$
($0\le a\le p$, $0\le b\le q$) (we label each divisor by the weight of the dominant
monomial).
The moment polytopes for these divisors are exactly the components of
$\R^2\setminus \Pi_0$, and the tropical curve $\Pi_0$ depicts the moment
map images of the codimension 2 strata where they intersect 
(a configuration of $\PP^1$'s and $\mathbb{A}\!^1$'s); see Figure
\ref{fig:W0Hcrit} left (and compare with Figure \ref{fig:Pi0-genus2} right).

The leading-order superpotential $W_0$ of Definition \ref{def:LGmirror} consists of five terms: 
$w_0=-T^\epsilon+T^\epsilon v_0$, where $v_0$ is the toric monomial of weight $(0,0,1)$,
which vanishes with multiplicity 1 on each of the toric divisors $D_{a,b}$; 
and four terms $w_1,\dots,w_4$ corresponding to the facets of $\Delta_V$.
Up to constant factors, $w_1$ is the toric monomial with weight $(-1,0,0)$, 
which vanishes with multiplicity
$a$ on $D_{a,b}$; $w_2$ is the toric monomial with weight
$(0,-1,0)$, vanishing with multiplicity $b$ on $D_{a,b}$;
$w_3$ is the monomial with weight $(1,0,p)$, with multiplicity $(p-a)$ on
$D_{a,b}$; and $w_4$ is the monomial with weight $(0,1,q)$, with
multiplicity $(q-b)$ on $D_{a,b}$ (compare Example \ref{ex:genus2part2}).

Our main results for the open curve $H^0\subset V^0=(\C^*)^2$ are the following:
\begin{enumerate}
\item the complement $Y^0$ of $w_0^{-1}(0)\simeq (\C^*)^2$
in the toric 3-fold $Y$ is SYZ mirror to the conic bundle
$X^0=\{(x_1,x_2,y,z)\in (\C^*)^2\times \C^2\,|\,yz=f(x_1,x_2)\}$;
\item the $B$-side Landau-Ginzburg model $(Y^0,w_0)$ is SYZ mirror to
the blowup of $(\C^*)^2\times\C$ along $H^0\times 0$;
\item the $B$-side Landau-Ginzburg model $(Y,-v_0)$ is a generalized SYZ 
mirror to the open genus $g$ curve $H^0$.
\end{enumerate}

\noindent
The $B$-side Landau-Ginzburg models $(Y^0,w_0)$ and $(Y,-v_0)$ have regular fibers
isomorphic to $(\C^*)^2$, while the singular fiber
$w_0^{-1}(-T^\epsilon)=v_0^{-1}(0)$ is the union of all the toric divisors
$D_{a,b}$. In particular, the singular fiber consists of $(p+1)(q+1)$ toric surfaces
intersecting pairwise along a configuration of $\PP^1$'s and 
$\mathbb{A}\!^1$'s (the 1-dimensional strata of $Y$), themselves
intersecting at triple points (the 0-dimensional strata of $Y$); the
combinatorial structure of the trivalent configuration of $\PP^1$'s and 
$\mathbb{A}\!^1$'s is exactly given by the tropical curve $\Pi_0$.
(See Figure \ref{fig:W0Hcrit} left).

If we partially compactify to $V'=\C^2$, then we get:
\begin{enumerate}
\item[(2')] the $B$-side Landau-Ginzburg model $(Y^0,w_0+w_1+w_2)$ is SYZ mirror to
the blowup of $\C^3$ along $H'\times 0$;
\item[(3')] the $B$-side Landau-Ginzburg model $(Y,-v_0+w_1+w_2)$ is mirror to $H'$.
\end{enumerate}

Adding $w_1+w_2$ to the superpotential results in a partial smoothing of the
singular fiber; namely, the singular fiber is now the union of the 
toric surfaces $D_{a,b}$ where $a>0$ and $b>0$ (over which $w_1+w_2$
vanishes identically) and a single noncompact surface $S'\subset Y$, 
which can be thought of as a smoothing (or partial smoothing)
of $S'_0=\left(\bigcup_a
D_{a,0}\right) \cup \left(\bigcup_b D_{0,b}\right)$.

By an easy calculation in the toric affine charts of $Y$,
the critical locus of $W_{H'}=-v_0+w_1+w_2$ (i.e.\
the pairwise intersections of components of $W_{H'}^{-1}(0)$ and the
possible self-intersections of $S'$) is again a union of 
$\PP^1$'s and $\mathbb{A}\!^1$'s  meeting at triple points;
the combinatorics of this configuration is obtained from the planar graph
$\Pi_0$ (which describes the critical locus of $W_{H^0}=-v_0$) by deleting
all the unbounded edges in the directions of $(-1,0)$ and $(0,-1)$, 
then inductively collapsing the bounded edges that connect to univalent 
vertices and merging the edges that meet at bivalent vertices (see
Figure \ref{fig:W0Hcrit} middle); this
construction can be understood as a sequence of ``tropical modifications''
applied to the tropical curve $\Pi_0$.

\begin{figure}
\setlength{\unitlength}{4mm}\psset{unit=\unitlength}
\begin{picture}(8.5,8.5)(-4.25,-4.25)
\psline(-4.25,-3)(-3,-3)
\psline(-4.25,0)(-2,0)
\psline(-3,-4.25)(-3,-3)
\psline(-1,-4.25)(-1,-2)
\psline(1,-4.25)(1,-1)
\psline(4.25,3)(3,3)
\psline(4.25,0)(2,0)
\psline(3,4.25)(3,3)
\psline(1,4.25)(1,2)
\psline(-1,4.25)(-1,1)
\pspolygon(-2,-2)(-1,-2)(0,-1)(0,1)(-1,1)(-2,0)
\psline(-3,-3)(-2,-2)
\pspolygon(2,2)(1,2)(0,1)(0,-1)(1,-1)(2,0)
\psline(3,3)(2,2)
\put(-4,-3.8){\small \makebox(0,0)[cc]{$D_{00}$}}
\put(-2,-3.5){\small \makebox(0,0)[cc]{$D_{10}$}}
\put(0,-2.7){\small \makebox(0,0)[cc]{$D_{20}$}}
\put(2.5,-2){\small \makebox(0,0)[cc]{$D_{30}$}}
\put(4,3.8){\small \makebox(0,0)[cc]{$D_{32}$}}
\put(2,3.5){\small \makebox(0,0)[cc]{$D_{22}$}}
\put(0,2.7){\small \makebox(0,0)[cc]{$D_{12}$}}
\put(-2.5,2){\small \makebox(0,0)[cc]{$D_{02}$}}
\put(-3.2,-1.5){\small \makebox(0,0)[cc]{$D_{01}$}}
\put(-1,-0.5){\small \makebox(0,0)[cc]{$D_{11}$}}
\put(1,0.5){\small \makebox(0,0)[cc]{$D_{21}$}}
\put(3.2,1.5){\small \makebox(0,0)[cc]{$D_{31}$}}
\end{picture}
\qquad
\begin{picture}(8.5,8.5)(-4.25,-4.25)
\psline(4.25,3)(3,3)
\psline(4.25,0)(2,0)
\psline(3,4.25)(3,3)
\psline(1,4.25)(1,2)
\psline(-1,4.25)(-1,1)
\pscurve(0,-1)(-1,-2)(-2,-2)(-2,0)(-1,1)
\psline(0,-1)(0,1)(-1,1)
\psline(2,0)(2,2)(1,2)(0,1)(0,-1)
\pscurve(0,-1)(1,-1)(2,0)
\psline(3,3)(2,2)
\put(4,3.8){\small \makebox(0,0)[cc]{$D_{32}$}}
\put(2,3.5){\small \makebox(0,0)[cc]{$D_{22}$}}
\put(0,2.7){\small \makebox(0,0)[cc]{$D_{12}$}}
\put(-1,-0.5){\small \makebox(0,0)[cc]{$D_{11}$}}
\put(1,0.5){\small \makebox(0,0)[cc]{$D_{21}$}}
\put(3.2,1.5){\small \makebox(0,0)[cc]{$D_{31}$}}
\put(-2.5,2){\small \makebox(0,0)[cc]{$S'$}}
\put(2.5,-2){\small \makebox(0,0)[cc]{$S'$}}
\end{picture}
\qquad
\begin{picture}(8.5,8.5)(-4.25,-4.25)
\pscurve(0,-1)(-1,-2)(-2,-2)(-2,0)(-1,1)(0,1)
\pscurve(0,1)(1,2)(2,2)(2,0)(1,-1)(0,-1)
\psline(0,-1)(0,1)
\put(-1,-0.5){\small \makebox(0,0)[cc]{$D_{11}$}}
\put(1,0.5){\small \makebox(0,0)[cc]{$D_{21}$}}
\put(-2.5,2){\small \makebox(0,0)[cc]{$S$}}
\put(2.5,-2){\small \makebox(0,0)[cc]{$S$}}
\end{picture}
\caption{The singular fibers of the mirrors to 
$H^0=H\cap(\C^*)^2$ (left) and $H'=H\cap \C^2$ (middle), and of the 
leading-order terms of the mirror to $H$ (right). Here $H$ is a genus 2
curve of bidegree $(3,2)$ in $\PP^1\times\PP^1$.}
\label{fig:W0Hcrit}
\end{figure}
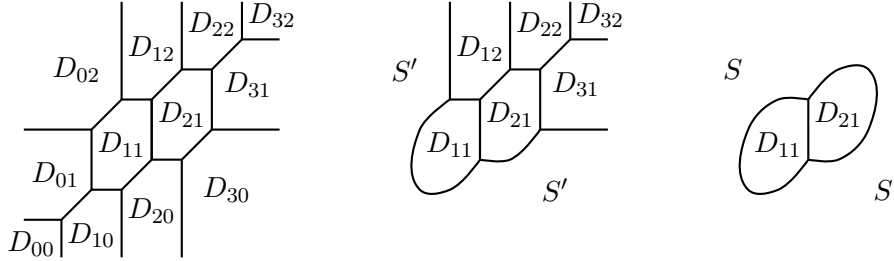

The closed genus $g$ curve $H$ does not satisfy Assumption
\ref{ass:affinecase}, so our main results do not apply to it.
However, it is instructive to consider the leading-order mirrors
$(Y^0,W_0)$ to the blowup $X$ of $\PP^1\times\PP^1\times\C$ along $H\times
0$ and $(Y,W_0^H)$ to the curve $H$ itself. Indeed, in this case the 
additional instanton corrections (i.e., virtual counts of
configurations that include exceptional rational curves in $\tilde{V}$)
are expected to only have a mild effect on the mirror:
specifically, they should not affect the 
topology of the critical locus, but merely deform it in a way that can
be accounted for by corrections to the mirror map. We will return to this
question in a forthcoming paper.

The zero set of the leading-order superpotential $W_0^H=-v_0+w_1+w_2+w_3+w_4$ 
is the union of the compact toric surfaces $D_{a,b}$, $0<a<p$, $0<b<q$, with
a single noncompact surface $S\subset Y$, which can be thought of as a
smoothing (or partial smoothing) of the union $S_0$ of the noncompact toric 
divisors of $Y$. (There may also be new critical points which do not lie
over $0$; we shall not discuss them.)

Here again, an easy calculation in the toric affine charts shows that
the singular locus of $(W_0^H)^{-1}(0)$ (i.e.,
the pairwise intersections of components and the
possible self-intersections of $S$) forms a configuration of $3g-3$
$\PP^1$'s meeting at triple points.
Combinatorially, this configuration is obtained from the planar graph
$\Pi_0$ by deleting all the unbounded edges,
then inductively collapsing the bounded edges that connect to univalent 
vertices and merging the edges that meet at bivalent vertices (see
Figure~\ref{fig:W0Hcrit} right); this
can be understood as a sequence of tropical modifications
turning $\Pi_0$ into a closed genus $g$ tropical curve (i.e.,
a trivalent graph without unbounded edges).

(The situation is slightly different when $p=q=2$ and $g=1$: in this
case $(W_0^H)^{-1}(0)=D_{1,1}\cup S$, and the critical locus $D_{1,1}
\cap S$ is a smooth elliptic curve. In this case, the higher instanton
corrections are easy to analyze, and simply amount to rescaling the
first term $-v_0$ of the superpotential by a multiplicative factor which encodes certain
genus 0 Gromov-Witten invariants of $\PP^1\times\PP^1$.)

\section{Generalizations} \label{s:generalizations}
In this section we mention (without details) a couple of straightforward
generalizations of our construction.

\subsection{Non-maximal degenerations}\label{ss:nonmaximal}

In our main construction we have assumed that the hypersurface $H\subset
V$ is part of a maximally degenerating family $(H_\tau)_{\tau\to 0}$
(see Definition \ref{def:maxdegeneration}). This was
used for two purposes: (1) to ensure that, for each weight
$\alpha\in A$, there exists a connected component of $\R^n\setminus
\mathrm{Log}(H)$ over which the corresponding monomial in the defining
equation \eqref{eq:Ht} dominates all other terms, and (2) to ensure that
the toric variety $Y$ associated to the polytope \eqref{eq:DeltaY} is smooth.

(Note that the regularity of $\mathcal{P}$ also ensures the
smoothness of $H$ throughout, and of $H'_\sigma$ in the discussion before 
Lemma \ref{l:Wsigma2}; without the regularity assumption,
smoothness can still be achieved by making generic choices of the 
coefficients $c_\alpha$ in \eqref{eq:Ht}.)

In general, removing the assumption of maximal degeneration, some of the terms in the tropical polynomial
$$\varphi(\xi)=\max\,\{\langle \alpha,\xi\rangle-\rho(\alpha)\,|\,\alpha\in A\}$$
may not achieve the maximum under any circumstances; denote by $A_{red}$ the
set of those weights which do achieve the maximum for some value of $\xi$.
Equivalently, those are exactly the vertices of the polyhedral decomposition $\mathcal{P}$
of $\mathrm{Conv}(A)$ induced by the function $\rho:A\to\R$.
Then the elements of $A\setminus A_{red}$ do not give rise to connected components of the complement
of the tropical curve, nor to facets of $\Delta_Y$, and should be discarded altogether. 
Thus, the main difference with the maximal degeneration case is that
the rays of the fan $\Sigma_Y$ are the
vectors $(-\alpha,1)$ for $\alpha\in A_{red}$, and
the toric variety $Y$ is usually singular.

Indeed, the construction of the Lagrangian torus fibration $\pi:X^0\to B$ proceeds
as in \S \ref{s:fibrations}, and the arguments in Sections \ref{s:fibrations}
to \ref{s:generalcase} remain valid, the only difference being that only the
weights $\alpha\in A_{red}$ give rise to chambers $U_\alpha$ of tautologically unobstructed
fibers of $\pi$, and hence to affine coordinate charts $U_\alpha^\vee$ for the
SYZ mirror $Y^0$ of $X^0$. Replacing $A$ by $A_{red}$ throughout the arguments
addresses this issue. 

The smooth mirrors obtained from maximal degenerations are 
crepant resolutions of the singular mirrors obtained from non-maximal ones.
Starting from a non-maximal polyhedral decomposition
$\mathcal{P}$, the various ways in which it can be refined to a regular
decomposition correspond to different choices of resolution.
We give a few examples.

\begin{example}
Revisiting the example of the $A_k$-Milnor fiber considered in
\S \ref{ss:ALE}, we now consider the case where the roots of the polynomial
$f_{k+1}$ satisfy $|x_1|=\dots=|x_{k+1}|$, for example
$f_{k+1}(x)=x^{k+1}-1$, which gives
$$X'=\{(x,y,z)\in\C^3\,|\,yz=x^{k+1}-1\}.$$
Then the tropical polynomial $\varphi:\R\to\R$ is
$\varphi(\xi)=\max(0,(k+1)\xi)$, and the polytope $\Delta_Y=\{(\xi,\eta)
\in\R^2\,|\,\eta\ge \varphi(\xi)\}$ determines the singular toric variety
$\{st=u^{k+1}\}\subset\K^3$, i.e.\ the $A_k$ singularity, rather than
its resolution as previously.

Geometrically, the Lagrangian torus fibration $\pi$ normally consists
of $k+2$ chambers, depending on how many of the roots of $f_{k+1}$ lie
inside the projection of the fiber to the $x$ coordinate plane. In the
case considered here, all the walls lie at $|x|=1$, and the fibration
$\pi$ only consists of two chambers ($|x|<1$ and $|x|>1$). 

In fact, $\Z/(k+1)$ acts freely on $X_k^0=\{(x,y,z)\in
\C^*\times\C^2\,|\,yz=x^{k+1}-1\}$, making it an unramified cover of
$X^0_0=\{(\hat{x},y,z)\in\C^*\times\C^2\,|\,yz=\hat{x}-1\}\simeq \C^2\setminus
\{yz=-1\}$ via the map $(x,y,z)\mapsto (x^{k+1},y,z)$. The Lagrangian tori we consider on $X^0_k$ are simply the
preimages of the SYZ fibration on $X^0_0$,
which results in the mirror being the quotient of the mirror of $X^0_0$
(namely, $\{(\hat{s},\hat{t},u)\in\K^3\,|\,\hat{s}\hat{t}=u,\ u\neq 1\}$) by a 
$\Z/(k+1)$-action (namely $\zeta\cdot(\hat{s},\hat{t},u)=(\zeta
\hat{s},\zeta^{-1}\hat{t},u)$).
As expected, the quotient is nothing other than
$Y_k^0=\{(s,t,u)\in\K^3\,|\,st=u^{k+1},\ u\neq 1\}$ (via the map
$(\hat{s},\hat{t},u)\mapsto (\hat{s}^{k+1},\hat{t}^{k+1},u)$).
\end{example}

\begin{example}
The higher-dimensional analogue of the previous example is that of
Fermat hypersurfaces in $(\C^*)^n$ or in $\CP^n$. Let $H$ be the Fermat
hypersurface in $\CP^n$ given by the equation $\sum X_i^d=0$ in
homogeneous coordinates, i.e.\ $x_1^d+\dots+x_n^d+1=0$ in 
affine coordinates, and let $X$ be the blowup of $\CP^n\times\C$ at
$H\times 0$. In this case, the open Calabi-Yau manifold $X^0$ is
$$X^0=\{(x_1,\dots,x_n,y,z)\in (\C^*)^n\times\C^2\,|\,yz=x_1^d+\dots+
x_n^d+1\}.$$
The tropical polynomial corresponding to $H$ is $\varphi(\xi_1,\dots,\xi_n)=
\max(d\xi_1,\dots,d\xi_n,0)$, which is highly degenerate. Thus the
toric variety $Y$ associated to the polytope $\Delta_Y$ given by 
\eqref{eq:DeltaY} is singular, in fact it can be described as
$$Y=\{(z_1,\dots,z_{n+1},v)\in\K^{n+2}\,|\,z_1\dots z_{n+1}=v^d\},$$
which can be viewed as the quotient of $\K^{n+1}$ by the diagonal action of $(\Z/d)^n$
(multiplying all coordinates by roots of unity but preserving their
product), via the map $(\tilde{z}_1,\dots,\tilde{z}_{n+1})\mapsto
(\tilde{z}_1^d,\dots,\tilde{z}_{n+1}^d,\tilde{z}_1\dots\tilde{z}_{n+1})$.
As in the previous example, this is consistent with the observation that
$X^0$ is a $(\Z/d)^n$-fold cover of the conic bundle considered in 
\S \ref{ss:pants}, where $(\Z/d)^n$ acts diagonally by multiplication on the
coordinates $x_1,\dots,x_n$.

(As usual, considering a maximally degenerating family of hypersurfaces 
of degree $d$ instead of a Fermat hypersurface would yield a crepant resolution of $Y$.)

By Theorem \ref{cor:main}, the affine Fermat hypersurface $H^0=H\cap (\C^*)^n$
is mirror to the singular $B$-side Landau-Ginzburg model $(Y,W_0^H=-v)$ or, in other terms, 
the quotient of $(\K^{n+1},\tilde{W}_0^H=-\tilde{z}_1\dots\tilde{z}_{n+1})$
by the action of $(\Z/d)^n$, which is consistent with 
\cite{Sheridan}. 

Furthermore, by Remark \ref{rmk:m0shift} the theorem also
applies to projective Fermat hypersurfaces of degree $d<n$ in $\CP^n$.
Setting $a=\frac{1}{n+1}\int_{\CP^1}\omega_{\CP^n}$, and placing the 
barycenter of the moment polytope of $\CP^n$ at the origin, 
we find that
$$\left(Y,W_0^H=-v+T^{a}(z_1+\dots+z_{n+1})\right)$$
is mirror to $H$ (for $d<n$; otherwise this is only the leading-order
approximation to the mirror). Equivalently, this can be viewed as the
quotient of 
$$\left(\K^{n+1},\tilde{W}_0^H=-\tilde{z_1}\dots\tilde{z}_{n+1}+
T^{a}(\tilde{z}_1^d+\dots+\tilde{z}_{n+1}^d)\right)$$
by the action of
$(\Z/d)^n$, which is again consistent with Sheridan's work.
\end{example}

\begin{example}
We now revisit the example considered in \S \ref{ss:planecurves},
where we found the mirrors of nearly tropical plane curves of
bidegree $(p,q)$
to be smooth toric 3-folds (equipped with suitable superpotentials)
whose topology is determined by the
combinatorics of the corresponding tropical plane curve $\Pi_0$ (or dually, 
of the regular triangulation $\mathcal{P}$ of the rectangle
$[0,p]\times [0,q]$).

A particularly simple way to modify the combinatorics is to ``flip'' a
pair of adjacent triangles of $\mathcal{P}$ whose union is a unit parallelogram;
this affects the toric 3-fold 
$Y$ by a flip. This operation can be implemented by a continuous deformation of the
tropical curve $\Pi_0$ in which the length of a bounded edge shrinks to
zero, creating a four-valent vertex, which is then resolved by creating
a bounded edge in the other direction and increasing its length. The
intermediate situation where $\Pi_0$ has a 4-valent vertex corresponds to
a non-maximal degeneration where $\mathcal{P}$ is no longer a maximal 
triangulation of $[0,p]\times [0,q]$, instead containing a single 
parallelogram of unit area; the mirror toric variety $Y$ then acquires 
an ordinary double point singularity. The two manners in which the
four-valent vertex of the tropical curve can be deformed to a pair of
trivalent vertices connected by a bounded edge then amount to the two
small resolutions of the ordinary double point, and differ by a flip.
\end{example}

\subsection{Hypersurfaces in abelian varieties} 
As suggested to us by Paul Seidel, the methods we use to study
hypersurfaces in toric varieties can also be applied to the case of
hypersurfaces in abelian varieties. For simplicity, we only discuss
the case of abelian varieties $V$ which can be viewed as quotients of
$(\C^*)^n$ (with its standard K\"ahler form) by the action of a real 
lattice $\Gamma_B\subset \R^n$, where
$\gamma\in\Gamma_B$ acts by $(x_1,\dots,x_n)\mapsto (e^{\gamma_1} x_1,\dots,
e^{\gamma_n} x_n)$. In other terms, the logarithm map identifies
$V$ with the product $T_B\times T_F$ of two real Lagrangian tori, the
``base'' $T_B=\R^n/\Gamma_B$ and the ``fiber'' $T_F=i\R^n/(2\pi\Z)^n$
(which corresponds to the orbit of a $T^n$-action).

Since the $T^n$-action on $V$ is not Hamiltonian, there is no globally
defined $\R^n$-valued moment map. However, there is an analogous map which
takes values in a real torus, namely the quotient of $\R^n$ by the
lattice spanned by the periods of $\omega_V$ on $H_1(T_B)\times H_1(T_F)$;
due to our choice of the standard K\"ahler form on $(\C^*)^n$, this
period lattice is simply $\Gamma_B$, and the ``moment map'' is the logarithm
map projecting from $V$ to the real torus $T_B=\R^n/\Gamma_B$.

A tropical hypersurface $\Pi_0\subset T_B$ can be thought of as the image of a
$\Gamma_B$-periodic tropical hypersurface $\tilde{\Pi}_0\subset \R^n$ under the natural
projection $\R^n\to \R^n/\Gamma_B=T_B$. Such a tropical hypersurface occurs
naturally as the limit of the amoebas (moment map images) of a degenerating
family of hypersurfaces $H_\tau$ inside the degenerating family of abelian
varieties $V_\tau$ ($\tau\to 0$) corresponding to rescaling the lattice $\Gamma_B$ by
a factor of $|\log\tau|$. (We keep the K\"ahler class
$[\omega_V]$ and its period lattice $\Gamma_B$ constant by rescaling the
K\"ahler form of $(\C^*)^n$ by an appropriate factor, so that the moment map
is given by the base $\tau$ logarithm map, $\mu_V=\Log_\tau:V_\tau\to T_B$.)
As in \S \ref{s:notations} we call $H_\tau\subset V_\tau$ ``nearly
tropical'' if its amoeba $\Pi_\tau=\Log_\tau(H_\tau)\subset T_B$ is contained in a tubular 
neighborhood of the tropical hypersurface $\Pi_0$; we place ourselves
in the nearly tropical setting, and elide $\tau$ from the notation.

Concretely, the hypersurface $H$ is defined by a section of a
line bundle $\mathcal{L}\to V$ whose pullback to $(\C^*)^n$ is trivial;
$\mathcal{L}$ can be viewed as the quotient of $(\C^*)^n\times\C$ by
$\Gamma_B$, where $\gamma\in\Gamma_B$ acts by 
\begin{equation}\label{eq:thetatriv}
\gamma_\#:(x_1,\dots,x_n,v)\mapsto
(\tau^{-\gamma_1}x_1,\dots,\tau^{-\gamma_n}x_n,\tau^{\kappa(\gamma)}\mathbf{x}^{\lambda(\gamma)}
v),
\end{equation}
where $\lambda\in\hom(\Gamma_B,\Z^n)$ is a homomorphism determined by the Chern class $c_1(\mathcal{L})$
(observe that $\hom(\Gamma_B,\Z^n)\simeq H^1(T_B,\Z)\otimes H^1(T_F,\Z)
\subset H^2(V,\Z)$), and $\kappa:\Gamma_B
\to\R$ satisfies a cocycle-type condition in order to make
\eqref{eq:thetatriv} a group action. A basis of sections of $\mathcal{L}$ is
given by the theta functions
\begin{equation}\label{eq:theta}\vartheta_\alpha(x_1,\dots,x_n)=
\sum_{\gamma\in\Gamma_B} \gamma_\#^*(\mathbf{x}^\alpha),\qquad
\alpha\in \Z^n/\lambda(\Gamma_B).
\end{equation}
(Note: for $\gamma\in \Gamma_B$, $\vartheta_\alpha$ and
$\vartheta_{\alpha+\lambda(\gamma)}$ actually differ by a constant factor.)
The defining section $f$ of $H$ is a finite linear combination of
these theta functions; equivalently, its lift to $(\C^*)^n$ can be viewed
as an infinite Laurent series of the form \eqref{eq:Ht}, invariant under
the action \eqref{eq:thetatriv} (which forces the set of weights
$A$ to be $\lambda(\Gamma_B)$-periodic.)
We note that the corresponding tropical function $\varphi:\R^n\to\R$ is also
$\Gamma_B$-equivariant, in the sense that $\varphi(\xi+\gamma)=
\varphi(\xi)+\langle\lambda(\gamma),\xi\rangle-\kappa(\gamma)$ for all
$\gamma\in\Gamma_B$.

Let $X$ be the blowup of $V\times\C$ along $H\times 0$, equipped with
an $S^1$-invariant K\"ahler form $\omega_\epsilon$ such that the fibers 
of the exceptional divisor have area $\epsilon>0$ (chosen sufficiently
small).
Denote by $\tilde{V}$ the proper transform of 
$V\times 0$, and let $X^0=X\setminus \tilde{V}$. Then $X^0$ carries an
$S^1$-invariant Lagrangian torus fibration $\pi:X^0\to B=T_B\times \R_+$,
constructed as in \S \ref{s:fibrations} by assembling fibrations on
the reduced spaces of the $S^1$-action. This allows us to determine
SYZ mirrors to $X^0$ and $X$ as in \S \ref{s:opencase} and \S
\ref{s:generalcase}. 

The construction can be understood either directly at the level of 
$X$ and $X^0$, or by viewing the whole process as a $\Gamma_B$-equivariant
construction on the cover $\tilde{X}$, namely the blowup of 
$(\C^*)^n\times\C$ along $\tilde{H}\times 0$, where $\tilde{H}$ is the
preimage of $H$ under the covering map $q:(\C^*)^n\to (\C^*)^n/\Gamma_B=V$.
The latter viewpoint makes it easier to see that the enumerative geometry
arguments from the toric case extend to this setting.

As in the toric case, each weight $\bar\alpha\in \bar{A}:=A/\lambda(\Gamma_B)$ 
determines a connected component of the complement $T_B\setminus \Pi_0$ of
the tropical hypersurface $\Pi_0$, and hence a chamber
$U_{\bar\alpha}\subset B^{reg}\subset B$ over which the fibers of $\pi$ are
tautologically unobstructed. Each of these determines an affine coordinate chart 
$U_{\bar\alpha}^\vee$ for the SYZ mirror of $X^0$, and these charts are
glued to each other via coordinate transformations of the form
\eqref{eq:gluing}. 

Alternatively, we can think of the mirror as a quotient by $\Gamma_B$ of
a space built from an infinite collection of charts $U_\alpha^\vee$, $\alpha\in A$,
where each chart $U_\alpha^\vee$ has coordinates $(v_{\alpha,1},\dots,
v_{\alpha,n},w_0)$, glued together by \eqref{eq:gluing}. Specifically, for each element
$\gamma=(\gamma_1,\dots,\gamma_n)\in\Gamma_B$, we identify $U_\alpha^\vee$ with
$U_{\alpha+\lambda(\gamma)}^\vee$ via the map
\begin{equation}\label{eq:mirroraction}
\gamma^\vee_\#:(v_{\alpha,1},\dots,v_{\alpha,n},w_0)\in U_\alpha^\vee\mapsto
(T^{\gamma_1}v_{\alpha,1},\dots,T^{\gamma_n}v_{\alpha,n},w_0)\in
U_{\alpha+\lambda(\gamma)}^\vee,
\end{equation}
where the multiplicative factors $T^{\gamma_i}$ account for the amount
of symplectic area separating the different lifts to $\tilde{X}$ of a
given fiber of $\pi$.

Setting $v_0=1+T^{-\epsilon}w_0$, we can again view the SYZ mirror $Y^0$
of $X^0$ as the complement of the hypersurface $w_0^{-1}(0)=v_0^{-1}(1)$
in a ``locally toric'' variety $Y$ covered (outside of codimension 2 strata)
by local coordinate charts $Y_\alpha=(\K^*)^n\times\K$ ($\alpha\in
A$) glued together by \eqref{eq:toriccharts} and identified under the
action of $\Gamma_B$. Namely, for all $\alpha,\beta\in A$ and $\gamma\in
\Gamma_B$ we make the identifications
\begin{eqnarray}\label{eq:abeliancharts1}
(v_{1},\dots,v_{n},v_0)\in Y_\alpha&\sim&
(v_0^{\alpha_1-\beta_1}v_{1},\dots,v_0^{\alpha_n-\beta_n}v_{n},
v_0)\in Y_\beta,\\
\label{eq:abeliancharts2}
(v_{1},\dots,v_{n},v_0)\in Y_\alpha
&\sim& (T^{\gamma_1}v_{1},\dots,T^{\gamma_n}v_{n},v_0)\in
Y_{\alpha+\lambda(\gamma)}.
\end{eqnarray}
Finally, the abelian variety $V$ is aspherical, and any
holomorphic disc bounded by $\pi^{-1}(b)$, $b\in B^{reg}$ must be entirely
contained in a fiber of the projection to $V$, so that the only
contribution to the superpotential is $w_0$ (as in the case of hypersurfaces
in $(\C^*)^n$). With this understood, our main results become:
\begin{theorem}\label{thm:mainabelian}
Let $H$ be a nearly tropical hypersurface in an abelian variety
$V$, let $X$ be the blowup of $V\times\C$ along $H\times 0$, and let
$Y$ be as above. Then:
\begin{enumerate}
\item $Y^0=Y\setminus w_0^{-1}(0)$ is SYZ mirror to $X^0=X\setminus
\tilde{V}$;
\item the $B$-side Landau-Ginzburg model $(Y^0,w_0)$ is SYZ mirror to $X$;
\item the $B$-side Landau-Ginzburg model $(Y,-v_0)$ is generalised SYZ mirror to $H$.
\end{enumerate}
\end{theorem}

\noindent
Note that, unlike Theorems \ref{thm:main} and \ref{cor:main}, this result
holds without
any restrictions: when $V$ is an abelian variety, Assumption
\ref{ass:affinecase}
always holds and there are never any higher-order instanton corrections. 
On the other hand, the statement of part (3) implicitly uses the
properties of Fukaya categories of Landau-Ginzburg models whose proofs are sketched in Section \ref{s:backtoH}
(whereas parts (1) and (2) rely only on familiar versions of the Fukaya category).

The smooth fibers of $-v_0:Y\to\K$ (or equivalently up to a reparametrization,
$w_0:Y^0\to\K^*$) are all abelian varieties, in fact quotients of
$(\K^*)^n$ (with coordinates $\mathbf{v}=(v_1,\dots,v_n)$) by the identification
$$\mathbf{v}^m\sim v_0^{\langle \lambda(\gamma),m\rangle} T^{\langle
\gamma,m\rangle} \mathbf{v}^m\qquad\text{for all }m\in\Z^n\ \text{and}\ \gamma\in
\Gamma_B,$$
while the singular fiber is a union of toric varieties 
$$v_0^{-1}(0)=\bigcup_{\bar\alpha\in
\bar{A}} D_{\bar\alpha}$$ glued (to each other or to themselves) along toric
strata. The moment polytopes for the toric varieties
$D_{\bar\alpha}$
are exactly the components of  $T_B\setminus \Pi_0$, and the tropical
hypersurface $\Pi_0$ depicts the moment map images of the codimension 2
strata of $Y$ along which they intersect.

\begin{example}
When $H$ is a set of $n$ points on an elliptic curve $V$, we find that
the fibers of $-v_0:Y\to\K$ are a family of elliptic curves, all smooth
except $v_0^{-1}(0)$ which is a union of $n$ $\PP^1$'s forming a cycle
(in the terminology of elliptic fibrations, this is known as an $I_n$
fiber). In this case the superpotential $-v_0$ has $n$ isolated critical
points, all lying in the fiber over zero, as expected.
\end{example}

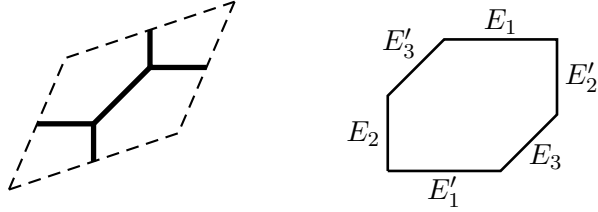
\begin{figure}
\setlength{\unitlength}{5mm}\psset{unit=\unitlength}
\begin{picture}(6,5)(-1.25,-0.75)
\psline[linewidth=2pt](-0.5,1)(1,1)(1,0)
\psline[linewidth=2pt](1,1)(2.5,2.5)
\psline[linewidth=2pt](2.5,3.5)(2.5,2.5)(4,2.5)
\psline[linestyle=dashed](-1.25,-0.75)(0.25,2.75)(4.75,4.25)
\psline[linestyle=dashed](-1.25,-0.75)(3.25,0.75)(4.75,4.25)
\end{picture}
\qquad\qquad
\begin{picture}(5,4)(-0.5,-0.5)
\psline[linewidth=1pt](0,0)(0,2)(1.5,3.5)(4.5,3.5)(4.5,1.5)(3,0)(0,0)
\put(-0.2,1){\makebox(0,0)[rc]{\small $E_2$}}
\put(4.7,2.5){\makebox(0,0)[lc]{\small $E'_2$}}
\put(1.5,-0.15){\makebox(0,0)[ct]{\small $E'_1$}}
\put(3,3.65){\makebox(0,0)[cb]{\small $E_1$}}
\put(0.75,3){\makebox(0,0)[rb]{\small $E'_3$}}
\put(3.75,0.7){\makebox(0,0)[lt]{\small $E_3$}}
\end{picture}
\caption{A tropical genus 2 curve on the 2-torus (left); the singular 
fiber of the mirror Landau-Ginzburg model is the quotient of the toric
Del Pezzo surface shown (right) by identifying $E_i\sim E'_i$.}
\label{fig:genus2abelian}
\end{figure}

\begin{example}
Now consider the case where $H$ is a genus 2 curve embedded in an abelian surface $V$ 
(for example its Jacobian torus). The tropical genus 2 curve $\Pi_0$ is a
trivalent graph on the 2-torus $T_B$ with two vertices and three edges,
see Figure~\ref{fig:genus2abelian} left. Since $T_B\setminus \Pi_0$ is
connected, the singular fiber $v_0^{-1}(0)$ of the mirror $B$-side Landau-Ginzburg
model is irreducible. Specifically,
it is obtained from the toric Del Pezzo
surface shown in Figure \ref{fig:genus2abelian} right, i.e.\ $\PP^2$ 
blown up in 3 points, by identifying each exceptional curve $E_i$ with the 
``opposite'' exceptional curve $E'_i$ (the proper
transform of the line through the two other points).
Thus the critical locus of the superpotential
is a configuration of three rational curves $E_1=E'_1$, $E_2=E'_2$,
$E_3=E'_3$ intersecting at two triple points. (Compare with \S
\ref{ss:planecurves}: the mirrors are very different, but
the critical loci are the same).
\end{example}


\section{Complete intersections}\label{ss:completeint}

In this section we explain (without details) how to extend our main results to
the case of complete intersections in toric varieties (under a
suitable positivity assumption for rational curves, which always holds
in the affine case).

\subsection{Notations and statement of the results}

Let $H_1,\dots,H_d$ be smooth nearly tropical hypersurfaces in a toric
variety $V$ of dimension $n$, in general position. We denote by $f_i$
the defining equation of $H_i$, a section of a line bundle $\mathcal{L}_i$
which can be written as a Laurent polynomial \eqref{eq:Ht} 
in affine coordinates $\mathbf{x}=(x_1,\dots,x_n)$; by $\varphi_i:\R^n\to\R$ the corresponding
tropical polynomial; and by $\Pi_i\subset\R^n$ the tropical
hypersurface defined by $\varphi_i$. (To ensure smoothness of the mirror, 
it is useful to assume that
the tropical hypersurfaces $\Pi_1,\dots,\Pi_d$ intersect transversely,
though this assumption is actually not necessary).

We denote by $X$ the blowup of $V\times\C^d$ along the $d$ codimension 2
subvarieties $H_i\times \C^{d-1}_i$, where $\C^{d-1}_i=\{y_i=0\}$ 
is the $i$-th coordinate hyperplane in $\C^d$. (The blowup is smooth since
the subvarieties $H_i\times \C^{d-1}_i$ intersect transversely). Explicitly, $X$ can be a
described as a smooth submanifold of the total space of the $(\PP^1)^d$-bundle
$\prod_{i=1}^d \PP(\mathcal{L}_i\oplus \O)$ over $V\times\C^d$, 
\begin{equation}
\label{eq:Xci}
X=\{(\mathbf{x},y_1,\dots,y_d,(u_1\!:\!v_1),\dots,(u_d\!:\!v_d))\,|\,
f_i(\mathbf{x})v_i=y_iu_i\ \forall i=1,\dots,d\}.
\end{equation}
Outside of the union of the hypersurfaces $H_i$, the fibers of the 
projection $p_V:X\to V$ obtained by composing the blowup map $p:X\to
V\times\C^d$ with projection to the first factor are isomorphic to $\C^d$;
above a point which
belongs to $k$ of the $H_i$, the fiber consists of $2^k$ components, each of which is
a product of $\C$'s and $\PP^1$'s.

The action of $T^d=(S^1)^d$ on $V\times\C^d$ by rotation on the last $d$
coordinates lifts to $X$; we equip $X$ with a $T^d$-invariant K\"ahler form
for which the exceptional $\PP^1$ fibers of the $i$-th exceptional divisor
have area $\epsilon_i$ (where $\epsilon_i>0$ is chosen small enough).
As in \S \ref{ss:blowup}, we arrange for the K\"ahler form on $X$ to
coincide with that on $V\times\C^d$ away from the exceptional divisors.
We denote by $\mu_X:X\to\R^d$ the moment map.

The dense open subset $X^0\subset X$ over which we can construct an SYZ fibration
is the complement of the proper transforms of the toric strata of
$V\times\C^d$; it can be viewed as an iterated conic bundle over the
open stratum $V^0\simeq (\C^*)^n\subset V$, namely
\begin{equation}
\label{eq:X0ci}
X^0\simeq \{(\mathbf{x},y_1,\dots,y_d,z_1,\dots,z_d)\in V^0\times
\C^{2d}\,|\, y_iz_i=f_i(\mathbf{x})\ \forall i=1,\dots,d\}.
\end{equation}

Consider the polytope $\Delta_Y\subseteq \R^{n+d}$ defined by
\begin{equation}
\label{eq:DeltaYci}
\Delta_Y=\{(\xi,\eta_1,\dots,\eta_d)\in \R^n\oplus\R^d\,|\,\eta_i\ge
\varphi(\xi_i)\ \forall i=1,\dots,d\},
\end{equation}
and let $Y$ be the corresponding toric variety. For $i=1,\dots,d$, denote by
$v_{0,i}$ the monomial with weight $(0,\dots,0,1,\dots,0)$ (the $(n+i)$-th
entry is $1$), and set
\begin{equation}
\label{eq:w0ici}
w_{0,i}=-T^{\epsilon_i}+T^{\epsilon_i}v_{0,i}.
\end{equation}

Denote by $A$ the set of connected components of $\R^n\setminus
(\Pi_1\cup\dots\cup \Pi_d)$, and index each component by the tuple of
weights $\vec\alpha=(\alpha^1,\dots,\alpha^d)\in
\Z^{n\times d}$ corresponding to the dominant monomials of
$\varphi_1,\dots,\varphi_d$ in that component.
Then for each $\vec\alpha\in A$ we have a coordinate
chart $Y_{\vec{\alpha}}\simeq (\K^*)^n\times\K^d$ with coordinates
$\mathbf{v}_{\vec\alpha}=(v_{\vec{\alpha},1},\dots,v_{\vec\alpha,n})
\in(\K^*)^n$ and $(v_{0,1},\dots,v_{0,d})\in\K^d$, where the monomial
$\mathbf{v}_{\vec{\alpha}}^m=v_{\vec\alpha,1}^{m_1}\dots v_{\vec\alpha,n}^{m_n}$
is the toric monomial with weight $(-m_1,\dots,-m_n,\langle \alpha^1,m\rangle,
\dots,\langle\alpha^d,m\rangle)\in\Z^{n+d}$. These charts glue via
\begin{equation}
\label{eq:gluingci}
\mathbf{v}_{\vec\alpha}^m=\left(\prod_{i=1}^d{(1+T^{-\epsilon_i} w_{0,i})}^{\langle
\beta^i-\alpha^i,m\rangle}\right)\,\mathbf{v}_{\vec\beta}^m.
\end{equation}

Denoting by $\sigma_1,\dots,\sigma_r\in \Z^n$ the primitive generators of
the rays of the fan $\Sigma_V$, and writing the moment polytope of $V$ in 
the form \eqref{eq:Delta_V}, for
$j=1,\dots,r$ we define
\begin{equation}
\label{eq:wjci}
w_j=T^{\varpi_j}\mathbf{v}_{\vec{\alpha}_{min}(\sigma_j)}^{\sigma_j},
\end{equation}
where $\vec{\alpha}_{min}(\sigma_j)\in A$ is chosen so that all $\langle
\sigma_j,\alpha^i\rangle$ are minimal. In other terms,
$\mathbf{v}_{\vec{\alpha}_{min}(\sigma_j)}^{\sigma_j}$ is the toric
monomial with weight
$(-\sigma_j,\lambda_1(\sigma_j),\dots,\lambda_d(\sigma_j))\in \Z^{n+d}$,
where $\lambda_1,\dots,\lambda_d:\Sigma_V\to\R$ are the piecewise linear
functions defining $\mathcal{L}_i=\O(H_i)$.

Finally, define $Y^0$ to be the subset of $Y$ where
$w_{0,1},\dots,w_{0,d}$ are all non-zero, and define the leading-order
superpotentials
\begin{equation}
\label{eq:W0ci}
W_0=w_{0,1}+\dots+w_{0,d}+w_1+\dots+w_r=\sum_{i=1}^d (-T^{\epsilon_i}+
T^{\epsilon_i}v_{0,i})\ +\ \sum_{i=1}^r
T^{\varpi_j}\mathbf{v}_{\vec{\alpha}_{min}(\sigma_j)}^{\sigma_j},
\end{equation}
\begin{equation}
\label{eq:W0Hci}
W_0^H=-v_{0,1}-\dots-v_{0,d}+w_1+\dots+w_r=
-\sum_{i=1}^d v_{0,i}\ +\ \sum_{i=1}^r
T^{\varpi_j}\mathbf{v}_{\vec{\alpha}_{min}(\sigma_j)}^{\sigma_j}.
\end{equation}
With this understood, the analogue of Theorems
\ref{thm:main}--\ref{thm:conicbundle} is the following

\begin{theorem}\label{thm:mainci} With the above notations:
\begin{enumerate}
\item $Y^0$ is SYZ mirror to the iterated conic bundle $X^0$;
\item assuming that all rational curves in $X$ have positive Chern
number (e.g.\ when $V$ is affine), the $B$-side Landau-Ginzburg model $(Y^0,W_0)$ is
SYZ mirror to~$X$;
\item assuming that $V$ is affine, the $B$-side Landau-Ginzburg model $(Y,W_0^H)$
is a generalized SYZ mirror to the complete intersection $H_1\cap\dots\cap H_d\subset V$.
\end{enumerate}
\end{theorem}
\noindent
As in Theorem \ref{thm:mainabelian}, part (3) of this theorem relies on 
the expected properties of Fukaya categories of Landau-Ginzburg models.

\begin{remark}
Denoting by $X_i$ the blowup of $V\times\C$ at $H_i\times 0$ and by $X_i^0$
the corresponding conic bundle over $V^0$, the space $X$ (resp.\ $X^0$) is
the fiber product of $X_1,\dots,X_d$ (resp.\ $X^0_1,\dots,X^0_d$) with
respect to the natural projections to $V$. This perspective explains many
of the geometric features of the construction.
\end{remark}

\subsection{Sketch of proof} The argument proceeds along the same lines as
for the case of hypersurfaces, of which it is really a straightforward
adaptation. We outline the key steps for the reader's convenience.

As in \S \ref{s:fibrations}, a key observation to be made about the
$T^d$-action on $X$ is that the reduced spaces
$X_{red,\lambda}=\mu_X^{-1}(\lambda)/T^d$ ($\lambda\in\R_{\ge 0}^d$) are
all isomorphic to $V$ via the projection $p_V$ (though the K\"ahler forms
may differ near $H_1\cup\dots\cup H_d$). This allows us to build a (singular)
Lagrangian torus fibration
$$\pi:X^0\to B=\R^n\times (\R_+)^d$$
(where the second component is the moment map) by assembling standard 
Lagrangian torus fibrations on the reduced spaces. The singular fibers of
$\pi$ correspond to the points of $X^0$ where the $T^d$-action is not free;
therefore $$B^{sing}=\bigcup_{i=1}^d \ \Pi'_i\times
\{(\lambda_1,\dots,\lambda_d)\,|\,\lambda_i=\epsilon_i\},$$
where $\Pi'_i\subset \R^n$ is essentially the amoeba of $H_i$.
The potentially obstructed fibers of $\pi:X^0\to B$ are precisely those that intersect
$p_V^{-1}(H_1\cup\dots\cup H_d)$, and for each $\vec\alpha\in A$ we have
an open subset $U_{\vec\alpha}\subset B$ of tautologically unobstructed fibers which
project under $p$ to standard product tori in $V^0\times\C^d$.

Each of the components $U_{\vec\alpha}\subset B$ determines an affine
coordinate chart $U_{\vec\alpha}^\vee$ in the SYZ mirror to $X^0$.
Namely, for $b\in U_{\vec\alpha}\subset B$, the Lagrangian torus 
$L=\pi^{-1}(b)\subset X^0$ is the preimage by $p$ of a
standard product torus in $V\times\C^d$. Denoting by 
$(\zeta_1,\dots,\zeta_n,\lambda_1,\dots,\lambda_d)\in \Delta_V\times\R_+^d$ 
the corresponding value of the moment map of $V\times\C^d$, and by $(\gamma_1,\dots,\gamma_n,
\gamma_{0,1},\dots,\gamma_{0,d})$ the natural basis of $H_1(L,\Z)$, we equip
$U_{\vec\alpha}^\vee$ with the coordinate system
\begin{multline}\label{eq:affinechartci}
(L,\nabla)\mapsto (v_{\vec\alpha,1},\dots,v_{\vec\alpha,n},w_{0,1},
\dots,w_{0,d})\\:=\left(T^{\zeta_1}\nabla(\gamma_1),\dots,T^{\zeta_n}
\nabla(\gamma_n),T^{\lambda_1}\nabla(\gamma_{0,1}),\dots,T^{\lambda_d}
\nabla(\gamma_{0,d})\right).
\end{multline}
For $b\in U_{\vec\alpha}$, the Maslov index 2 holomorphic discs bounded by 
$L=\pi^{-1}(b)$ in $X$ can 
be determined explicitly as in \S \ref{s:opencase}, by projecting to
$V\times\C^d$. Specifically, these
discs intersect the proper transform of exactly one of the toric divisors
transversely in a single point, and there are two cases:

\begin{lemma}
For any $i=1,\dots,d$, $L$ bounds a
unique family of Maslov index $2$ holomorphic discs in $X$ which 
intersect the proper transform of\/ $V\times\C^{d-1}_i=\{y_i=0\}$ transversely
in a single point; the images
of these discs under $p$ are contained in lines parallel to the $y_i$ 
coordinate axis, and their contribution to the superpotential is $w_{0,i}$. 
\end{lemma}

\begin{lemma}
For any
$j=1,\dots,r$, denote by $D_{\sigma_j}$ the toric divisor in $V$ associated
to the ray $\sigma_j$ of the fan $\Sigma_V$, and let $k_i=\langle 
\alpha^i-\alpha^i_{min}(\sigma_j),\sigma_j\rangle$ $(i=1,\dots,d)$. Then
$L$ bounds $2^{k_1+\dots+k_d}$ families of
Maslov index $2$ holomorphic discs in $X$ which 
intersect the proper transform of\/ $D_{\sigma_j}\times\C^d$ transversely
in a single point (all of which have the same projections to $V$),
and their total contribution to the superpotential is
$$\left(\prod_{i=1}^d (1+T^{-\epsilon_i} w_{0,i})^{k_i}\right)\,T^{\varpi_i}
\mathbf{v}_{\vec\alpha}^{\sigma_j}.$$
\end{lemma}

\noindent
The proofs are essentially identical to those of
Lemmas \ref{l:countw0} and \ref{l:Wsigma}, and left to the reader.
As in \S \ref{s:opencase}, the first lemma implies that the coordinates
$w_{0,i}$ agree on all charts $U_{\vec\alpha}^\vee$, and the second one
implies that the coordinates $v_{\vec\alpha,i}$ transform according
to \eqref{eq:gluingci}. The first two statements in Theorem \ref{thm:mainci}
follow.

The last statement in the theorem follows from equipping $X$ with the
superpotential $W^\vee=y_1+\dots+y_d:X\to\C$, which has Morse-Bott
singularities along the intersection of the proper transform of $V\times 0$
with the $d$ exceptional divisors, i.e.\ $\mathrm{crit}(W^\vee)\simeq
H_1\cap\dots\cap H_d$. As in \S \ref{s:backtoH}, the nontriviality of
the normal bundle forces us to twist the Fukaya category of $(X,W^\vee)$
by a background class $s\in H^2(X,\Z/2)$, in this case Poincar\'e dual to 
the sum of the exceptional divisors (or equivalently to the sum of the
proper transforms of the toric divisors $V\times \C^{d-1}_i$). The
thimble construction then provides a fully faithful $A_\infty$-functor
from $\F(H_1\cap\dots\cap H_d)$ to $\F_s(X,W^\vee)$. The twisting
affects the superpotential by changing the signs of the terms $w_{0,1},
\dots,w_{0,d}$. Moreover, the thimble functor modifies 
the value of the superpotential by an additive constant, which equals
$T^{\epsilon_1}+\dots+T^{\epsilon_d}$ when $V$ is affine (the $i$-th term 
corresponds to a family of small discs of area $\epsilon_i$ in the normal
direction to $H_i$). Putting everything together, the result follows
by a straightforward adaptation of the arguments in \S \ref{s:backtoH}.

\appendix

\section{Moduli of objects in the Fukaya category} \label{sec:moduli-objects-fukay}

\subsection{General theory}
\label{sec:general-theory}

Let $L$ be an embedded spin Lagrangian of vanishing Maslov class in the K\"ahler manifold $X^0=X\setminus D$, where $D$ is an anticanonical divisor which satisfies Assumption \ref{ass:nef}. We begin with a brief overview of the results of \cite{Fcyclic}, which in part implement the constructions of \cite{FO3book} in the setting of de Rham cohomology.

For each positive real number $E$, Fukaya defines a \emph{curved} $A_{\infty}$ structure on the de Rham cochains with coefficients in $\Lambda_{0}/ T^{E}$, which we denote by
\begin{equation*}
  \Omega^{*}(L; \Lambda_{0}/ T^{E} \Lambda_0) \equiv \Omega^{*}(L; \R) \otimes_{\R} \Lambda_{0}/ T^{E}\Lambda_0.
\end{equation*}
The operations are obtained from the moduli space of holomorphic discs in $X^0=X \setminus D$ with boundary on $L$, whose energy is bounded by $E$. 
By induction, one obtains an unbounded sequence of real numbers $E_i$, together with formal diffeomorphisms on
$\Omega^{*}(L; \Lambda_{0}/ T^{E_{i}} \Lambda_0)$
which pull back the $A_{\infty}$ structure constructed from discs of energy bounded by $E_i$ to the projection of the  $A_{\infty}$ structure on $    \Omega^{*}(L; \Lambda_{0}/ T^{E_{i+1}} \Lambda_{0}) $ modulo $T^{E_{i}}$. After applying such a formal diffeomorphism, we may therefore assume that the $A_{\infty}$ map
\begin{equation*}
 \Omega^{*}(L; \Lambda_{0}/ T^{E_{i+1}} \Lambda_{0}) \to  \Omega^{*}(L; \Lambda_{0}/ T^{E_{i}} \Lambda_{0}) 
\end{equation*}
is defined by projection of coefficient rings. Taking the inverse limit over $E_i$, we obtain an $A_{\infty}$ structure on
$\Omega^{*}(L; \Lambda_0)$.
By passing to the canonical model (i.e. applying a filtered version of the homological perturbation lemma \cite{Kad}), we can reduce this $A_{\infty}$ structure to $H^{*}(L; \Lambda_0)$.

Fukaya checks that any class $b \in H^{1}(L; U_{\Lambda})$ defines a deformed $A_{\infty}$ structure on the cohomology. In particular, there is a subset
\begin{equation*}
\hat{ \mathcal{Y}}_{L} \subset H^{1}(L; U_{\Lambda})
\end{equation*}
consisting of elements for which this $A_{\infty}$ structure has vanishing curvature (i.e.  solutions to the Maurer-Cartan equation). Gauge transformations \cite[Section 4.3]{FO3book}  define an equivalence relation on this set; we call the quotient  the \emph{moduli space of simple objects supported on $L$,} which we denote $\mathcal{Y}_{L} $. 
\begin{remark}
 The original formalism of Fukaya, Oh, Ohta, and Ono \cite{FO3book} considered deformation classes corresponding to $b \in H^{1}(L; \Lambda_+)$, called bounding cochains, which  via exponentiation $\Lambda_+ \to 1+\Lambda_+ $ can also be reinterpreted as local systems. As noted in the discussion following Theorem 1.2 of \cite{Fcyclic}, there are inclusions $1+ \Lambda_+  \subset U_{\Lambda} \subset \Lambda^{*}$, and the original construction of Floer cohomology can be generalised to all unitary local systems using an idea of Cho. 
\end{remark}

The invariance statement of Floer cohomology \cite[Theorem 14.1-14.3]{FO3book} asserts that $\mathcal{Y}_{L} $  does not depend on the choice of auxiliary data (e.g almost-complex structure) in the following sense: let $ \mathcal{Y}^{1}_{L} $ and $ \mathcal{Y}^{2}_{L}  $ denote the moduli spaces for different choices of auxiliary structures.  A homotopy between the auxiliary data  induces an isomorphism
\begin{equation} \label{eq:invariance_Floer}
   \mathcal{Y}^{1}_{L} \cong \mathcal{Y}^{2}_{L}
\end{equation}
which is invariant under homotopies of homotopies. 
\begin{assumption} \label{ass:trivial_differential}
The $A_{\infty}$ structure on $ H^{*}(L; \Lambda_0)$ is isomorphic to the undeformed structure.
\end{assumption}
\begin{remark}
For most Lagrangians that we consider, this condition holds automatically because there is a choice of almost complex structure for which the Lagrangian bounds no holomorphic discs. 
\end{remark}

In this setting, the Maurer-Cartan equation vanishes identically, and the gauge equivalence relation is trivial. We therefore obtain an identification of the moduli space $\mathcal{Y}_{L} $ of simple objects of the Fukaya category supported on $L$ with its first cohomology with coefficients in $U_{\Lambda}$:
\begin{equation*}
 \mathcal{Y}_{L} \equiv H^{1}(L; U_{\Lambda}).
\end{equation*}

Let $L_{t}$ be a Hamiltonian path of Lagrangians in $X^0$ with vanishing Maslov class, and $J_{t}$ a family of almost complex structures on $X$ which we assume are fixed at infinity. 
We describe the isomorphism \eqref{eq:invariance_Floer} in the special situation which we consider in this paper. We first identify $H_1(L_0;\Z) \cong H_{1}(L_t; \Z)$ via the given path. A basis for this group yields an identification
\begin{equation*}
   (z_1, \ldots, z_n)  : H^{1}(L_0; U_{\Lambda}) \to   U_{\Lambda}^{n}.
\end{equation*}
\begin{assumption} \label{ass:all_curves_in_fixed_class}
For the family $(L_{t}, J_{t})$, all stable holomorphic discs represent multiples of a given relative homology class $\beta \in H_{2}(X, L_0; \Z)$.
\end{assumption}
The wall-crossing map is then of the form
\begin{equation} \label{eq:transformation_Hamiltonian_isotopy}
 z_{i} \mapsto h_i(z_{\beta}) z_i, 
\end{equation}
where $h_i$ is a power series with $\Q$ coefficients and leading order term equal to $1$, and $z_{\beta}$ denotes the monomial
$T^{\omega(\beta)} z^{[\partial \beta]}$.
Equation \eqref{eq:transformation_Hamiltonian_isotopy} can be extracted from the construction in Section 11 of \cite{Fcyclic}. For an explicit derivation, see \cite [Lemma 4.4]{Tu}: for bounding cochains, the transformation corresponds to adding a power series in $z_{\beta}$ with vanishing constant term, and Equation \eqref{eq:transformation_Hamiltonian_isotopy} follows by exponentiation.

By Proposition \ref{prop:gluing}, the following assumption holds in the geometric setting of the main theorem:
\begin{assumption} \label{ass:h-polynomial}
The power series $h_i$ is the expansion of a rational function in $z_{\beta}$.
\end{assumption}
In this case, the transformation in  Equation \eqref{eq:transformation_Hamiltonian_isotopy} converges away from the zeroes and poles of $h_i$. This is stronger than the general result proved by Fukaya namely that the transformation converges in an analytic neighbourhood of the unitary elements in $H^{1}(L; \Lambda^{*}) $.

In order to extend this construction to the non-Hamiltonian setting, we use the main construction of \cite{Fcyclic} which identifies the moduli space of simple objects supported on Lagrangians near $L$ (but not necessarily Hamiltonian isotopic to it) with an affinoid domain in $H^{1}(L; \Lambda^{*})$ in the sense of Tate. 

Given a path $\{ L_{t} \}_{t \in [0,1]}$ between Lagrangians $L_0$ and $L_1$ in which there is no wall crossing (e.g. so that no Lagrangian in the family bounds a holomorphic disc), the natural gluing map between these domains is obtained from the flux homomorphism 
\begin{equation*}
\Phi(\{ L_{t} \}) \in  H^{1}(L_0; \R)  
\end{equation*}
and the product on cohomology groups
\begin{equation*} 
 H^{1}(L_0; \R)  \times H^{1}(L_0; \Lambda^{*}) \to  H^{1}(L_0; \Lambda^{*})
\end{equation*}
induced by the map on coefficients $(\lambda, f) \mapsto T^{\lambda} f$. 
In the absence of wall crossing we identify $H^{1}(L_1; \Lambda^{*})$ with $H^1(L_0;\Lambda^*)$ via this rescaling map.

Given a general path between Lagrangians $L_0$ and $L_1$ (subject to
Assumptions \ref{ass:all_curves_in_fixed_class} and \ref{ass:h-polynomial}), this identification
is modified by the wall crossing formula given in Equation \eqref{eq:transformation_Hamiltonian_isotopy}, yielding a birational map
\begin{equation*}
  H^{1}(L_0; \Lambda^{*}) \dashrightarrow H^{1}(L_1; \Lambda^{*}),
\end{equation*}
defined away from a hypersurface.  We glue the moduli spaces of objects supported near $L_{0}$ and $L_{1}$ using this identification. 

\begin{remark} \label{rem:non-unitary_reduction}
The construction of a map for a Lagrangian path can be reduced to the case of Hamiltonian paths as follows: any path $(L_{t}, J)$ can be deformed, with fixed endpoints, to a path $(L_{t}', J_{t})$ which is a concatenation of paths for which the Lagrangian is constant and paths in which there is no wall-crossing. The desired  map is then obtained as a composition of the  wall-crossing maps for Hamiltonian paths  and the rescalings given by the flux homomorphism.

The idea for constructing the deformed path follows the main strategy for proving convergence in \cite{Fcyclic}.  Whenever $\epsilon$ is sufficiently small, there is a (compactly supported) diffeomorphism $\psi_{\epsilon}$ taking $L_{t}$ to $L_{t+\epsilon}$  which preserves the tameness of $J$. For tautological reasons, there is a path without wall-crossing from $(L_t,J)$ to $(L_{t+\epsilon}, J_{t+\epsilon})$ if $J_{t+\epsilon}$ is the pullback of $J$ by $\psi_{\epsilon}$. Interpolating between this pullback and $(L_{t+\epsilon},J)$, via pullbacks of $(L_{t+s},J)$, we then reach $(L_{t+\epsilon},J)$ via a path for which the Lagrangian is constant and Assumption \ref{ass:all_curves_in_fixed_class} remains satisfied.
\end{remark}

\begin{remark}\label{rmk:wcf_generalized}
(1) More generally, given a path from $L_0$ to $L_1$ that can 
be decomposed into finitely many sub-paths $\{L_t\}_{t\in [t_j,t_{j+1}]}$,
each satisfying Assumption \ref{ass:all_curves_in_fixed_class} for some 
relative class $\beta_j$, we again obtain a wall-crossing 
map
\begin{equation}\label{eq:wcf_birational}
H^{1}(L_0; \Lambda^{*}) \dashrightarrow H^{1}(L_1; \Lambda^{*})
\end{equation}
 by composing the maps associated to the various sub-paths.

(2) When all the classes $\beta_j$ have the same
boundary in $H_1(L_t,\Z)$ and the same symplectic areas, the monomials
$z_{\beta_j}$ are all equal and the birational transformation
\eqref{eq:wcf_birational} again takes
the form of Equation \eqref{eq:transformation_Hamiltonian_isotopy} up to 
rescaling of the coefficients.
\end{remark}

If we restrict attention to the smooth fibers of a Lagrangian torus fibration, we obtain an embedding of the moduli space  $\mathcal{Y}_{\pi}^{0}$  of all simple objects supported on such Lagrangians into the rigid analytic space
\begin{equation} \label{eq:analytic_space_glue_all_tori}
\coprod   H^{1}(L; \Lambda^{*})/\sim
\end{equation}
where the equivalence relation identifies points which correspond to each other under the birational wall-crossing transformations of Equation \eqref{eq:wcf_birational} induced by all paths among smooth fibres. 
It does not automatically follow from the above considerations that this quotient is a well-behaved (e.g.\ separated) analytic space, but in our case this will not be an issue.
By the invariance of Floer cohomology \cite[Theorem 14.1-14.3]{FO3book}, the transformations induced by homotopic paths are equal. The fact that these transformations should in general depend only on the homotopy class of the path in the space of all fibres (i.e.\ allowing fibres which are not necessarily embedded), is expected to follow as a consequence of forthcoming developments in the study of family Floer cohomology in the presence of singular fibres. In our main example, this independence will be manifest from Proposition \ref{prop:gluing}, and the quotient \eqref{eq:analytic_space_glue_all_tori} can easily be seen to be a smooth analytic space.

\begin{remark}
We can think of \eqref{eq:analytic_space_glue_all_tori}  as the natural (analytic) completion of $\mathcal{Y}_{\pi}^{0}$.  While the points of this completion do not necessarily correspond to unitary local systems on Lagrangians in $X^0$ with the given K\"ahler form, in good situations, they can be interpreted as Lagrangians in $X^0$ equipped with a completed K\"ahler form. Slightly strengthening Assumption \ref{ass:nef} by requiring that $X^0$ be the complement of a nef divisor, we can obtain such a completion by inflation along the divisor at infinity.
\end{remark}

It shall be convenient for our purposes to consider a completion which is obtained by gluing only finitely many charts. To this end, assume that $\{L_{t}\}_{t \in [0,1]}$ is a path of Lagrangians so that the wall-crossing map defines an embedding
\begin{equation}
   H^{1}(L_0; U_{\Lambda}) \hookrightarrow  H^{1}(L_1; \Lambda^{*}).
\end{equation}
In this case, the above construction yields that all elements of $\mathcal{Y}_{L_0} $ can be represented in Equation \eqref{eq:analytic_space_glue_all_tori} by elements of $H^{1}(L_1; \Lambda^{*})$. 

More generally, assume that $\{ L_\alpha\}_{\alpha \in A}$ is a collection of fibers with the property that for some fixed almost complex structure $J$,  any smooth fiber $L$ can be connected to some fiber $L_{\alpha}$ in our collection by a path such that the wall-crossing map defines an embedding $H^{1}(L; U_{\Lambda}) \to  H^{1}(L_\alpha; \Lambda^{*})$. We define
\begin{equation}
   \hat{\mathcal{Y}}_{\pi}^{0} \equiv \coprod_{\alpha \in A}  H^{1}(L_{\alpha}; \Lambda^{*})/\sim.
\end{equation}
\begin{lemma} \label{lem:finitely_many_charts}
There is a natural analytic embedding of
$\mathcal{Y}_{\pi}^{0}$ into $\hat{\mathcal{Y}}_{\pi}^{0}$.
\qed
\end{lemma}

Next, we study the moduli spaces of holomorphic discs in $X$ with boundary on a Lagrangian $L \subset X^0$ of vanishing Maslov class.  Since $D$ is an anticanonical divisor, stable holomorphic discs whose intersection number with $D$ is $1$ have Maslov index equal to $2$. Assumption \ref{ass:nef} implies that there are no discs of negative Maslov index, and that those of vanishing Maslov index are disjoint from $D$. For each unitary local system $\nabla$ on $L$, choice of almost complex structure $J$,  and action cutoff $E$ we obtain a $\Lambda_0/ T^{E} \Lambda_{0}$-valued de Rham cochain
\begin{equation}
  \sum_{\substack{\beta\in \pi_2(X,L)\\ \beta\cdot D=1}} z_\beta(L,\nabla)\, \mathrm{ev}_*[\mathcal{M}_1(L,\beta,J)]\in \Omega^{0}(L;\Lambda_0/ T^{E} \Lambda_{0})
\end{equation}
which is closed with respect to the Floer differential. Passing to the canonical model and to the inverse limit over $E$  we obtain a multiple of the unit in the self-Floer cohomology of $(L,\nabla)$:
\begin{equation}
  \m_0(L,\nabla,J)=W(L,\nabla,J)\,e_{L} \in H^{0}(L; \Lambda).
\end{equation}

Since the moduli spaces of discs of vanishing Maslov index in $X$ and in $X \setminus D$ agree, the invariance of Floer theory and in particular of the potential function \cite[Theorem B]{FO3book}, as extended to non-unitary local systems in \cite{Fcyclic}, implies that $W(L,\nabla,J)$ gives rise to a well-defined convergent function on $\mathcal{Y}_{\pi}^{0}$.  Because of this, we shall henceforth drop $J$ from the notation.  For non-unitary local systems, $W(L,\nabla)$ may not in general converge, so we have to impose this as an additional assumption. With this in mind, the proof of the following result follows from the unitary case by Remark \ref{rem:non-unitary_reduction}.
\begin{lemma} \label{l:potential_matched_by_gluing}
  If for each $\alpha \in A$, the map $ \nabla \mapsto W(L_\alpha,\nabla)$ converges on $H^{1}(L_{\alpha}; \Lambda^{*})$, then $W$ defines a regular function
  on $\hat{\mathcal{Y}}_{\pi}^{0}$. \qed
\end{lemma}

We record the following consequence:
\begin{corollary} \label{cor:potential_matched_by_gluing}
  If $(L_i,\nabla_i)$ and $(L_j, \nabla_j)$ are identified by a wall-crossing gluing map, then
  $W(L_i,\nabla_i) = W(L_j,\nabla_j)$.
 \qed
\end{corollary}

\begin{remark}
Fukaya has announced that rank $1$ unitary local systems on immersed Lagrangians which are fibers of $\pi$ define a rigid analytic space which includes $ \hat{\mathcal{Y}}_{\pi}^{0}$ as an analytic subset. The general idea is to describe the nearby smooth fibers as the result of Lagrangian surgery, and understand the behaviour of holomorphic discs under such surgeries sufficiently explicitly to produce an analytic structure on this neighbourhood which can be seen to be compatible with the analytic structure on  $\hat{\mathcal{Y}}_{\pi}^{0}$.

We expect that, in the presence of a potential function, similar ideas can be applied to associate analytic charts to certain admissible non-compact Lagrangians arising as limits of smooth fibers. While we do not develop the general theory in this paper, Example \ref{ex:FScompact} explains how one can use equivalences in the Fukaya category (rather than surgery formulae) to produce the desired charts in the class of examples we encounter.
\end{remark}

\subsection{Convergence of the wall-crossing}
\label{sec:conv-wall-cross}

In this section, we verify that the assumptions of Lemma \ref{lem:finitely_many_charts} hold for the smooth fibers of the map
$\pi : X^{0} \to B$
introduced in Definition \ref{def:Lagr-torus-fibration}. Recall that the
moment map $\mu_X$ of the $S^1$-action descends to a natural map from
$B$ to $\R_+$; we write $X^0_\lambda=\mu_X^{-1}(\lambda)\cap X^0$.
If $\epsilon$ is the blowup parameter in the definition of $X$, then all 
fibers of $\pi$ contained in $X^{0}_{\lambda}$ are smooth whenever $\lambda \neq 
\epsilon$; and the smooth fibers in $X^{0}_{\epsilon}$ are exactly those whose image under 
the blowdown map 
$p : X^{0} \to V^0 \times \C$ is disjoint from $H \times \C$.

Assumption \ref{ass:trivial_differential} follows immediately from 
Proposition \ref{prop:unobstructed}  for all fibers of $\pi$ whose images 
under $p$ are disjoint from $H \times \C$, since these bound no holomorphic discs. 
In general, invariance of Floer cohomology shows that Assumption \ref{ass:trivial_differential}
is independent of the choice of almost complex structure. Moreover,
the identification of the $A_\infty$ structure obtained by deforming by an element in $H^{1}(L; \Lambda_+) $
with the deformed Floer theory for the associated local system in $ H^{1}(L; 1 + \Lambda_+)$ 
implies that  Assumption \ref{ass:trivial_differential} holds for the Floer theory of $L$ equipped with unitary local systems as well,
since an analytic function vanishing on $1 + \Lambda_+ $  must vanish on all of $U_{\Lambda}$. The same argument
shows that the $A_{\infty}$ structure on $L$ equipped with a non-unitary local system is also undeformed, as long as the valuation is sufficiently small. By Fukaya's work on Family Floer cohomology \cite{Fcyclic}, we conclude that the $A_{\infty}$ structure on a Lagrangian fibre $L'$ sufficiently close to $L$ is undeformed. Here, sufficiently close means that there is a diffeomorphism preserving the tameness of $J$ and moving $L$ to $L'$; in compact subsets of the space of smooth fibres, there are uniform bounds on the size of such neighbourhoods, so we conclude that the condition of having undeformed $A_{\infty}$ structure is open and closed among smooth fibres of $\pi$. Therefore, all smooth fibres of $\pi$ satisfy Assumption \ref{ass:trivial_differential}.

We next choose Lagrangians $\{ L_{\alpha} \}_{\alpha \in A}$, labelled by the monomials in the equation defining the hypersurface $H$. We require that $L_{\alpha}$ be contained in $X^{0}_{\epsilon}$, and that its projection to $B$ lie in the chamber $U_\alpha \subset B$ (see Definition \ref{def:chamber}).

\begin{lemma} \label{lem:conv-wall-cross}
Any smooth fiber $L$ of $\pi$ can be connected to some fiber $L_{\alpha}$ so that the wall-crossing map defines an embedding
\begin{equation}
       H^{1}(L; U_{\Lambda}) \to  H^{1}(L_{\alpha}; \Lambda^{*}).
\end{equation}
\end{lemma}
\begin{proof}
  There are two cases to consider:

 {\bf Case 1:} Assume that the smooth fiber $L$ lies in $X^{0}_{\epsilon}$.
  Then $\pi_\epsilon(L)$ lies outside of the amoeba of $H$ 
  (cf.\ Equation \eqref{eq:Bsing}) and $L$ is tautologically unobstructed
  (cf.\ Proposition \ref{prop:unobstructed}). By Remark \ref{rmk:Bregretracts}, 
  the component of the complement of the amoeba over which $L$ lies
  determines a chamber $U_{\alpha}$, and $L$ can be connected to $L_\alpha$
  by a path of tautologically unobstructed fibers. 
  The absence of holomorphic discs in this region implies that there are no non-trivial walls, and hence that the map
  \begin{equation}
     H^{1}(L; \Lambda^{*}) \to  H^{1}(L_{\alpha}; \Lambda^{*})
  \end{equation}
is given simply by a rescaling of the coefficients (see the discussion following Equation \eqref{eq:transformation_Hamiltonian_isotopy}). This completes the argument in this case.

{\bf Case 2:} Assume that $L$ lies in $X^{0}_{\lambda}$, with $\lambda \neq \epsilon$. Choose a smooth fiber $L^{\lambda}_{\alpha}$ which is also contained in $X^{0}_{\lambda}$ and whose projection lies in some chamber $U_{\alpha}$, and consider the concatenation of a path from $L$ to $L^{\lambda}_{\alpha}$ via Lagrangians contained in $X^{0}_{\lambda}$ with a path from $L^{\lambda}_{\alpha}$ to $L_{\alpha}$ over the chamber $U_{\alpha}$. 
Since the map associated to the latter path is a simple rescaling as in the previous case, it suffices to show convergence of the wall-crossing map for the path
from $L$ to $L_\alpha^\lambda$. 

To this end, recall from Proposition \ref{prop:unobstructed} that the
simple holomorphic discs bounded by the Lagrangian torus fibers along the 
path all have the same area $|\lambda - \epsilon|$ and their boundaries
all represent the same homology class in $H_1(L_t,\Z)$. Thus, the monomials
$z_\beta=T^{\omega(\beta)}z^{[\partial \beta]}$ associated to their
homology classes are all equal, and
by Remark \ref{rmk:wcf_generalized}~(2)
the wall crossing map is of the form
\begin{equation}\label{eq:wcf_again}
    z_{i} \mapsto h_i(z_\beta) z_i, 
\end{equation}
where $h_i$ is a power series with coefficients in $\Q$ and leading order term
equal to~1. Whenever we evaluate at a point of $H^1(L;U_\Lambda)$,
the valuation of $z_\beta$ is $|\lambda-\epsilon|>0$, and so
$h_i(z_\beta)$ and its inverse both converge and take values in 
$U_\Lambda$. Thus the leading order term of
\eqref{eq:wcf_again} is identity, and the wall-crossing map defines an embedding
  \begin{equation*}
     H^{1}(L; U_{\Lambda}) \hookrightarrow  H^{1}(L_{\alpha}^{\lambda}; \Lambda^{*}).
  \end{equation*}
Composing this map with the rescaling isomorphism
induced by the flux homomorphism of a path over $U_{\alpha}$, we arrive at the desired result.
\end{proof}

\section{The geometry of the reduced spaces}\label{s:reducedappendix}

In this section we study in more detail the symplectic geometry of the
reduced spaces $X_{red,\lambda}=\mu_X^{-1}(\lambda)/S^1$ 
and prove Lemma \ref{l:rectifyomegared}.

Recall from \S \ref{ss:redspaces} that the moment map for the $S^1$-action
on $X$ is given by \eqref{eq:muX}, and that the only fixed
points apart from $\tilde{V}=\mu_X^{-1}(0)$ occur along $\tilde{H}$, which
lies in the level set $\mu_X^{-1}(\epsilon)$. Also recall that, for all
$\lambda>0$, the natural projection to $V$ (obtained by composing $p:X\to
V\times\C$ with projection to the first factor) yields a natural
identification of $X_{red,\lambda}$ with $V$.

We will exploit the toric structure of $V$ to construct families of
Lagrangian tori in $X_{red,\lambda}$ equipped with the reduced K\"ahler form
$\omega_{red,\lambda}$. The two obstacles are (1) the lack of
smoothness along $H$ at $\lambda=\epsilon$, and (2) the lack of
$T^n$-invariance near $H$.

We start with the first issue, giving a formula for $\omega_{red,\lambda}$
near $\tilde{H}$ and introducing an explicit family of smoothings.
Consider a small neighborhood of 
$\tilde{H}$ where, without loss of generality, we may assume that
$\chi\equiv 1$.

\begin{lemma} \label{l:omegared}
Identifying $X_{red,\lambda}$ with $V$ as above, 
where $\chi\equiv 1$ we have
\begin{equation}\label{eq:omegared}
\omega_{red,\lambda}=\omega_V-\max(0,\epsilon-\lambda)\,c_1(\mathcal{L})+d\alpha_{0,\lambda},
\end{equation}
where $c_1(\mathcal{L})=iF_\mathcal{L}/2\pi$ is the Chern form of the chosen Hermitian metric
on $\mathcal{L}$, and \begin{equation}\label{eq:singular_form}
\alpha_{0,\lambda}=\frac{\min(\lambda,\epsilon)\,d^c(|f(\x)|^2)}{
2\left(\sqrt{4\pi\epsilon |f(\x)|^2+(\lambda-\epsilon+\pi|f(\x)|^2)^2}+
\pi|f(\x)|^2+|\lambda-\epsilon|\right).}
\end{equation}
\end{lemma}

\proof Recall that away from $\tilde{V}$ we can write $X$ as a conic
bundle $f(\x)=yz$. Where $f\neq 0$ and $\chi\equiv 1$, the restriction of
$\omega_\epsilon$ to $\mu_X^{-1}(\lambda)$ is equal to
$$p_V^*\omega_V+d\left(\frac14 |y|^2 d^c(\log |y|^2)+
\frac{\epsilon}{4\pi}\frac{|z|^2}{1+|z|^2}d^c(\log |z|^2)\right).
$$
Since $d^c\log |y|^2+d^c\log |z|^2=d^c\log |f|^2$, using
\eqref{eq:muX_explicit} we can rewrite the 1-form in this expression as
either $$\frac14 |y|^2 d^c(\log
|f|^2)+\frac{\epsilon-\lambda}{4\pi}d^c(\log |z|^2)\quad\text{or}\quad
\frac{\epsilon}{4\pi}\frac{|z|^2}{1+|z|^2}d^c(\log
|f|^2)+\frac{\lambda-\epsilon}{4\pi} d^c(\log |y|^2).$$
Now $dd^c\log |y|^2=0$, whereas $dd^c\log |z|^2=-4\pi p_V^*c_1(\mathcal{L})$,
so we find that (still where $f\neq 0$ and $\chi\equiv 1$)
\begin{align}\label{eq:omegaeps_restricted}
(\omega_\epsilon)_{|\mu_X^{-1}(\lambda)}&=p_V^*\bigl(\omega_V+(\lambda-\epsilon)c_1(\mathcal{L})\bigr)
+d\left(\frac{d^c(|f(\x)|^2)}{4|z|^2}\right)\\
\notag &=p_V^*\omega_V+d\left(\frac{\epsilon}{4\pi}
\frac{d^c(|f(\x)|^2)}{|y|^2+|f(\x)|^2}\right).
\end{align}
The first expression makes sense wherever $z\neq 0$, in particular for
$\lambda<\epsilon$; the second one makes sense wherever $y\neq 0$, in
particular for $\lambda>\epsilon$. 
Solving \eqref{eq:muX_explicit} for $|y|$, we obtain
\begin{align*}
&2\pi|y|^2=\sqrt{4\pi\epsilon|f(\x)|^2+(\lambda-\epsilon+\pi|f(\x)|^2)^2}
-\pi|f(\x)|^2+(\lambda-\epsilon),\\
\text{and}\quad&
2\lambda|z|^2=\sqrt{4\pi\epsilon|f(\x)|^2+(\lambda-\epsilon+\pi|f(\x)|^2)^2}
+\pi|f(\x)|^2-(\lambda-\epsilon).
\end{align*}
Substituting into \eqref{eq:omegaeps_restricted} gives the desired
expression.
\endproof

We can smooth the singularity of $\omega_{red,\lambda}$ by considering
the modified K\"ahler forms given near $H$ by
$$\omega_{sm,\lambda}=\omega_V-\max(0,\epsilon-\lambda)\,c_1(\mathcal{L})+d\alpha_{\kappa,\lambda}$$
where $\kappa>0$ is an arbitrarily small constant, and 
\begin{equation}\label{eq:smoothed_form}
\alpha_{t,\lambda}=\frac{\min(\lambda,\epsilon)\,d^c(|f(\x)|^2)}{
2\left(\sqrt{4\pi\epsilon |f(\x)|^2+(\lambda-\epsilon+\pi|f(\x)|^2)^2+t^2\tilde\chi}+
\pi|f(\x)|^2+|\lambda-\epsilon|\right),}
\end{equation}
where $\tilde\chi=\tilde\chi(|f(\x)|,\lambda)$ is a suitable cut-off function
which equals 1 near $\tilde{H}$ and vanishes outside of the region where 
$\chi\equiv 1$. (We can also assume that $\tilde\chi$ vanishes whenever $\lambda$ is not close to $\epsilon$.)
We set $\omega_{sm,\lambda}=\omega_{red,\lambda}$ wherever $\chi\neq 1$.
Choosing $\kappa$ small enough ensures that $\omega_V-\max(0,\epsilon-\lambda)\,c_1(\mathcal{L})+d\alpha_{t,\lambda}$ is non-degenerate for all $t\in
[0,\kappa]$; it is then a K\"ahler form, because $\alpha_{t,\lambda}$ can be written as $d^c$ of some 
function of $|f(\x)|$.

The K\"ahler forms $\omega_{sm,\lambda}$ are all smooth, coincide
with $\omega_{red,\lambda}$ away from $H$ for all $\lambda$, and 
everywhere when $\lambda$ is not very close to $\epsilon$. Moreover,
$[\omega_{sm,\lambda}]=[\omega_{red,\lambda}]$ by construction, and
the dependence of $\omega_{sm,\lambda}$ on $\lambda$ is piecewise smooth.

Like $\omega_{red,\lambda}$, the K\"ahler form $\omega_{sm,\lambda}$ 
is not invariant under the given torus action, but there exist
toric K\"ahler forms in the same cohomology class. Such a K\"ahler form
$\omega'_{V,\lambda}$ can be constructed by averaging $\omega_{sm,\lambda}$
with respect to the standard $T^n$-action on $V$:
\begin{equation}\label{eq:omega'Vlambda}
\omega'_{V,\lambda}=\frac{1}{(2\pi)^n}\int_{g\in T^n}
g^*\omega_{sm,\lambda}\,dg.
\end{equation}
To see that the $T^n$-orbits are Lagrangian with respect to
$\omega'_{V,\lambda}$, we note that the pullback of $\omega_{sm,\lambda}$ 
to an orbit represents the trivial cohomology class, since the classes 
$[\omega_V]$ and $[H]$ are both trivial on a torus fibre. The pullback 
of $\omega'_{V,\lambda}$ is therefore also trivial in cohomology, but 
since it is invariant, it must vanish pointwise. This in turn implies that
the $T^n$-action not only preserves $\omega'_{V,\lambda}$ but in fact it is
Hamiltonian.

We now state again Lemma \ref{l:rectifyomegared} and give its proof:

\begin{lemma}
There exists a family of homeomorphisms $(\phi_{\lambda})_{\lambda\in
\R_+}$ of\/ $V$ such that: 
\begin{enumerate}
\item $\phi_\lambda$ preserves the toric divisor
$D_V\subset V$;
\item the restriction of $\phi_\lambda$ to $V^0$ is a 
diffeomorphism for $\lambda\neq \epsilon$, and a diffeomorphism outside of $H$
for $\lambda=\epsilon$;
\item $\phi_\lambda$ intertwines the reduced
K\"ahler form $\omega_{red,\lambda}$ and the toric K\"ahler form
$\omega'_{V,\lambda}$;
\item $\phi_\lambda=\mathrm{id}$ at every point whose $T^n$-orbit is
disjoint from the support of $\chi$;
\item $\phi_\lambda$ depends on $\lambda$ in a continuous manner, and
smoothly except at $\lambda=\epsilon$.
\end{enumerate}
\end{lemma}

\proof We proceed in two stages, obtaining $\phi_\lambda$ as the composition
of two maps $\phi_{sm,\lambda}$, taking $\omega_{red,\lambda}$ to $\omega_{sm,\lambda}$, and
$\phi_{avg,\lambda}$ taking $\omega_{sm,\lambda}$ to $\omega'_{V,\lambda}$,
each satisfying all the other conditions in the statement. The
arguments are quite similar in both cases; we start with the construction of 
$\phi_{avg,\lambda}$ (Steps 1--2), then proceed with $\phi_{sm,\lambda}$
(Steps 3--4).
\medskip

{\bf Step 1.} Let $\beta_\lambda=\omega_{sm,\lambda}-\omega'_{V,\lambda}$. Since
$\omega'_{V,\lambda}$ is $T^n$-invariant, for
$\theta\in \mathfrak{t}^n\simeq \R^n$ we have
\begin{align*}\exp(\theta)^*\omega_{sm,\lambda}-\omega_{sm,\lambda}=
\exp(\theta)^*\beta_\lambda-\beta_\lambda&=\int_0^1 \frac{d}{dt}
\left(\exp(t\theta)^*\beta_\lambda\right)\,dt\\&=d\left[\int_0^1 
\exp(t\theta)^*\left(\iota_{\theta_\#}\beta_\lambda\right)\,dt\right].
\end{align*}
Hence, averaging over all elements of $T^n$, we see that the 1-form
$$a'_\lambda=\frac{1}{(2\pi)^n}\int_{[-\pi,\pi]^n}\int_0^1 \exp(t\theta)^*
\left(\iota_{\theta_\#}\beta_\lambda\right)\,dt\,d\theta$$
satisfies $\omega'_{V,\lambda}-\omega_{sm,\lambda}=da'_\lambda$ (i.e.,
$da'_\lambda=-\beta_\lambda$).

Let $U\subset V$ be the orbit of the support of $\chi$ under the 
standard $T^n$-action on $X_{red,\lambda}\cong V$. Outside of $U$, 
the K\"ahler forms $\omega_{sm,\lambda}=\omega_{red,\lambda}$ are $T^n$-invariant, and
$\omega_{sm,\lambda}$ and $\omega'_{V,\lambda}$ coincide (in fact
they both coincide with $\omega_V$). Therefore, $\beta_\lambda$ is supported
in $U$, and consequently so is $a'_\lambda$.

Let $\omega'_{t,\lambda}=t\omega'_{V,\lambda}+(1-t)\omega_{sm,\lambda}$ (for $t\in
[0,1]$ these are
K\"ahler forms since $\omega'_{V,\lambda}$ and $\omega_{sm,\lambda}$ are
K\"ahler). Denote
by $v_t$ the vector field such that
$\iota_{v_t}\omega'_{t,\lambda}=-a'_\lambda$ and by $\psi_t$
the flow generated by $v_t$. Then by Moser's trick,
$$\frac{d}{dt}(\psi_t^*\omega'_{t,\lambda})=\psi_t^*\left(L_{v_t}\omega'_{t,\lambda}+\frac{d\omega'_{t,\lambda}}{dt}\right)=
\psi_t^*(d\iota_{v_t}\omega'_{t,\lambda}+da'_\lambda)=0,$$
so $\psi_t^*\omega'_{t,\lambda}=\omega_{sm,\lambda}$, and the time 1 flow
$\psi_1$ intertwines $\omega_{sm,\lambda}$ and
$\omega'_{V,\lambda}$ as desired. Moreover, because $a'_\lambda$ is supported in
$U$, outside of $U$ we have $\psi_t=\mathrm{id}$. However, it is not clear
that the flow preserves the toric divisors of $V$.
\medskip

{\bf Step 2.}
To remedy the issue with the flow not preserving the toric divisors, we modify $a'_\lambda$ in a neighborhood of $D_V$.
Let $f'_{\lambda,t}$ be a family of $C^1$ 
real-valued functions (with locally Lipschitz first derivatives), smooth on $V^0$, with the following properties:
\begin{itemize}
\item the support of $f'_{\lambda,t}$ is contained in the intersection of $U$ with a small
tubular neighborhood of $D_V$;
\item at every point $x\in D_V$, belonging to a toric stratum $S\subset V$,
\begin{equation}\label{eq:stratumcond}
\text{the 1-form $a'_\lambda+df'_{\lambda,t}$ vanishes on $(T_xS)^\perp$,} 
\end{equation}
where the orthogonal is with respect to $\omega'_{t,\lambda}$;
\item $f'_{\lambda,t}$ depends smoothly on $t$, and piecewise smoothly on
$\lambda$.
\end{itemize}
We construct $f'_{\lambda,t}$ by induction over toric strata of increasing 
dimension, successively constructing functions $f'_{\lambda,t,\le k}:V\to\R$ which 
satisfy \eqref{eq:stratumcond} for all strata of dimension at most~$k$ and
are smooth outside of strata of dimension $<k$.
We start by setting $f'_{\lambda,t,\le 0}=0$, which satisfies \eqref{eq:stratumcond}
at the fixed points of the torus action since they lie away from the support 
of $a'_\lambda$. 

Assume $f'_{\lambda,t,\le k}$ constructed, and consider a
stratum $S$ of dimension $k+1$. At each point $x\in S$, 
the restriction of $a'_\lambda+df'_{\lambda,t,\le k}$ to $(T_xS)^\perp$ 
is a real-valued linear form, vanishing whenever $x$ belongs to a
lower-dimensional stratum, and smooth outside of strata of dimension $<k$.
Let $f'^0_{\lambda,t,S}$ be a $C^{1}$ function on a neighborhood of $S$,
smooth outside of the strata of dimension $\le k$, which 
vanishes on $S$ and whose derivative in the normal directions at each point of $S$ satisfies
$(df'^0_{\lambda,t,S})_{|(T_xS)^\perp} = -(a'_\lambda+df'_{\lambda,t,\le k})_{|(T_xS)^\perp}.$
(For instance, identify a neighborhood of $S$ with a subset of its normal bundle in
a manner compatible with the toric structure, and take $f'^0_{\lambda,t,S}$ to
be linear in the fibers).

Let $\chi_S$ be a cut-off function with values in
$[0,1]$, defined and smooth outside of the strata of dimension $\le k$, 
equal to 1 at all points of a neighborhood of $S$ which are much closer to $S$ than to any other 
$(k+1)$-dimensional stratum, and with support disjoint from those of the
corresponding cut-off functions for all other $(k+1)$-dimensional strata.
Specifically, picking an auxiliary metric, we take $\chi_S$ to be the product of a 
standard smooth cut-off function supported in a tubular neighborhood of $S$ with
functions $\chi_{S/\Sigma}$ for all strata $\Sigma$ with $\dim \Sigma\ge
k+1$ and $\dim(\Sigma\cap S)\le k$,
chosen so that $\chi_{S/\Sigma}$ equals 1 except near $\Sigma$, where it depends 
on the ratio between distance to $S$ and distance to $\Sigma$, equals 1
at all points that lie much closer to $S$ than to $\Sigma$, and vanishes
at all points that lie closer to $\Sigma$ than to $S$. 

We note that near a lower-dimensional stratum $S'$, the norm of 
$d\chi_S$ is bounded by a constant over distance to $S'$.
We then set $f'_{\lambda,t,S}=\chi_S f'^0_{\lambda,t,S}$. By construction, this function is
smooth away from strata of dimension $\le k$. Moreover, near a lower-dimensional
stratum $S'$, $f'^0_{\lambda,t,S}$ is bounded by a constant multiple of 
distance to $S$ times distance to $S'$, so the regularity of $f'_{\lambda,t,S}$ is
indeed as desired.

By construction, $f'_{\lambda,t,\le k+1}=f'_{\lambda,t,\le k}+\sum_{\dim S=k+1}
f'_{\lambda,t,S}$ has the desired properties on all strata of dimension $\le k+1$. 
(Note that, since $a'_\lambda$ vanishes outside of $U$, so do the various functions 
we construct.) Finally, we let $f'_{\lambda,t}=f'_{\lambda,t,\le n-1}$.

We now use Moser's trick again, replacing $a'_\lambda$ by 
$\tilde{a}'_{t,\lambda}=a'_\lambda+df'_{\lambda,t}$. Namely, denote by
$\tilde{v}_{t,\lambda}$ the vector field such that
$\iota_{\tilde{v}_{t,\lambda}}\omega'_{t,\lambda}=-\tilde{a}_{t,\lambda}$.
This vector field is locally Lipschitz continuous along $D_V$, and smooth
on $V^0$; moreover, by construction it is supported in $U$ and, by
\eqref{eq:stratumcond}, tangent to
each stratum of $D_V$. 
We thus obtain $\phi_{avg,\lambda}$ with all the desired properties
by considering the time 1 flow generated by $\tilde{v}_{t,\lambda}$.
(Note: because we have assumed that $\omega_V$ defines a complete K\"ahler
metric on $V$, it is easy to check that even when $V$ is noncompact the
time 1 flow is well-defined.)

{\bf Step 3.} We now turn to the construction of $\phi_{sm,\lambda}$.
We interpolate between $\omega_{red,\lambda}$ and $\omega_{sm,\lambda}$ via
the family of K\"ahler forms $\omega_{t,\lambda}$, $t\in [0,\kappa]$, defined
by $$\omega_{t,\lambda}=\omega_V-\max(0,\epsilon-\lambda)\,c_1(\mathcal{L})+d\alpha_{t,\lambda}$$ where
$\chi\equiv 1$ (where $\alpha_{t,\lambda}$ is given by \eqref{eq:smoothed_form})
and $\omega_{t,\lambda}=\omega_{red,\lambda}$ wherever $\chi\neq 1$.

These K\"ahler forms are smooth whenever $t>0$ or $\lambda\neq \epsilon$. 
Let $a_{t,\lambda}$ be the 1-form with support contained in the region where $\chi\equiv 1$, 
and defined by $a_{t,\lambda}=d\alpha_{t,\lambda}/dt$ inside that region.
By construction, $d\omega_{t,\lambda}/dt=da_{t,\lambda}$.
We use Moser's trick again, and
denote by $v_{t,\lambda}$ the vector field such that 
$\iota_{v_{t,\lambda}}\omega_{t,\lambda}=-a_{t,\lambda}$.
This vector field vanishes outside of $U$, and is smooth except for
$t=0$ and $\lambda=\epsilon$, in which case it is singular along $H$.
We will momentarily check that the flow of $v_{t,\lambda}$
is well-defined even for $\lambda=\epsilon$; the time $\kappa$ flow
then intertwines $\omega_{red,\lambda}$ and $\omega_{sm,\lambda}$ as
desired, except it need not preserve the toric divisors of $V$, an issue
which we will address in Step 4 below.

Differentiating
\eqref{eq:smoothed_form} with respect to $t$, we have
\begin{equation}\label{eq:smoothing_1form}
a_{t,\lambda}=\frac{t\,\tilde{\chi}\min(\lambda,\epsilon)\,d^c(|f(\x)|^2)}{
2 \sqrt{\Phi}\bigl(\sqrt{\Phi}+\pi |f(\x)|^2+|\lambda-\epsilon|\bigr)^2},
\end{equation}
where 
\begin{equation}\label{eq:Phi}
\Phi=4\pi \epsilon |f(\x)|^2+(\lambda-\epsilon+\pi|f(\x)|^2)^2+t^2\tilde\chi.
\end{equation}
Taking the dual vector field, we find that
\begin{equation}\label{eq:smoothing_vect_field}
v_{t,\lambda}=\frac{t\,\tilde{\chi}\min(\lambda,\epsilon)\,\nabla^{t,\lambda}(|f(\x)|^2)}{
2 \sqrt{\Phi}\bigl(\sqrt{\Phi}+\pi |f(\x)|^2+|\lambda-\epsilon|\bigr)^2},
\end{equation}
where $\nabla^{t,\lambda}$ is the gradient with respect to the K\"ahler metric
determined by $\omega_{t,\lambda}$.

We restrict our attention to the neighborhood of $\tilde{H}$ where 
$\tilde{\chi}\equiv 1$, since it is clear that $v_{t,\lambda}$ is well-defined
and smooth everywhere else. To estimate the norm of $\nabla^{t,\lambda}(|f(\x)|^2)$,
we differentiate \eqref{eq:smoothed_form} to find that, in this region, 
\begin{multline}\label{eq:explicit_dalphat}
d\alpha_{t,\lambda}=\frac{2\min(\lambda,\epsilon)
\left(\pi(\epsilon+\lambda)|f|^2+(\lambda-\epsilon)^2+t^2+|\lambda-\epsilon|\sqrt\Phi\right)
d|f|\wedge d^c|f|}{\sqrt\Phi\bigl(\sqrt\Phi+\pi|f|^2+|\lambda-\epsilon|\bigr)^2}\\
-\frac{2\pi\min(\lambda,\epsilon)|f|^2\, c_1(\mathcal{L})}{\bigl(\sqrt\Phi+\pi|f|^2+|\lambda-\epsilon|\bigr)}.
\end{multline}
(Here we have used the fact that $dd^c|f|^2=-4\pi |f|^2c_1(\mathcal{L})+
4d|f|\wedge d^c|f|$.)

When $\lambda-\epsilon$ and $|f(\x)|^2$ are much smaller than $\epsilon$, we have
$\Phi\sim 4\pi\epsilon|f|^2+(\lambda-\epsilon)^2+t^2$. Estimating the
various terms in \eqref{eq:explicit_dalphat}, we find that the second term tends
to zero near $H$, while the leading-order part of the coefficient of $d|f|\wedge d^c|f|$ 
is bounded from below by $\epsilon/\sqrt\Phi$ (and from above by $4\epsilon/\sqrt\Phi$). Hence:
\begin{equation}\label{eq:formblowsup}
d\alpha_{t,\lambda}\gtrsim \frac{\epsilon}{\sqrt\Phi}\,d|f|\wedge d^c|f|.
\end{equation}
(where $\gtrsim$ means that the inequality holds up to lower-order terms.)
In more geometric terms, the K\"ahler metrics induced by $\omega_{t,\lambda}$
blow up in the normal direction to $H$, by an amount of the order of
$\epsilon/\sqrt\Phi$, while remaining well-behaved in the other directions.

This implies in turn that the norms of $d(|f(\x)|^2)$ and 
$\nabla^{t,\lambda}(|f(\x)|^2)$ with respect to the
K\"ahler metric $\omega_{t,\lambda}$ are
bounded by $2(\sqrt\Phi/\epsilon)^{1/2}|f(\x)|$;
and, more importantly, the norm
of $\nabla^{t,\lambda}(|f(\x)|^2)$ with respect to a suitable fixed auxiliary
metric is locally bounded by a constant multiple of $(\sqrt\Phi/\epsilon)|f(\x)|$.
Plugging into \eqref{eq:smoothing_vect_field}, we conclude that the norm of $v_{t,\lambda}$ (again with
respect to a smooth auxiliary metric) is bounded by a constant multiple
of $t|f(\x)|/\Phi\le t|f(\x)|/(t^2 + 4\pi\epsilon |f(\x)|^2)$, and hence 
uniformly bounded.
Thus, even though $v_{t,\lambda}$ itself is not continuous at $(t,\lambda,|f(\x)|)=(0,\epsilon,0)$,
its flow is well-defined and continuous even for $\lambda=\epsilon$, and depends continuously on~$\lambda$.

Geometrically, for $\lambda-\epsilon$ sufficiently small, near $H$ the
leading-order term in $v_{t,\lambda}$ points radially away from $H$, in
the same direction as the gradient of $|f(\x)|$ with respect to $\omega_V$,
and the time $t$ flow rescales the radial coordinate $r=|f(\x)|$ in a 
suitable manner. A complicated explicit formula for the leading-order term
of the rescaling can be obtained by comparing the K\"ahler areas of small 
discs in the direction normal to $H$; for example, for $\lambda=\epsilon$ 
one finds that the time $t$ 
flow maps points where $|f(\x)|=r_0$ to points where
$|f(\x)|^2\approx \frac12 r_0(r_0+(r_0^2+\frac{1}{\pi\epsilon}t^2)^{1/2}).$
\medskip

{\bf Step 4.} We now modify the flow constructed in Step 3 in order to arrange
for the toric divisors of $V$ to be preserved. We proceed as in Step 2,
i.e.\ we replace the 1-forms $a_{t,\lambda}$ used in Step 3 with 
$a_{t,\lambda}+df_{t,\lambda}$ for carefully constructed real-valued
functions $f_{t,\lambda}$, smooth on $V^0$ except for $(t,\lambda)=(0,\epsilon)$,
such that:
\begin{itemize}
\item the support of $f_{t,\lambda}$ is contained in the intersection of $U$ with a small
tubular neighborhood of $D_V$;
\item at every point $x\in D_V$, belonging to a toric stratum $S\subset V$,
\begin{equation}\label{eq:stratumcond2}
\text{the 1-form $a_{t,\lambda}+df_{t,\lambda}$ vanishes on $(T_xS)^\perp$,} 
\end{equation}
where the orthogonal is with respect to $\omega_{t,\lambda}$;
\item where it is smooth, $f_{t,\lambda}$ depends smoothly on $t$, and piecewise smoothly on
$\lambda$.
\end{itemize}
We construct $f_{t,\lambda}$ inductively to satisfy \eqref{eq:stratumcond2}
on toric strata of increasing dimension, by exactly the same method as in
Step 2. The main new difficulty is that we need to control the behavior
of $f_{t,\lambda}$ near $H$ for $(t,\lambda)$ close to $(0,\epsilon)$.

We begin with a geometric digression.
Fix a collection of smooth foliations $\F_S$ of neighborhoods of $H\cap S$ in
$V$ for all toric strata $S\subset V$, with the following properties:
\begin{itemize}
\item each leaf of $\mathcal{F}_S$ intersects $S$ transversely at a single
point;
\item $|f|$ is constant on the leaves; in particular the leaves through
$H\cap S$ are contained in $H$;
\item given two strata $S'\subset S$, the leaves of $\mathcal{F}_{S'}$ are
unions of leaves of $\mathcal{F}_S$.
\item given two strata $S$ and $\Sigma$ which intersect transversely along a stratum
$S'=S\cap \Sigma$, the leaves of $\mathcal{F}_S$ through $S'$ foliate
$\Sigma$.
\end{itemize}

\noindent
The existence of $\mathcal{F}_S$ with these properties follows from
the transversality of $H$ to all toric strata. Indeed, near a
$k$-dimensional stratum $S'$ and away from all lower-dimensional strata, 
consider a toric chart of the form $(\C^*)^k\times \C^{n-k}$, and
modify the first $k$ coordinates (in a $C^\infty$ manner) so that, near
$H$, $|f|$ only depends on these coordinates, without changing the remaining 
$n-k$ coordinates. 
Each stratum $S\supset S'$ is then defined by the vanishing of a certain 
subset of the last $n-k$ coordinates; we choose the leaves of $\F_S$ to be
given by letting these coordinates vary and fixing all others.
(More globally, start from a collection of toric charts identifying
neighborhoods of strata with toric vector bundles over them, and modify the
bundle structures compatibly along $H$ so that $|f|$ is constant in the
fibers and the strata containing a given one remain given by distinguished
sub-bundles.)

Henceforth, unless stated otherwise,
all estimates (on distances, derivatives, etc.) are with respect to a fixed reference metric (independent
of $t$ and $\lambda$), rather than the metric $g_{t,\lambda}$ determined
by $\omega_{t,\lambda}$; and the notation $O(\dots)$ means that an 
inequality holds up to a constant factor which is uniformly bounded independently
of $t$ and $\lambda$ over any compact subset of $V$. 

Recall that $\omega_{t,\lambda}$ blows up (by a factor of the order of
$\epsilon/\sqrt\Phi$, cf.\ \eqref{eq:formblowsup}) in the directions
transverse to the complex hyperplane field
$$\Theta=\mathrm{Ker}\,(d|f|)\cap \mathrm{Ker}\,(d^c|f|).$$
In what follows, we will often have better estimates on derivatives along $\Theta$ than on
arbitrary derivatives. We will call derivatives of order $(\ell,m)$, denoted
by $D^{(\ell,m)}(\dots)$, the derivatives of order $\ell+m$ along $\ell$
vector fields tangent to $\Theta$ and $m$ arbitrary vector fields. 
Since the hyperplane distribution $\Theta$ is not 
integrable, estimates on higher derivatives in the direction of $\Theta$ 
only make sense up to lower-order derivatives along the level sets of $|f|$;
however, the curvature of $\Theta$ is $O(|f|^2)$, and the estimates we
will obtain below on derivatives of order $(\ell+2,m)$ will generally be
no better than $O(|f|^2)$ times the bounds on derivatives of order 
$(\ell,m+1)$.

Along a stratum $S$, denote by $\pi^S_{t,\lambda}:TV_{|S}\to TS^\perp$
the orthogonal projection (with respect to $\omega_{t,\lambda})$.
Because $S$ is transverse to $H$, and hence to $\Theta$ near $H$, 
the behavior of $\omega_{t,\lambda}$ in the directions transverse to
$\Theta$ implies that,
near $H\cap S$, the $\omega_{t,\lambda}$-orthogonal to $S$
becomes nearly tangent to $\Theta$ for $(t,\lambda)$ close to $(0,\epsilon)$.
Specifically, near $H\cap S$, the maximum angle (with respect to a fixed
reference metric) between a unit vector in 
$TS^\perp$ and $\Theta$ is $O(\epsilon^{-1}\sqrt\Phi)$. Thus, denoting by
$(\pi^S_{t,\lambda})^{\parallel}$ and $(\pi^S_{t,\lambda})^\bot$ the components 
of $\pi^S_{t,\lambda}$ along $\Theta$ and its orthogonal for
the reference metric,
pointwise we have $(\pi^S_{t,\lambda})^{\bot}=O(\epsilon^{-1}\sqrt\Phi)$.
This in turns implies that
$$\left|d^c(|f|^2)\circ
\pi^S_{t,\lambda}\right|=O\bigl(\epsilon^{-1}|f|\sqrt\Phi\bigr).$$ 

Along the level sets of $|f|$, the coefficient of $d|f|\wedge d^c|f|$ in
\eqref{eq:explicit_dalphat} remains constant, and so the geometric behavior
of the $\omega_{t,\lambda}$-orthogonals $TS^\perp$ can be controlled
uniformly. In particular, the derivatives along $\Theta$ of 
$(\pi^S_{t,\lambda})^{\bot}$  are bounded by $O(\sqrt\Phi)$ to all orders.
On the other hand, the variation of \eqref{eq:explicit_dalphat} in the directions
transverse to the level sets of $|f|$ implies that each derivative in those directions
worsens the bounds by a factor of $1/\sqrt\Phi$. We conclude that 
$D^{(\ell,m)}((\pi^S_{t,\lambda})^\bot)=O(\Phi^{(1-m)/2})$.
Meanwhile, by a similar reasoning,
$D^{(\ell,m)}((\pi^S_{t,\lambda})^\parallel)=O(\Phi^{-m/2})$.

These estimates on $\pi^S_{t,\lambda}$ (and the inequality
$|f|\le (\Phi/4\pi\epsilon)^{1/2}$) in turn imply that
$$D^{(\ell,m)}\left(d^c(|f|^2)\circ
\pi^S_{t,\lambda}\right)=O\left(\Phi^{(2-m)/2}\right).$$
Thus, the 1-form $a_{t,\lambda}$ from Step 3 satisfies
$$\left|a_{t,\lambda}\circ \pi^S_{t,\lambda}\right|=
\frac{t\tilde\chi \min(\lambda,\epsilon)\,\bigl|d^c(|f|^2)\circ
\pi^S_{t,\lambda}\bigr|}{
2\sqrt\Phi\bigl(\sqrt\Phi+\pi|f|^2+|\lambda-\epsilon|\bigr)^2}=O\left(\frac{t|f|}{\Phi}\right)
=O\left(\frac{t}{\sqrt\Phi}\right)$$
$$\text{and}\quad D^{(\ell,m)}\left( a_{t,\lambda}\circ \pi^S_{t,\lambda}\right)=
O\left(\frac{t}{\Phi^{(m+1)/2}}\right).$$

We now return to our main construction. Starting with $f_{\lambda,t,\le
0}=0$ as before, assume that we have already constructed $f_{\lambda,t,\le k}$, supported
in a neighborhood of the intersection of $H$ with the toric strata of dimension $\le k$, in such 
a way that \eqref{eq:stratumcond2} holds for all strata of dimension $\le
k$. We further require that, away from all strata of dimension $\le k-1$,
resp.\ near a stratum $S'$ of dimension $\le k-1$ (and assuming $S'$ is the
closest such stratum),
\begin{equation}\label{eq:stratumcond_der}
D^{(\ell,m)}(f_{t,\lambda,\le k})=O\left(\frac{t}{\Phi^{(m+1)/2}}\right),
\quad \text{resp.}\ 
O\left(\frac{t}{\Phi^{(m+1)/2}}\mathrm{dist}_{S'}^{\min(0,2-\ell-m)}\right),
\end{equation}
where $\mathrm{dist}_{S'}$ is the distance to $S'$ with respect to the fixed
reference metric.

Let $S$ be a stratum of dimension $k+1$. 
The above estimates on the derivatives of $\pi^S_{t,\lambda}$,
together with \eqref{eq:stratumcond_der}, imply that
at any point of $S$ which lies away
from the strata of dimension $\le k-1$, 
resp.\ near (and closest to) such 
a stratum $S'$,
\begin{equation}\label{eq:stratumcond_der2}
D^{(\ell,m)}\left((a_{t,\lambda}+df_{t,\lambda,\le k})\circ
\pi^S_{t,\lambda}\right)=O\left(\frac{t}{\Phi^{(m+1)/2}}\right),
\ \ \text{resp.}\ 
O\left(\frac{t}{\Phi^{(m+1)/2}}\mathrm{dist}_{S'}^{\min(0,1-\ell-m)}\right)\!.
\end{equation}
(Note that, while the quantity in \eqref{eq:stratumcond_der2} involves
an additional derivative of $f_{t,\lambda,\le k}$, the extra factor of
$\Phi^{-1/2}$ when this derivative is taken in a direction transverse to
$\Theta$ is offset by the factor of $\Phi^{1/2}$ in the estimates
for the transverse component of $\pi^S_{t,\lambda}$.)

Near a stratum $S'\subset S$ with $\dim S'\le k$, 
condition \eqref{eq:stratumcond2} for $f_{t,\lambda,\le k}$ along $S'$
implies that
$(a_{t,\lambda}+df_{t,\lambda,\le k})\circ \pi^S_{t,\lambda}$
vanishes along $S'$. Since $\Theta$ is transverse to
$S'$, \eqref{eq:stratumcond_der2} for $(\ell,m)=(1,0)$ in turn 
implies that, at all points of $S$ which lie near $S'$,
\begin{equation}\label{eq:stratumcond_der3a}
\bigl|(a_{t,\lambda}+df_{t,\lambda,\le k})\circ \pi^S_{t,\lambda}\bigr|=
O\Bigl(\frac{t\,\mathrm{dist}_{S'}}{\sqrt\Phi}\Bigr),
\end{equation}
Meanwhile, since the distance to the nearest $k$-dimensional stratum is 
no greater than the distance to the nearest lower-dimensional stratum, 
the bounds in the second part of \eqref{eq:stratumcond_der2} also hold 
when $\dim S'=k$. Hence, at any point of $S$ which lies near (and closest)
to a stratum $S'\subset S$ of dimension $\le k$,
\begin{equation}\label{eq:stratumcond_der3}
D^{(\ell,m)}\left((a_{t,\lambda}+df_{t,\lambda,\le k})\circ
\pi^S_{t,\lambda}\right)=
O\left(\frac{t}{\Phi^{(m+1)/2}}\mathrm{dist}_{S'}^{1-\ell-m}\right)\!.
\end{equation}

Define a function $f^0_{\lambda,t,S}$ on a neighborhood of the given
$(k+1)$-dimensional stratum $S$,
smooth outside of the leaves of $\F_S$ through strata of dimension $\le k-1$ (and $H$ if $(\lambda,t)=(\epsilon,0)$), 
which vanishes on $S$ and whose derivative at each point of $S$ satisfies
\begin{equation}\label{eq:df0ltS}
df^0_{\lambda,t,S}=-(a_{t,\lambda}+df_{\lambda,t,\le k})\circ \pi^S_{t,\lambda}.
\end{equation}
Specifically, we identify the leaves of $\F_S$ with open subsets in the
fibers of the normal bundle to $S$, and take $f^0_{\lambda,t,S}$ to be 
linear in the fibers. We then define $f_{\lambda,t,S}=\chi_S f^0_{\lambda,t,S}$, where
$\chi_S$ is the same cut-off function as in Step 2.

By construction, $f^0_{t,\lambda,S}=O(t\,\mathrm{dist}_S/\sqrt\Phi)$.
Moreover, using \eqref{eq:stratumcond_der3a}, along the leaf of $\F_S$
through a point $x\in S$ which lies near a lower-dimensional stratum $S'$ we have
$f^0_{t,\lambda,S}=O(t\,\mathrm{dist}_{S'}(x)\,\mathrm{dist}_S/\sqrt\Phi)$.

The derivative of $f^0_{\lambda,t,S}$ along the leaves of $\F_S$ is
the constant extension of \eqref{eq:df0ltS} along $\F_S$; whereas its
derivative in the directions transverse to $\F_S$ is a cross-term which
grows linearly with distance to $S$ and involves the dependence of
\eqref{eq:df0ltS} on the point of $S$. Moreover,
the leaves of $\F_S$ are tangent to the level sets of $|f|$ near $H$, and
hence nearly tangent to $\Theta$: the maximum angle between vectors
in $T\F_S$ and $\Theta$ is $O(|f|)$.  It then follows from
\eqref{eq:stratumcond_der2} that, away from $(k-1)$-dimensional strata,
\begin{equation}\label{eq:est_f0}
D^{(\ell,m)}(f^0_{\lambda,t,S})=O\left(\frac{t}{\Phi^{(m+1)/2}}\right).
\end{equation}
Meanwhile, along the leaf of $\F_S$ through a point $x\in S$ which lies
near (and closest to) a stratum $S'\subset S$ with $\dim S'\le k$,
\eqref{eq:stratumcond_der3} implies that
$$D^{(\ell,m)}(f^0_{\lambda,t,S})=O\left(\frac{t}{\Phi^{(m+1)/2}}
\left(\mathrm{dist}_{S'}(x)^{2-\ell-m}+
\mathrm{dist}_{S'}(x)^{1-\ell-m}\,\mathrm{dist}_{S}(\cdot)\right)\right)\,.$$
The leaf of $\F_S$ through $x$ locally stays close to a
leaf through $S'$, which by construction is contained in some other stratum
of $D_V$. In particular, as soon as the distance to $S$ is sufficiently
large compared to $\mathrm{dist}_{S'}(x)$, points on the leaf through $x$ lie 
closer to some other stratum $\Sigma$ of dimension $\ge k+1$ (and not
containing $S$) than to $S$, and so the cut-off function $\chi_S$ vanishes 
identically.  Thus, over the support of $\chi_S$, $\mathrm{dist}_{S'}(\cdot)$
and $\mathrm{dist}_{S'}(x)$ are within bounded factors of each other. Since
$\mathrm{dist}_S\le \mathrm{dist}_{S'}$, we conclude that, at all points
of the support of $\chi_S$ which lie near (and closest to) $S'$,
\begin{equation}\label{eq:est_f0_near_S'}
D^{(\ell,m)}(f^0_{\lambda,t,S})=O\left(\frac{t}{\Phi^{(m+1)/2}}
\mathrm{dist}_{S'}^{2-\ell-m}\right)\,.
\end{equation}

Now we observe that the derivatives of the cut-off function $\chi_S$ are
$O(1)$ away from strata of dimension $\le k$, and near a stratum
$S'\subset S$ of dimension $\le k$ the derivatives of order $r$ are
$O(1/\mathrm{dist}_{S'}^{r})$. Thus, \eqref{eq:est_f0} and
\eqref{eq:est_f0_near_S'} imply that away from $k$-dimensional strata,
resp.\ near (and closest to) $S'\subset S$ with $\dim S'\le k$,
\begin{equation}\label{eq:est_fS}
D^{(\ell,m)}(f_{\lambda,t,S})=O\left(\frac{t}{\Phi^{(m+1)/2}}\right),
\quad \text{resp.}\ 
O\left(\frac{t}{\Phi^{(m+1)/2}} \mathrm{dist}_{S'}^{2-\ell-m}\right)\,.
\end{equation}

We now set $$f_{t,\lambda,\le k+1}=f_{t,\lambda,\le
k}+\textstyle\sum\limits_{\dim S=k+1}
f_{t,\lambda,S}.$$ 
By construction, $f_{t,\lambda,\le k+1}$ is supported 
in a neighborhood of the intersection of $H$ with the strata of dimension 
at most $k+1$, and satisfies
\eqref{eq:stratumcond2} for all strata of dimension $\le k+1$. 
Indeed, by \eqref{eq:est_fS}, $df_{t,\lambda,S}$
vanishes along strata of dimension $\le k$, so \eqref{eq:stratumcond2}
continues to hold for those; whereas, over the interior of the
stratum $S$, $df_{t,\lambda,S}=df^0_{t,\lambda,S}$, and the contributions from other
$(k+1)$-dimensional strata vanish. 

Moreover, $f_{t,\lambda,\le k+1}$ satisfies the estimate
\eqref{eq:stratumcond_der} (with $k+1$ instead of $k$), as needed for the
induction to proceed. Indeed, this follows immediately from the estimates
\eqref{eq:stratumcond_der} for $f_{t,\lambda,\le k}$ (note that the second
estimate also holds near
$k$-dimensional strata, since the distance to the nearest $k$-dimensional
stratum is no greater than that to the nearest lower-dimensional stratum),
and \eqref{eq:est_fS} for $f_{t,\lambda,S}$.

Thus, we can indeed carry out the construction of $f_{t,\lambda,\le k}$ with the desired
properties by induction on $k$. Finally, we let $f_{t,\lambda}=f_{t,\lambda,\le n-1}$.

As a consequence of the estimates \eqref{eq:est_fS} on individual terms,
$f_{t,\lambda}$ is $C^1$ with locally Lipschitz first derivatives, and
smooth on $V^0$, except along $H$ for $(t,\lambda)=(0,\epsilon)$.
By construction, it is supported in the intersection of $U$ with
a neighborhood of $D_V$, and
satisfies \eqref{eq:stratumcond2} for all toric strata. 

By \eqref{eq:stratumcond_der}, $|df_{t,\lambda}|=O(t/\Phi)$, while 
$|{df_{t,\lambda}}_{|\Theta}|=O(t/\sqrt\Phi)$.

Because the K\"ahler form $\omega_{t,\lambda}$
blows up like $\epsilon/\sqrt\Phi$ in the directions transverse to $\Theta$,
we conclude that the Hamiltonian vector field of $f_{t,\lambda}$ with 
respect to $\omega_{t,\lambda}$ is bounded by $O(t/\sqrt\Phi)$ (again
with respect to the fixed reference metric), hence locally uniformly
bounded. (Recall that $\sqrt\Phi\ge t$ wherever $\tilde\chi\equiv 1$,
while the other terms are bounded below wherever $\tilde\chi<1$.)
Moreover, the regularity of $f_{t,\lambda}$ implies that this vector
field is locally Lipschitz continuous, and smooth on $V^0$, except along $H$ for
$(t,\lambda)=(0,\epsilon)$.

Combining this with the outcome of Step 3, we find that
the vector field $\tilde{v}_{t,\lambda}$ defined by
$\iota_{\tilde{v}_{t,\lambda}}\omega'_{t,\lambda}=-a_{t,\lambda}-df_{t,\lambda}$
is smooth on $V^0$ (and locally Lipschitz continuous along $D_V$),
except along $H$ for $(t,\lambda)=(0,\epsilon)$, and
its norm (again with respect to a smooth reference metric) is bounded
by $O(t/\sqrt{\Phi})$, hence locally uniformly bounded.
Thus, even though $\tilde{v}_{t,\lambda}$ is not continuous along $H$
for $(t,\lambda)=(0,\epsilon)$,
its flow is well-defined and continuous even for $\lambda=\epsilon$.
We then obtain $\phi_{sm,\lambda}$ with all the desired properties by
considering the time $\kappa$ flow generated by $\tilde{v}_{t,\lambda}$.
\endproof

\end{document}